\documentclass[oneside,a4paper,english, 11pt]{amsart}

\usepackage{adjustbox}
\usepackage[latin9]{inputenc}
\usepackage{hyperref}
\usepackage{lscape}
\usepackage{mathrsfs}  
\usepackage{enumerate}  
\usepackage{subcaption}
\usepackage{bm}

\makeatletter
 \theoremstyle{plain}
\newtheorem{thm}{Theorem}[subsection]
\theoremstyle{plain}
  \newtheorem{prop}[thm]{Proposition}
\theoremstyle{plain}
 \newtheorem{lemma}[thm]{Lemma}
\theoremstyle{plain}
 
\theoremstyle{plain}

\theoremstyle{plain}
\newtheorem{cor}[thm]{Corollary}
\theoremstyle{definition}
  \newtheorem{defn}[thm]{Definition}
\theoremstyle{definition}
  
 \theoremstyle{definition}

\theoremstyle{remark}
\newtheorem{rmk}[thm]{Remark}
\numberwithin{equation}{section}

\usepackage{amsmath}
\usepackage{amssymb}
\usepackage{amscd}
\usepackage{amsthm}
\usepackage{array,multirow}
\usepackage[arrow,matrix,tips,all,cmtip,poly,color]{xy}

\usepackage{stmaryrd}
\usepackage{url}
\usepackage{xcolor}
\usepackage{caption}
\usepackage{enumitem}
\usepackage{tikz-cd}
\usetikzlibrary{arrows}
\usetikzlibrary{decorations,decorations.text,decorations.markings,decorations.shapes}
\tikzset{ closed/.style = {decoration = {markings, mark = at position 0.5 with { \node[transform shape, xscale = .8, yscale=.4] {/}; } }, postaction = {decorate} },
open/.style = {decoration = {markings, mark = at position .5 with { \node[transform shape, scale =1.2] {$\circ$}; } }, postaction = {decorate} }
}
\usepackage{float}

\pdfpageheight\paperheight
\pdfpagewidth\paperwidth
\newcommand{\Z}{\mathbb{Z}}

\newcommand{\Q}{\mathbb{Q}}

\newcommand{\Qp}{\mathbb{Q}_p}

\newcommand{\R}{\mathbb{R}}

\newcommand{\bG}{\mathbb{G}}

\newcommand{\F}{\mathbb{F}}

\newcommand{\N}{\mathbb{N}}

\newcommand{\A}{\mathbb{A}}

\newcommand{\fM}{\mathfrak{M}}

\newcommand{\fP}{\mathfrak{P}}
\newcommand{\fp}{\mathfrak{p}}

\newcommand{\fS}{\mathfrak{S}}

\newcommand{\fm}{\mathfrak{m}}

\newcommand{\bF}{\mathbb{F}}

\newcommand{\bN}{\mathbb{N}}

\newcommand{\bT}{\mathbb{T}}

\newcommand{\cA}{\mathcal{A}}

\newcommand{\cC}{\mathcal{C}}

\newcommand{\cE}{\mathcal{E}}

\newcommand{\cG}{\mathcal{G}}
\newcommand{\cH}{\mathcal{H}}
\newcommand{\cI}{\mathcal{I}}

\newcommand{\cJ}{\mathcal{J}}

\newcommand{\cO}{\mathcal{O}}
\newcommand{\cP}{\mathcal{P}}

\newcommand{\cS}{\mathcal{S}}

\newcommand{\cX}{\mathcal{X}}

\newcommand{\cZ}{\mathcal{Z}}

\newcommand{\eps}{\varepsilon}

\newcommand{\phz}{\varphi}

\newcommand{\La}{\Lambda}

\newcommand{\Zp}{\mathbb{Z}_p}

\newcommand{\Gal}{\mathrm{Gal}}
\newcommand{\Hom}{\mathrm{Hom}}

\newcommand{\Res}{\mathrm{Res}}

\newcommand{\ind}{\mathrm{ind}}

\newcommand{\GL}{\mathrm{GL}}

\newcommand{\ad}{\mathrm{ad}}

\newcommand{\Spec}{\mathrm{Spec}\ }

\newcommand{\id}{\mathrm{id}}
\newcommand{\isom}{\cong}
\newcommand{\Adm}{\mathrm{Adm}}

\newcommand{\speci}{\mathrm{sp}}
\newcommand{\semis}{\mathrm{ss}}

\newcommand{\Fp}{\F_p}

\newcommand{\un}[1]{\underline{#1}}
\renewcommand{\bf}[1]{\mathbf{#1}}
\newcommand{\Rep}{\mathrm{Rep}}

\newcommand{\tld}[1]{\widetilde{#1}}

\newcommand{\JH}{\mathrm{JH}}

\newcommand{\supp}{\mathrm{Supp}}

\newcommand{\rbar}{\overline{r}}
\newcommand{\rhobar}{\overline{\rho}}
\newcommand{\taubar}{\overline{\tau}}

\newcommand{\Spf}{\mathrm{Spf}}

\newcommand{\der}{\mathrm{der}}
\newcommand{\orient}{\mathrm{or}}

\newcommand{\Trns}{\mathfrak{Tr}}
\newcommand{\defeq}{\stackrel{\textrm{\tiny{def}}}{=}}

\newcommand{\ovl}[1]{\overline{#1}}

\newcommand{\obv}{\mathrm{obv}}

\newif\iffinalrun
\iffinalrun
  \newcommand{\mar}[1]{}
\else
  \newcommand{\mar}[1]{\marginpar{\raggedright\tiny #1}}

\DeclareMathOperator{\Mod}{Mod}
\DeclareMathOperator{\Coh}{Mod}

\DeclareMathOperator{\Ad}{Ad}

\DeclareMathOperator{\Mat}{Mat}
\DeclareMathOperator{\Ann}{Ann}

\DeclareMathOperator{\Gr}{Gr}
\DeclareMathOperator{\Fl}{Fl}

\DeclareMathOperator{\Iw}{\cI}

\DeclareMathOperator{\leqeta}{\text{$\setlength{\thickmuskip}{0mu}\leq\eta$}}

\newcommand{\ra}{\rightarrow}

\newcommand{\iarrow}{\hookrightarrow}
\newcommand{\into}{\hookrightarrow}
\newcommand{\surj}{\twoheadrightarrow}
\newcommand{\onto}{\twoheadrightarrow}

\title{Serre weights for three-dimensional wildly ramified Galois representations}

\author{Daniel Le}
\address{Department of Mathematics,
Purdue University,
150 N. University Street, 
West Lafayette, IN 47907-2067}
\email{ledt@purdue.edu}

\author{Bao V.~Le Hung}
\address{Department of Mathematics,
Northwestern University, 
2033 Sheridan Road\\
Evanston, IL 60208, USA}
\email{lhvietbao@googlemail.com}

\author{Brandon Levin}
\address{Department of Mathematics,
Rice University, 
P.O. Box 1892,
Houston, TX 77005-1892}
\email{bwlevin@rice.edu}

\author{Stefano Morra}
\address{LAGA, UMR 7539, CNRS, Universit\'e Paris 13 - Sorbonne Paris Cit\'e, 
Universit\'e de Paris 8,
99 avenue Jean Baptiste Cl\'ement,
93430 Villetaneuse,
France }
\email{morra@math.univ-paris13.fr}

\begin{document}

\begin{abstract}
We formulate and prove the weight part of Serre's conjecture for three-dimensional mod $p$ Galois representations under a genericity condition when the field is unramified at $p$.  This removes the assumption in \cite{LLLM, LLLM2} that the representation be tamely ramified at $p$.    We also prove a version of Breuil's lattice conjecture and a mod $p$ multiplicity one result for the cohomology of $U(3)$-arithmetic manifolds.  The key input is a study of the geometry of the Emerton--Gee stacks \cite{EGstack} using the local models introduced in \cite{MLM}.  
\end{abstract}

\maketitle
\tableofcontents

\section{Introduction}

The goal of this paper is to prove a generalization of the weight part of Serre's conjecture for three-dimensional mod $p$ Galois representations which are \emph{generic} at $p$. We also prove a generalization of Breuil's lattice conjecture for these representations and the Breuil--M\'ezard conjecture for generic tamely potentially crystalline deformation rings of parallel weight $(2,1,0)$. 
For a detailed discussion of these conjectures, see \cite{LLLM2} where we establish the tame case of these conjectures. 

\subsection{Results}\label{sec:global:prelim}
\subsubsection{The weight part of Serre's conjecture}
 
Let $p$ be a prime, and let $F/F^+$ be a CM extension of a totally real field $F^+\neq \Q$.  Assume that all places in $F^+$ above $p$ split in $F/F^+$. 
Let $G$ be a definite unitary group over $F^+$ split over $F$ which is isomorphic to $U(3)$ at each infinite place and split at each place above $p$. 
A \emph{(global) Serre weight} is an irreducible $\overline{\F}_p$-representation $V$ of $G(\cO_{F^+, p})$.
These are all of the form $\otimes_{v \mid p} V_v$ with $V_v$ an irreducible $\overline{\F}_p$-representation of $G(k_v)$, where $k_v$ is the residue field of $F^+$ at $v$. 
For a mod $p$ Galois representation $\rbar:G_F \rightarrow \GL_3(\overline{\F}_p)$, let $W(\rbar)$ denote the collection of modular Serre weights for $\rbar$. 
That is, $V \in W(\rbar)$ if the Hecke eigensystem attached to $\rbar$ appears in a space of mod $p$ automorphic forms on $G$ of weight $V$ for some prime to $p$ level. 

For each place $v \mid p$, fix a place $\tld{v}$ of $F$ dividing $v$ which identifies $G(k_v)$ with $\GL_3(k_v)$. 
Define $\rhobar_v := \rbar|_{\Gal(\overline{F}_{\tld{v}}/F_{\tld{v}})}$. 
We can now state the main theorem:

\begin{thm}[Theorem \ref{thm:SWC}]
\label{thm:5:prelim}  Assume that $p$ is unramified in $F$ and that $\rhobar_v$ is $8$-generic for all $v \mid p$.   Assume that $\rbar$ is modular (i.e.~$W(\rbar)$ is nonempty) and satisfies Taylor-Wiles hypotheses. 

Then 
\[
\otimes_{v \mid p} V_v \in W(\rbar) \Longleftrightarrow V_v \in  W^g(\rhobar_v) \text{ for all } v \mid p
\]
where $W^g(\rhobar_v)$ is an explicit set of irreducible $\ovl{\F}_p$-representations of $\GL_3(k_v)$ attached to $\rhobar_v$ $($see Definition \ref{def:intro:geoweight} below$)$. 
\end{thm}

\noindent
In particular, this affirms the expectation from local-global compatibility that $W(\rbar)$ depends only on the restrictions of $\rbar$ to places above $p$. 

\begin{rmk}
This is the first complete description of $W(\rbar)$ in dimension greater than two for representations $\rbar$ that are wildly ramified above $p$. 
Some lower bounds were previously obtained in \cite{gee-geraghty,MP,HLM,LMP,LLLM}.
\end{rmk}

The first obstacle we overcome is the lack of a conjecture. 
One basic problem is that while tame representations (when restricted to inertia) depend only on discrete data, wildly ramified representations vary in nontrivial moduli.
Buzzard--Diamond--Jarvis defined a recipe in terms of crystalline lifts in dimension two. 
However, after \cite{LLLM}, it was clear that the crystalline lifts perspective is  insufficient in higher dimension.  %
In higher dimensions, Herzig defined a combinatorial/representation theoretic recipe for a collection of weights $W^?(\rhobar_v)$ when $\rhobar_v$ is tame. 
For possibly wildly ramified $\rhobar$, \cite{GHS} makes a conjectural conjecture: they define a conjectural set \emph{conditional} on a version of the Breuil--M\'ezard conjecture. 
Our first step is to prove a version of the Breuil--M\'ezard conjecture (Theorem \ref{thm:1:prelim} and Remark \ref{rmk:BM:prelim}) when $n=3$. 

Having established a version of the Breuil--M\'ezard conjecture when $n=3$, the weight set from \cite{GHS} turns out (in generic cases) to have a simple geometric interpretation. 
Let $\cX_3$ be the moduli stack of $(\varphi,\Gamma)$-modules recently constructed by Emerton--Gee \cite{EGstack}. %
The irreducible components of $\cX_3$ are labelled by irreducible mod $p$ representations of $\GL_3(k_v)$ and $W^g(\rhobar_v)$ is defined so that $V_v \in W^g(\rhobar_v)$ if and only if $\rhobar_v$ lies on $\cC_{V_v}$. 
However, this definition of $W^g(\rhobar_v)$ gives very little insight into its structure. 
We study $W^g(\rhobar_v)$ using the local models developed in \cite{MLM} combined with the explicit calculations of tamely potentially crystalline deformation rings in \cite{LLLM, LLLM2}.   We ultimately arrive at an explicit description of all possible weight sets $W^g(\rhobar_v)$ which allows us to then employ the Taylor--Wiles patching method to prove Theorem \ref{thm:5:prelim}. 

\subsubsection{Breuil's lattice conjecture and mod $p$ multiplicity one}

The weight part of Serre's conjecture can be viewed as a local-global compatibility result in the mod $p$ Langlands program.
In this section, we mention two further local-global compatibility results---one mod $p$ and one $p$-adic.
We direct the reader to the introduction of \cite{LLLM2} for further context for the following two results. 

In the global setup above assume further that $F/F^+$ is unramified at all finite places and $G$ is quasi-split at all finite places.
Let $r:G_F \ra \GL_3(\overline{\Q}_p)$ be a modular Galois representation which is tamely potentially crystalline with Hodge--Tate weights $(2,1,0)$ at each place above $p$ and unramified outside $p$ (though our results hold true when $r$ is minimally split ramified, cf.~\S \ref{subsec:global}).
Write $\lambda$ for the Hecke eigensystem corresponding to $r$. %
We fix places $\tld{v}|v|p$ of $F$ and $F^+$ respectively.
We let $\tld{H}$ be the integral $p$-adically completed cohomology with infinite level at $v$, hyperspecial level outside $v$, and constant coefficients.
Set $\rho\defeq r|_{G_{F_{\tld{v}}}}$ and let $\sigma(\tau)$ be the tame type corresponding to the Weil--Deligne representation associated to $\rho$ under the inertial local Langlands correspondence (so that $\tld{H}[\lambda][1/p]$ contains $\sigma(\tau)$ with multiplicity one). 
Let $\rbar$ and $\rhobar$ be the reductions of $r$ and $\rho$, respectively. 

\begin{thm}[Theorem \ref{thm:lattice}]
\label{mainthm1} 
Assume that $p$ is unramified in $F^+$, $r$ is unramified outside $p$, $\rhobar$ is $11$-generic, and $\rbar$ satisfies Taylor--Wiles hypotheses.
Then, the lattice 
\[
\sigma(\tau) \cap \tld{H}[\lambda]
\]
depends only on $\rho$.
\end{thm}

We now let $\overline{H}$ be the mod $p$ reduction of $\tld{H}$.
Thus, $\overline{H}$ is the mod $p$ cohomology with infinite level at $v$ (and hyperspecial level at places outside $v$ with constant coefficients) of a $U(3)$-arithmetic manifold.

\begin{thm}[Theorem \ref{thm:modpmultone}]
\label{thm:6:prelim}

Let $\sigma(\tau)^\circ$ be an $\cO$-lattice in $\sigma(\tau)$ with irreducible ``upper alcove'' cosocle.
Under the assumptions of Theorem \ref{mainthm1}, $\Hom_{\GL_3(\cO_{F_{\tld{v}}})}(\sigma(\tau)^\circ,\ovl{H}[\lambda])$ is a one-dimensional $\ovl{\F}_p$-vector space.
\end{thm}
\noindent (See \S \ref{subsec:notations}, \ref{subsec:param:SW}, for the notion of upper and lower alcove for Serre weights for $\GL_3$.)
The statement of Theorem \ref{thm:6:prelim} is also true when the cosocle is not necessarily upper alcove if one imposes a condition on the shape of $\rhobar$ with respect to $\tau$; see Theorem \ref{thm:modpmultone}.

\subsection{Methods}

\subsubsection{Local methods: Geometry of the Emerton--Gee stack and local models}

We begin by recalling the set $W^g(\rhobar)$ that appears in Theorem \ref{thm:5:prelim}.  
Let $K$ be a finite unramified extension of $\Qp$ of degree $f$, with ring of integers $\cO_K$ and residue field $k$. 
Let $\cX_{K,n}$ be the Noetherian formal algebraic stack over $\Spf\,\Zp$ defined in \cite[Definition 3.2.1]{EGstack}.
It has the property that for any complete local Noetherian $\Zp$-algebra $R$, the groupoid $\cX_{K,n}(R)$ is equivalent to the groupoid of rank $n$ projective $R$-modules equipped with a continuous $G_K$-action, see \cite[\S 3.6.1]{EGstack}.
In particular, $\cX_{K,n}(\ovl{\F}_p)$ is the groupoid of continuous Galois representations $\rhobar:G_K\rightarrow \GL_n(\ovl{\F}_p)$. 
As explained in \cite[\S 7.4]{MLM}, there is a bijection $\sigma\mapsto \cC_\sigma$  between irreducible $\overline{\F}_p$-representations of $\GL_n(k)$ and the irreducible components of the reduced special fiber of $\cX_{K,n}$.
(This is a relabeling of the bijection of \cite[Theorem 6.5.1]{EGstack}.)

\begin{defn} \label{def:intro:geoweight}  Let $\rhobar \in \cX_{K, n} (\overline{\F}_p)$.   Define the set of \emph{geometric weights} of $\rhobar$ to be 
\[
W^g(\rhobar) = \{ \sigma \mid \rhobar \in \cC_{\sigma}(\overline{\F}_p) \}.  
\]
\end{defn}

\noindent While Definition \ref{def:intro:geoweight} is simple, it does not appear to be an easy task to determine the possible sets $W^g(\rhobar)$. 
The irreducible components of $\cX_{K, n}$ are described in terms of closures of substacks, but we expect the closure relations and component intersections in $\cX_{K, n}$ to be rather complicated.%
We now specialize to the case $n=3$. 
A key tool in the analysis of the sets $W^g(\rhobar)$ in this setting is the description of certain potentially crystalline substacks. 
For a tame inertial type $\tau$, let $\cX^{\eta,\tau} \subset \cX_{K, 3}$ be the substack parametrizing potentially crystalline representations of type $\tau$ and parallel weight $(2,1, 0)$.   
Recall that $\sigma(\tau)$ denotes the representation of $\GL_3(\cO_K)$ obtained by applying the inertial local Langlands correspondence to $\tau$ (it is the inflation of a Deligne--Lusztig representation; see \S \ref{subsub:ILLC}).
The following is an application of the theory of local models of \cite{MLM}:

\begin{thm}[Corollary \ref{thm:local_model_main}]
\label{thm:1:prelim}
If $\tau$ is a $4$-generic tame inertial type, then $\cX^{\eta,\tau}$ is normal and Cohen--Macaulay and its special fiber $\cX^{\eta,\tau}_{\F}$ is reduced. Moreover, $\cX^{\eta,\tau}_{\F}$ is the scheme-theoretic union
\[
\bigcup_{\sigma \in \JH(\ovl{\sigma}(\tau))}\cC_\sigma.
\]
\end{thm}
\begin{rmk}\label{rmk:BM:prelim}
This shows that the choice of cycles $\cZ_\sigma=\cC_\sigma$ solves the Breuil--M\'ezard equations for the above $\cX^{\eta,\tau}$  (cf.~\cite[Conjecture 8.2.2]{EGstack} and \cite[Conjecture 8.1.1]{MLM}).
\end{rmk}

The equality of the underlying reduced $\cX^{\eta, \tau}_{\F, \mathrm{red}}$ and the scheme-theoretic union $\bigcup_{\sigma \in \JH(\ovl{\sigma}(\tau))}\cC_\sigma$ is proved in Theorem 1.3.5 in \cite{MLM} though we reprove it here with a weaker genericity condition (see Remark \ref{rmk:genericity:prelim}).    The key point is to prove that the special fiber of  $\cX^{\eta,\tau}$ is in fact reduced.  (If we replace $\eta$ by $\lambda+\eta$ with $\lambda$ dominant and nonzero or $n=3$ by $n>3$, the Breuil--M\'ezard conjecture predicts that the analogous stacks never have reduced special fiber.)  The special fiber of $\cX^{\eta, \tau}$ has an open cover with open sets labeled by $f$-tuples of $(2,1,0)$-admissible elements $(\tld{w}_j)$ in the extended affine Weyl group of $\GL_3$.   The complexity of the geometry of the open sets increases as the lengths of the $\tld{w}_j$ decrease.
When the length of $\tld{w}_j$ is greater than 1 for all $j$, the reducedness immediately follows from the calculations in \cite[\S 5.3]{LLLM}.   
Otherwise, the calculations of \cite[\S 8]{LLLM}  give an explicit upper bound on the special fiber which when combined with  $\cX^{\eta, \tau}_{\F,\mathrm{red}} = \bigcup_{\sigma \in \JH(\ovl{\sigma}(\tau))}\cC_\sigma$ must be an equality, and the reducedness follows.  

\begin{rmk}\label{rmk:genericity:prelim} 
An inexplicit genericity condition appears in the main theorems of \cite{MLM} (see \S 1.2.1 of \emph{loc. cit.}).  While we use the models constructed in \emph{loc. cit.}, we reprove some of its main theorems in \S \ref{sub:Mon:Cond}, \ref{subsub:SpFi} with the inexplicit condition replaced by the more typical genericity condition on the gaps between the exponents of the inertial characters in $\tau$. 
This is possible because of the computations in 
\cite{LLLM, LLLM2}.
\end{rmk} 

Finally, we analyze $W^g(\rhobar)$ using local models. 
The special fibers of the local models embed inside the affine flag variety where irreducible components appear as subvarieties of translated affine Schubert varieties. 
In \S \ref{sec:GSW}, we introduce a subset $W_{\mathrm{obv}}(\rhobar)\subset W^g(\rhobar)$ of \emph{obvious weights} for (possibly) wildly ramified $\rhobar$, which has a simple interpretation in terms of the affine flag variety. 
Obvious weights generalize the notion of ordinary weights that appear in \cite{gee-geraghty} and the additional weights appearing in the exceptional cases of \cite{MP,HLM,LMP}. 
The set $W_{\mathrm{obv}}(\rhobar)$ gives upper and lower bounds for $W^g(\rhobar)$. 
We finally show that, in almost all cases, one can determine $W^g(\rhobar)$ from $W_{\mathrm{obv}}(\rhobar)$ (Theorem \ref{thm:Wg}). 
This last part uses a curious piece of numerology from the calculations of \cite{LLLM}---points in the special fibers of the local models never lie on exactly three components.

\subsubsection{Global methods: Patching} 

To prove Theorems \ref{mainthm1} and \ref{thm:6:prelim} we combine the explicit description of the weight sets $W(\rbar)$, coming from Theorems \ref{thm:5:prelim} and \ref{thm:Wg}, with the Kisin--Taylor--Wiles methods developed in \cite{EGS} and employed in \cite[\S 5]{LLLM2}. 
A crucial ingredient is the analysis of certain intersections of cycles in the special fiber of deformation rings.
The local models introduced in \cite{MLM} allow us to algebraize the computations made for the tame case in \cite[\S 3.6]{LLLM2}. 

We now turn to Theorem \ref{thm:5:prelim}. 
The key input into its proof beyond the Kisin--Taylor--Wiles method is the fact that the local Galois deformation rings of type $(\eta,\tau)$ are domains when $\tau$ is $4$-generic. 
This is guaranteed by the fact that the stacks $\cX^{\eta,\tau}$ are normal (Theorem \ref{thm:1:prelim}). 
Then the supports of the patched modules of type $\tau$ are either empty or the entire potentially crystalline deformation rings of type $(\eta,\tau)$. 
The proof is then similar to the tame case in \cite{LLLM2}---one propagates modularity between obvious weights and then to shadow weights using carefully chosen types---except for one new wrinkle. 
From the axioms of a weak patching functor, one cannot deduce the modularity of an obvious weight to get started! 
Indeed one cannot rule out that $\rhobar_v$ lies on a unique component $\cC_{\sigma_v}$ and $W(\rbar)$ contains exactly one Serre weight $\sigma'\neq \sigma\defeq \otimes_{v|p}\sigma_v$ with the property that for any tame inertial type $\tau$, if $\JH(\ovl{\sigma}(\tau))$ contains $\sigma'$, then it also contains $\sigma$. 
We use a patched version of the weight cycling technique introduced in \cite{EGH} to rule out this pathology. 
In fact, we axiomatize our setup to make clear the ingredients that our method requires. 

\subsection{Overview}

\S \ref{sec:background} covers background on tame inertial $L$-parameters, representation theory (\S \ref{subsub:TIT:DL}), and Breuil--Kisin modules with tame descent data (\S \ref{subsub:BKM}), following \cite{MLM}. 
\S \ref{subsec:param:SW} gives a comparison between parametrizations of Serre weights in \cite{MLM} and \cite{LLLM2}.  

\S \ref{sec:LM:EG} establishes the main results about the geometry of local deformation rings. We specialize the theory of local models in \cite{MLM} to dimension three.
The main results are Theorem \ref{thm:local_model_cmpt} and Corollary \ref{thm:local_model_main} which establish the geometric properties that we need, some of which are specific to dimension three.

In \S \ref{sec:GSW}, we analyze possible sets of geometric weights using the affine flag variety.  
Theorem \ref{thm:Wg} gives a complete explicit description when $\rhobar$ is sufficiently generic. 

\S \ref{sec:PF:global_applications} contains our global applications.  
In \S \ref{subsec:WMPF}, we introduce the axioms of patching functors following \cite[\S 6]{MLM} and prove the weight part of Serre's conjecture \emph{assuming} the modularity of at least one obvious weight (Proposition \ref{prop:minimalcycle}).
The latter condition is then removed in \S \ref{subsec:AM} using modules with an arithmetic action (Theorem \ref{thm:axiomaticSWC}).
In \S \ref{sec:cyc}, we prove results on mod $p$ multiplicity one and Breuil's lattice conjectures for patched modules (Theorems \ref{thm:cyclic}, \ref{thm:gauge}), generalizing analogous results in \cite{LLLM2} to the wildly ramified case.
Finally, \S \ref{subsec:global} proves our main global theorems.

\subsection{Acknowledgements}
The main ideas of this article date back to the summer of 2016, but the formulation of the results were rather awkward due to the lack of the Emerton--Gee stack and its local model theory at that time.
For this reason we decided to write up this paper only after the release of \cite{MLM}.
We apologize for the long delay.
Part of the work has been carried out during visits at the Institut Henri Poinar\'e (2016), the Institute for Advanced Study (2017), Mathematisches Forschungsinstitut Oberwolfach (2019), University of Arizona, and Northwestern University.
We would like to heartily thank these institutions for the outstanding research conditions they provided, and for their support.

Finally, D.L. was supported by the National Science Foundation under agreements Nos.~DMS-1128155 and DMS-1703182 and an AMS-Simons travel grant. B.LH. acknowledges support from the National Science Foundation under grant Nos.~DMS-1128155, DMS-1802037 and the Alfred P. Sloan Foundation. B.L. was partially supported by Simons Foundation/SFARI (No.~585753), National Science Foundation grant DMS-1952556, and the Alfred P. Sloan Foundation. S.M. was supported by the Institut Universitaire de France and the ANR-18-CE40-0026 (CLap CLap).

\subsection{Notation}
\label{subsec:notations}

For any field $K$ we fix once and for all a separable closure $\ovl{K}$ and let $G_K \defeq \Gal(\ovl{K}/K)$. 
If $K$ is a nonarchimedean local field, we let $I_K \subset G_K$ denote the inertial subgroup.
We fix a prime $p\in\Z_{>0}$.
Let $E \subset \ovl{\Q}_p$ be a subfield which is finite-dimensional over $\Q_p$.
We write $\cO$ to denote its ring of integers, fix an uniformizer $\varpi\in \cO$ and let $\F$ denote the residue field of $E$.
We will assume throughout that $E$ is sufficiently large.

We consider the group $G\defeq\GL_3$ (defined over $\Z$).
We write $B$ for the subgroup of upper triangular matrices, $T \subset B$ for the split torus of diagonal matrices and $Z \subset T$ for the center of $G$.  
Let $\Phi^{+} \subset \Phi$ (resp. $\Phi^{\vee, +} \subset \Phi^{\vee}$) denote the subset of positive roots (resp.~positive coroots) in the set of roots (resp.~coroots) for $(G, B, T)$. 
Let $\Delta$ (resp.~$\Delta^{\vee}$) be the set of simple roots (resp.~coroots).
Let $X^*(T)$ be the group of characters of $T$ which we identify with $\Z^3$ by letting the standard $i$-th basis element $\eps_i=(0,\ldots, 1,\ldots, 0)$ (with the $1$ in the $i$-th position)  correspond to extracting the $i$-th diagonal entry of a diagonal matrix. 
In particular, we let $\eps_1'$ and $\eps'_2$ be $(1,0,0)$ and $(0,0,-1)$ respectively.

We write $W$ (resp.~$W_a$, resp.~$\tld{W}$) for the Weyl group (resp.~the affine Weyl group, resp. the extended affine Weyl group) of $G$.
If $\Lambda_R \subset X^*(T)$ denotes the root lattice for $G$ we then have
\[
W_a = \Lambda_R \rtimes W, \quad \tld{W} = X^*(T) \rtimes W
\]
and use the notation $t_{\nu} \in \tld{W}$ to denote the image of $\nu \in X^*(T)$. 
The Weyl groups $W$, $\tld{W}$, and $W_a$ act naturally on $X^*(T)$ and on $X^*(T)\otimes_{\Z} A$ for any ring $A$ by extension of scalars.

Let $\langle \ ,\,\rangle$ denote the duality pairing on $X^*(T)\times X_*(T)$, which extends to a pairing on $(X^*(T)\otimes_{\Z}A)\times (X_*(T)\otimes_{\Z}A)$ for any ring $A$.
We say that a weight $\lambda\in X^*(T)$ is \emph{dominant} if $0\leq \langle\lambda,\alpha^\vee\rangle$  for all $\alpha\in \Delta$.
Set $X^0(T)$ to be the subgroup consisting of characters $\lambda\in X^*(T)$ such that $\langle\lambda,\alpha^\vee\rangle=0$ for all $\alpha\in \Delta$, and $X_1(T)$ to be the subset consisting of characters $\lambda\in X^*(T)$ such that $0\leq \langle\lambda,\alpha^\vee\rangle< p$ for all $\alpha\in \Delta$

We fix an element $\eta\in X^*(T)$ such that $\langle \eta,\alpha^\vee\rangle = 1$ for all positive simple roots $\alpha$.
We define the \emph{$p$-dot action} as $t_\lambda w \cdot \mu = t_{p\lambda} w (\mu+\eta) - \eta$.
By letting $w_0$ denote the longest element in $W$ define $\tld{w}_h\defeq w_0 t_{-\eta}$.

Recall that for $(\alpha,n)\in \Phi^+\times \Z$, we have the $p$-root hyperplane $H_{\alpha,n}\defeq \{\lambda:\ \langle\lambda+\eta,\alpha^\vee\rangle=np\}$. A $p$-alcove is a connected component of  
the complement $X^*(T)\otimes_{\Z}\R\ \setminus\ \big(\bigcup_{(\alpha,n)}H_{\alpha,n}\big)$.
We say that a $p$-alcove $C$ is $p$-restricted (resp.~dominant) if $0<\langle\lambda+\eta,\alpha^\vee\rangle<p$ (resp.~$0<\langle\lambda+\eta,\alpha^\vee\rangle$) for all simple roots $\alpha\in \Delta$ and $\lambda\in C$.
If $C_0 \subset X^*(\un{T})\otimes_{\Z}\R$ denotes the dominant base alcove (i.e.~the alcove defined by the condition $0<\langle\lambda+\eta,\alpha^\vee\rangle<p$ for all positive roots $\alpha\in \Phi^+$, we let 
\[\tld{W}^+\defeq\{\tld{w}\in \tld{W}:\tld{w}\cdot C_0 \textrm{ is dominant}\}.\]
and
\[\tld{W}^+_1\defeq\{\tld{w}\in \tld{{W}}^+:\tld{w}\cdot {C}_0 \textrm{ is } p\textrm{-restricted}\}.\]
We sometimes refer to $C_0$ as the \emph{lower alcove} and $C_1\defeq \tld{w}_h\cdot C_0$ as the \emph{upper alcove}.

Let now $\cO_p$ be a finite \'etale $\Z_p$-algebra.
We have an isomorphism $\cO_p\cong\prod\limits_{v\in S_p} \cO_{v}$ where $S_p$ is a finite set and $\cO_{v}$ is the ring of integers of a finite unramified extension $F^+_{v}$ of $\Q_p$.
Let $G_0 \defeq \Res_{\cO_p/\Z_p} G_{/\cO_p}$ with Borel subgroup $B_0 \defeq  \Res_{\cO_p/\Z_p} B_{/\cO_p}$, maximal torus $T_0 \defeq \Res_{\cO_p/\Z_p} T_{/\cO_p}$, and $Z_0 \defeq \Res_{\cO_p/\Z_p} Z_{/\cO_p}$. 
We assume that $\cO$ contains the image of any ring homomorphism $\cO_p \ra \ovl{\Z}_p$ and write $\cJ\defeq \Hom_{\Zp}(\cO_p,\cO)$.
We can and do fix an identification of $\un{G} \defeq (G_0)_{/\cO}$ with the split reductive group $G_{/\cO}^{\cJ}$. 
We similarly define $\un{B}, \un{T},$ and $\un{Z}$.
Corresponding to $(\un{G}, \un{B}, \un{T})$, we have the set of positive roots $\un{\Phi}^+ \subset \un{\Phi}$ and the set of positive coroots $\un{\Phi}^{\vee, +}\subset \un{\Phi}^{\vee}$.
The notations $\un{\Lambda}_R$, $\un{W}$, $\un{W}_a$, $\tld{\un{W}}$, $\tld{\un{W}}^+$, $\tld{\un{W}}^+_1$ should be clear as should the natural isomorphisms $X^*(\un{T}) = X^*(T)^{\cJ}$ and the like.  
Given an element $j\in\cJ$, we use a subscript notation to denote $j$-components obtained from the isomorphism $\un{G}\cong G_{/\cO}^{\cJ}$ (so that, for instance, given an element $\tld{w}\in \tld{\un{W}}$ we write $\tld{w}_j$ to denote its $j$-th component via the induced identification $\tld{\un{W}}\cong \tld{W}^{\cJ}$).
For sake of readability, we abuse notation and still write $w_0$ to denote the longest element in $\un{W}$, and fix a choice of an element $\eta\in X^*(\un{T})$ such that $\langle \eta,\alpha^\vee\rangle = 1$ for all $\alpha\in\un{\Delta}$.
The meaning of $w_0$, $\eta$ and $\tld{w}_h\defeq w_0t_{-\eta}$ should be clear from the context.

The absolute Frobenius automorphism on $\cO_p/p$ lifts canonically to an automorphism $\varphi$ of $\cO_p$. We define an automorphism $\pi$ of the identified groups $X^*(\un{T})$ and $X_*(\un{T}^\vee)$ by the formula $\pi(\lambda)_\sigma = \lambda_{\sigma \circ \varphi^{-1}}$ for all $\lambda\in X^*(\un{T})$ and $\sigma: \cO_p \ra \cO$.
We assume that, in this case, the element $\eta\in X^*(\un{T})$ we fixed is $\pi$-invariant.
We similarly define an automorphism $\pi$ of $\un{W}$ and $\tld{\un{W}}$.

Let $F^+_p$ be $\cO_p[1/p]$ so that $F^+_p$ is isomorphic to the (finite) product $\prod\limits_{v \in S_p} F^+_{v}$ where $F^+_{v} \defeq \cO_{v}[1/p]$ for each $v \in S_p$.
Let 
\[
\un{G}^\vee_{/\Z} \defeq \prod_{F^+_p \ra E} G^\vee_{/\Z} 
\]
be the dual group of $\un{G}$ so that the Langlands dual group of $G_0$ is $^L \un{G}_{/\Z} \defeq \un{G}^\vee\rtimes \Gal(E/\Q_p)$ where $\Gal(E/\Q_p)$ acts on the set of homomorphisms $F^+_p \ra E$ by post-composition.

We now specialize to the case where $S_p=\{v\}$ is a singleton.
Hence $F^+_p=K$ is an unramified extension of degree $f$ with ring of integers $\cO_K$ and residue field $k$. 
Let $W(k)$ be ring of Witt vectors of $k$, which is also the ring of integers of $K$.  

We denote the arithmetic Frobenius automorphism on $W(k)$ by $\phz$; it acts as raising to $p$-th power on the residue field.

Recall that we fixed a separable closure $\ovl{K}$ of $K$.
We choose $\pi \in \ovl{K}$ such that $\pi^{p^f-1} = -p$ and let $\omega_K : G_K \ra \cO_K^\times$ be the character defined by $g(\pi) = \omega_K(g) \pi$, which is independent of the choice of $\pi$.
We fix an embedding $\sigma_0: K \into E$ and define $\sigma_j = \sigma_0 \circ \phz^{-j}$, which identifies $\cJ = \Hom(k, \F) = \Hom_{\Qp}(K, E)$ with $\Z/f \Z$. 
We write $\omega_f:G_K \ra \cO^\times$ for the character $\sigma_0 \circ \omega_K$.

Let $\eps$ denote the $p$-adic cyclotomic character.  
If $W$ is a de Rham representation of $G_K$ over $E$, then for each $\kappa \in \Hom_{\Qp}(K, E)$, we write $\mathrm{HT}_{\kappa}(W)$ for the multiset of Hodge--Tate weights labelled by embedding $\kappa$ normalized so that the $p$-adic cyclotomic character $\eps$ has Hodge--Tate weight $\{1\}$ for every $\kappa$.  For $\mu = (\mu_j)_{j\in\cJ} \in X^*(\un{T})$, we say that a $3$-dimensional representation $W$ has Hodge--Tate weights $\mu$ if 
\[
\mathrm{HT}_{\sigma_j}(W) = \{ \mu_{1, j}, \mu_{2, j}, \mu_{3, j} \}.
\]
Our convention is the opposite of that of \cite{EGstack,CEGGPS}, but agrees with that of \cite{GHS}.

We say that a $3$-dimensional potentially semistable representation $\rho:G_K \ra \GL_n(E)$ has type $(\mu, \tau)$ if $\rho$ has Hodge--Tate weights $\mu$ and the restriction to $I_K$ of the Weil-Deligne representation attached to $\rho$ (via the \emph{covariant} functor $\rho\mapsto\mathrm{WD}(\rho)$) is isomorphic to the inertial type $\tau$. 
Note that this differs from the conventions of \cite{GHS} via a shift by $\eta$.

Let $\Gamma$ be a group.
If $V$ is a finite length $\Gamma$-representation, we let $\JH(V)$ be the (finite) set of Jordan--H\"older factors of $V$.
If $V^\circ$ is a finite $\cO$-module with a $\Gamma$-action, we write $\ovl{V}^\circ$ for the $\Gamma$-representation $V^\circ\otimes_{\cO}\F$ over $\F$.

If $X$ is an ind-scheme defined over $\cO$, we write $X_E\defeq X\times_{\Spec\cO} \Spec E$ and $X_{\F}\defeq X\times_{\Spec \cO}\Spec \F$ to denote its generic and special fiber, respectively. 
If $M$ is any $\cO$-module we write $\ovl{M}$ to denote $M\otimes_{\cO}\F$.

If $P$ is a statement, the symbol $\delta_P\in \{0,1\}$ takes value $1$ if $P$ is true, and $0$ if $P$ is false.

\section{Background}
\label{sec:background}

\subsection{Affine Weyl group, tame inertial types and Deligne--Lusztig representations}
\label{subsub:TIT:DL}

Throughout this section, we assume {that} $S_p=\{v\}$. 
Thus $\cO_p=\cO_K$ is the ring of {integers} of a finite unramified extension $K$ of $\Qp$ and $G_0 = \Res_{\cO_K/\Z_p} G_{/\cO_K}$.
We drop subscripts $v$ from notation and we identify $\cJ=\Hom_{\Qp}(K,E)$ with $\Z/f\Z$ via $\sigma_{j} \defeq \sigma_0 \circ \phz^{-j} \mapsto j$.

\subsubsection{Admissibility}
We follow \cite[\S 2.1--\S 2.4]{MLM}, specializing to {the} case of $n=3$.
We denote by $\leq$ the Bruhat order on $\un{\tld{W}}\cong X^*(\un{T})\rtimes \un{W}$ associated to the choice of the dominant base alcove $\un{C}_0$ and set 
\[
\Adm (\eta) = \{ \tld{w} \in \un{\tld{W}} \mid \tld{w} \leq t_{s(\eta)} \text{ for some } s \in \un{W} \}. 
\] 
We will also consider the partially ordered group $\un{\tld{W}}^\vee$ which is identified with $\tld{\un{W}}$ as a group, but whose Bruhat order is defined by the antidominant base alcove (and still denoted as $\leq$).
Then $\Adm^\vee(\eta)$ is defined as above, using now the antidominant order.
We have an order reversing bijection $\tld{w}\mapsto \tld{w}^*$ between $\tld{\un{W}}$ and $\tld{\un{W}}^\vee$  defined as $(\tld{w}^*)_j\defeq w_j^{-1}t_{\nu_j}$ if $\tld{w}_j=t_{\nu_j}w_j$.

\subsubsection{Tame inertial types and Deligne--Lusztig representations}
\label{subsub:TIT:DL}
An inertial type (for $K$) is the $\GL_3(E)$-conjugacy class of {a} homomorphism  $\tau:I_{K}\ra\GL_3(E)$ with open kernel and which extends to the Weil group of $G_K$.
An inertial type is \emph{tame} {if} it factors through the tame quotient of $I_{K}$. 
We will sometimes identify a tame inertial type with a fixed choice of a representative in its class.

Given ${s=(s_0,\dots,s_{f-1})\in \un{W}} $ and $\mu\in X^*(\un{T})\cap \un{C}_0$, we have an associated integer $r\in \{1,2,3\}$ (which is the order of the element $s_0s_1\dots s_{f-1}\in W$), integers $\bf{a}^{\prime(j')}\in \Z^3$ for $0\leq j'\leq fr-1$ and a tame inertial type $\tau(s,\mu+\eta)$ defined as $\tau(s,\mu+\eta)\defeq \sum_{i=1}^3(\omega_{fr})^{\mathbf{a}_i^{\prime(0)}}$ (see \cite[Example 2.4.1, equations (5.2), (5.1)]{MLM} for the details of this construction).
We say that $(s,\mu)$ is \emph{the lowest alcove presentation} for the tame inertial type $\tau(s,\mu+\eta)$ and that $\tau(s,\mu+\eta)$ is \emph{N-generic} if $\mu$ is $N$-deep in alcove $\un{C}_0$.
We say that a tame inertial type $\tau$ has \emph{a lowest alcove presentation} if there exists a pair $(s,\mu)$ as above such that $\tau\cong \tau(s,\mu+\eta)$ (in which case we will say that $(s,\mu)$ is a lowest alcove presentation for $\tau$), and that $\tau$ is $N$-generic if $\tau$ has a lowest alcove presentation $(s,\mu)$ such that $\mu$ is $N$-deep in alcove $\un{C}_0$.
We remark that different choices of pairs $(s,\mu)$ as above can give rise to isomorphic tame inertial types (see \cite[Proposition 2.2.15]{LLL}).  If $\tau$ is a tame inertial type of the form $\tau=\tau(s,\mu+\eta)$, we write $\tld{w}(\tau)$ for the element $t_{\mu+\eta}s\in\tld{\un{W}}$. (In particular, when writing $\tld{w}(\tau)$ we use an implicit lowest alcove presentation for $\tau$).

Repeating the above with $E$ replaced by $\F$, we obtain the notion of inertial $\F$-types and lowest alcove presentations for tame inertial $\F$-types. 
We use the notation $\taubar$ to denote a tame inertial $\F$-type $\taubar:I_{K}\ra\GL_3(\F)$. %
We say that a tame inertial $\F$-type is $N$-generic if it admits a lowest alcove presentation $(s,\mu)$ such that $\mu$ is $N$-deep in $\un{C}_0$.

If $\mu$ is $1$-deep in $\un{C}_0$, then for each $0\leq j'\leq fr-1$ there is a unique element $s'_{\orient,j'}\in W$ such that $  (s'_{\orient,j'})^{-1}(\bf{a}^{\prime\,(j')})$ is dominant.
(In the terminology of \cite{LLLM}, cf.~Definition 2.6 of \emph{loc.~cit.}, the $fr$-tuple $(s'_{\orient,j'})_{0\leq j'\leq fr-1}$ is the \emph{orientation} of $(\bf{a}^{\prime\,(j')})_{0\leq j'\leq fr-1}$.)

To a pair $(s,\mu)\in \un{W}\times X^*(\un{T})$, we can also associate a virtual $G_0(\Fp)$-representation over $E$ which we denote $R_s(\mu)$ (cf.~\cite[Definition 9.2.2]{GHS}, where $R_s(\mu)$ is denoted as $R(s,\mu)$).
In particular, $R_1(\mu)$ is a principal series representation.
If $\mu-\eta$ is $1$-deep in $\un{C}_0$ then $R_s(\mu)$ is an irreducible representation.
In analogy with the terminology for tame inertial type, if $\mu-\eta$ is $N$-deep in alcove $\un{C}_0$ for $N\geq 0$, we call $(s,\mu-\eta)$ an $N$-generic \emph{lowest alcove presentation} for $R_s(\mu)$, and say that $R_s(\mu)$ is \emph{$N$-generic}. 

\subsubsection{Inertial local Langlands correspondence}
\label{subsub:ILLC}
Given a tame inertial type $\tau: I_{K}\ra \GL_3(E)$, \cite[Theorem 3.7]{CEGGPS} gives an irreducible smooth $E$-valued representation $\sigma(\tau)$ of $G_0(\Fp)= \GL_3(k)$ over $E$ satisfying results towards the inertial local Langlands correspondence (see \emph{loc.~cit}.~for the properties satisfied by $\sigma(\tau)$).
(By inflation, we will consider $\sigma(\tau)$ as a smooth representation of $G_0(\Zp)$ without further comment.)
This representation need not be uniquely determined by $\tau$ and in what follows $\sigma(\tau)$ will denote either a particular choice that we have made or any choice that satisfies the properties of \cite[Theorem 3.7]{CEGGPS} (see also \cite[Theorem 2.5.3]{MLM} and the discussion following it).

When $\tau=\tau(s,\mu+\eta)$ is a tame inertial type such that $\mu\in \un{C}_0$ is $1$-deep, the representation $\sigma(\tau)$ can be taken to be $R_s(\mu+\eta)$ thanks to \cite[Corollary 2.3.5]{LLL}.
\subsubsection{Serre weights}
\label{subsubsec:SW}
We finally recall the notion of Serre weights for $G_0(\Fp)$, and the notion of \emph{lowest alcove presentations} for them, following \cite[\S 2.2]{MLM}.
A \emph{Serre weight} for $G_0(\Fp)$ is the isomorphism class of an (absolutely) irreducible representation of $G_0(\Fp)$ over $\F$. 
(We will sometimes refer to a representative for the isomorphism class as a Serre weight.)

Given $\lambda\in X_1(\un{T})$, we write $F(\lambda)$ for the Serre weight with highest weight $\lambda$; the assignment $\lambda\mapsto F(\lambda)$ induces a bijection between $X_1(\un{T})/(p-\pi)X^0(\un{T})$ and the set of Serre weights (cf.~\cite[Lemma 9.2.4]{GHS}). We say that $F(\lambda)$ is $N$-deep if $\lambda$ is (this does not depend on the choice of $\lambda$).

Recall from \cite[\S 2.2]{MLM} the equivalence relation on $\tld{\un{W}} \times X^*(\un{T})$ defined by $(\tld{w},\omega) \sim (t_\nu \tld{w},\omega-\nu)$ for all $\nu \in X^0(\un{T})$. 
For (an equivalence class of) a pair $(\tld{w}_1,\omega-\eta)\in \tld{\un{W}}^+_1\times (X^*(\un{T})\cap \un{C}_0)/\sim$ the Serre weight $F_{(\tld{w}_1,\omega)}\defeq F(\pi^{-1}(\tld{w}_1)\cdot (\omega-\eta))$ is well defined i.e.~is independent of the representative of the equivalence class of $(\tld{w}_1,\omega)$. 
The equivalence class of $(\tld{w}_1,\omega)$ is called a \emph{lowest alcove presentation} for the Serre weight $F_{(\tld{w}_1,\omega)}$. 
The Serre weight $F_{(\tld{w}_1,\omega)}$ is $N$-deep if and only if $\omega-\eta$ is $N$-deep in alcove $\un{C}_0$.
As above, we sometimes implicitly choose a representative for a lowest alcove presentation to make \emph{a priori} sense of an expression, though it is \emph{a posteriori} independent of this choice. 

\subsubsection{Compatibility for lowest alcove presentations} 
Recall that we have a canonical isomorphism $\tld{\un{W}}/\un{W}_a \cong X^*(\un{Z})$ where $\un{W}_a\cong \Lambda_{\un{R}} \rtimes \un{W}$ is the affine Weyl group of $\un{G}$.
Given an algebraic character $\zeta\in X^*(\un{Z})$, we say tha t an element $\tld{w}\in \tld{\un{W}}$ is \emph{$\zeta$-compatible} if it corresponds to $\zeta$ via the isomorphism $\tld{\un{W}}/\un{W}_a \cong X^*(\un{Z})$.
In particular, a lowest alcove presentation $(s,\mu)$ for a tame inertial type (resp.~a lowest alcove presentation $(s,\mu-\eta)$ for a Deligne--Lusztig representation) is $\zeta$-compatible if the element $t_{\mu+\eta} s \in \tld{\un{W}}$ (resp.~$t_{\mu} s \in \tld{\un{W}}$) is $\zeta$-compatible.
Similarly, a lowest alcove presentation $(\tld{w}_1,\omega)$ for Serre weight is \emph{$\zeta$-compatible} if the element $t_{\omega-\eta}\tld{w}_1\in \tld{\un{W}}$ is $\zeta$-compatible. %

\subsubsection{A comparison to \cite{LLLM2}}
\label{subsec:param:SW}
In \cite{LLLM2}, the parametrization of Serre weights is slightly different from the one in \cite{MLM}.
Here, we give a dictionary between the two.

Define a map
\begin{align}
\label{eqn:centralchar} \tld{\un{W}} \times X^*(\un{T}) &\longrightarrow \tld{\un{W}}/\un{W}_a \cong X^*(\un{Z}) \\
\nonumber (\tld{w},\omega) &\longmapsto t_{\omega-\eta} \tld{w} \un{W}_a
\end{align}
and write $(\tld{\un{W}} \times X^*(\un{T}))^\zeta$ for the preimage of $\zeta \in X^*(\un{Z})$ (presentations \emph{compatible} with $\zeta$). 
The map \eqref{eqn:centralchar} is constant on equivalence classes, and we write $(\tld{\un{W}} \times X^*(\un{T}))^\zeta/\sim$ for the set of equivalence classes in the preimage of $\zeta$. 
The equivalence relation restricts to one on $\tld{\un{W}}_1 \times X^*(\un{T})$ or $\tld{\un{W}}_1 \times (X^*(\un{T}) \cap \un{C}_0 + \eta)$, and we use similar notation, e.g.~$(\tld{\un{W}}_1 \times (X^*(\un{T}) \cap \un{C}_0 + \eta))^\zeta/\sim$, for these subsets. 

We let $\un{\Lambda}_W$ and $\tld{W}^{\der}$ be $X^*(\un{T})/X^0(\un{T})$ and $\tld{W}/X^0(\un{T})$, respectively. 
Recall from \cite[\S 2.1]{LLLM2} the set 
\[
\cP^{\der} = \{(\omega,\tld{w}) \in \un{\Lambda}_W \times \tld{\un{W}}_1^{\der,+} \mid t_\omega \pi(\tld{w}) \in \un{W}_a\}. 
\]
Letting $\cA$ be the set of $p$-restricted alcoves in $X^*(\un{T}) \otimes_{\Z} \R$, the map 
\begin{align*} 
 \beta: \cP^{\der} &\longrightarrow \un{\Lambda}_W \times \cA \\
(\omega,\tld{w}) &\longmapsto (\omega,\pi(\tld{w}) \cdot \un{C}_0)
\end{align*}
is a bijection by \cite[Lemma 2.1.1]{LLLM2}. 

For $\lambda \in X^*(\un{T})$, the map
\begin{align*}
(\tld{\un{W}}_1 \times X^*(\un{T}))^{\lambda-\eta|_{\un{Z}}}/\sim &\overset{\iota_\lambda}{\longrightarrow} \cP^{\der} \\
(\tld{w},\omega) &\longmapsto (\omega - \lambda,\pi^{-1}(\tld{w}))
\end{align*}  
is a bijection. 
(Here $\omega-\lambda$ also denotes its image in $\un{\Lambda}_W$.) 

Then $\beta \circ \iota_\lambda: (\tld{\un{W}}_1 \times X^*(\un{T}))^{\lambda-\eta|_{\un{Z}}}/\sim \longrightarrow \un{\Lambda}_W \times \cA$ is a bijection which induces a bijection 
\begin{align}
\label{eq:bij:SW} (\tld{\un{W}}_1 \times X^*(\un{T}) \cap \un{C}_0 + \eta)^{\lambda-\eta|_{\un{Z}}}/\sim &\longrightarrow \un{\Lambda}_W^\lambda \times \cA\\
\nonumber (\tld{w},\omega) &\longmapsto (\omega-\lambda,\tld{w} \cdot \un{C}_0)
\end{align}
when $\lambda -\eta \in\un{C}_0$ {and $\un{\Lambda}_W^\lambda$ is defined to be the set of $\omega'\in \un{\Lambda}_W$ satisfying $\omega'+\lambda -\eta\in\un{C}_0$, see \cite[\S 2.1]{LLLM2}.}
By the definition of $\mathfrak{Tr}_{\lambda}$ in \cite[\S 2.1]{LLLM2}, for $(\tld{w},\omega) \in (\tld{\un{W}}_1 \times X^*(\un{T}) \cap \un{C}_0 + \eta)^{\lambda-\eta|_{\un{Z}}}/\sim$, we have: 
\begin{equation}
\label{eq:SW:LAP}
F_{(\tld{w},\omega)} = F(\mathfrak{Tr}_{\lambda}(\omega-\lambda,\tld{w}\cdot \un{C}_0)).
\end{equation}  

\subsubsection{Reduction of Deligne--Lusztig representations}\label{sec:reduction}

For $i\in\{1,2\}$, let $\eps_i$ denote the image of $\eps_i'$ via the surjection $X^*(\un{T})\onto \un{\La}_W$.

\begin{prop}\label{prop:JH}
Let $\lambda-\eta$ and $\mu-\eta$ be $0$-deep and $1$-deep in $\un{C}_0$, respectively, such that $\mu+\eta-\lambda\in \un{\Lambda}_R$. 
If $\sigma \in \JH(\ovl{R_s(\mu)})$ is a $0$-deep Serre weight, then $\sigma$ is contained in $F\Big(\mathfrak{Tr}_{\lambda}\big(t_{\mu-\lambda}s\big(\Sigma\big)\big)\Big)$, where $\Sigma=(\Sigma_0)^{f}\subseteq \un{\Lambda}^{\lambda+\eta}_W\times\mathcal{A}$ and
\[
\Sigma_0 \defeq  \begin{Bmatrix} (\eps_1+ \eps_2, 0), (\eps_1 - \eps_2, 0), (\eps_2 - \eps_1, 0) \\
 (0, 1) , (\eps_1, 1), (\eps_2, 1) \\
 (0, 0), (\eps_1, 0), (\eps_2, 0) 
 \end{Bmatrix}.
 \]
If $\mu-\eta$ is furthermore $2$-deep, then $\JH(\ovl{R_s(\mu)})$ is $F\Big(\mathfrak{Tr}_{\lambda}\big(t_{\mu-\lambda}s\big(\Sigma\big)\big)\Big)$. 
\end{prop}
\begin{proof}
\cite[Appendix, Theorem 3.4]{herzig-duke} gives the identity 
\[
\ovl{R_s(\mu)} = \sum_{\tld{w} \in \tld{\un{W}}_1^+/X^0(\un{T})} \ovl{W}(\tld{w} \cdot (t_\mu s (\tld{w}_h\tld{w})^{-1}(0) - \eta))
\]
at the level of characters (in our situation, ${\gamma'}^{\mathrm{Fr-}1}_{w_1,w_2}$ is $1$ if  $w_1=w_2$ and is $0$ otherwise). 
That $\mu - \eta$ is $1$-deep in $\un{C}_0$ implies that $\tld{w} \cdot (t_\mu s (\tld{w}_h\tld{w})^{-1}(0) - \eta)$ is $-1$-deep in a $p$-restricted alcove. 
This implies that for each $\tld{w} \in \tld{\un{W}}_1^+$, $\tld{w} \cdot (t_\mu s (\tld{w}_h\tld{w})^{-1}(0) - \eta)+\eta$ is dominant so that $\tld{w} \cdot (t_\mu s (\tld{w}_h\tld{w})^{-1}(0) - \eta)$ is dominant or $W(\tld{w} \cdot (t_\mu s (\tld{w}_h\tld{w})^{-1}(0) - \eta))$ is zero. 
If $\sigma \in \JH(\ovl{R_s(\mu)})$, then $\sigma \in \JH(\ovl{W}(\tld{w} \cdot (t_\mu s (\tld{w}_h\tld{w})^{-1}(0) - \eta)))$ for some $\tld{w} \in \tld{\un{W}}_1^+$. 
\cite[Proposition 4.9]{florian-thesis} gives a decomposition of such $p$-restricted Weyl modules. 
The proof of \cite[Proposition~2.3.4]{LLLM2} shows that $\sigma\in F\Big(\mathfrak{Tr}_{\lambda}\big(t_{\mu-\lambda}s\big(\Sigma\big)\big)\Big)$. 
If $\mu-\eta$ is $2$-deep, then $\tld{w} \cdot (t_\mu s (\tld{w}_h\tld{w})^{-1}(0) - \eta)$ is $0$-deep in a $p$-restricted alcove and hence dominant. 
The description of $\JH(\ovl{R_s(\mu)})$ again follows from the proof of \cite[Proposition~2.3.4]{LLLM2}. 
\end{proof}

If $\mu-\eta$ is $1$-deep in $\un{C}_0$, we let the subset $\JH_{\mathrm{out}}(\ovl{R_s(\mu)})\subset\JH(\ovl{R_s(\mu)})$ be the Serre weights of the form $F(\tld{w} \cdot (t_\mu s (\tld{w}_h\tld{w})^{-1}(0) - \eta))$ for some $\tld{w} \in  \tld{\un{W}}_1^+$. 
We call the elements of $\JH_{\mathrm{out}}(\ovl{R_s(\mu)})$ the \emph{outer weights} (of $\JH(\ovl{R_s(\mu)})$). 
In the notation of Proposition \ref{prop:JH}, $\JH_{\mathrm{out}}(\ovl{R_s(\mu)})$ is the subset $F\Big(\mathfrak{Tr}_{\lambda}\big(t_{\mu-\lambda}s\big(\Sigma_{\mathrm{out}}\big)\big)\Big)$, where $\Sigma_{\mathrm{out}}=(\Sigma_{\mathrm{out},0})^{f}\subseteq \un{\Lambda}^{\lambda+\eta}_W\times\mathcal{A}$ and
\[
\Sigma_{\mathrm{out},0} \defeq  \begin{Bmatrix} (\eps_1+ \eps_2, 0), (\eps_1 - \eps_2, 0), (\eps_2 - \eps_1, 0) \\
 (0, 1) , (\eps_1, 1), (\eps_2, 1) 
 \end{Bmatrix}.
 \]

If $\lambda-\eta$ and $\mu-\eta$ are $0$-deep and $2$-deep in $\un{C}_0$ respectively, and $\mu+\eta-\lambda \in \un{\Lambda}_R$, we define $W^?(\taubar(s,\mu+\eta))$ to be the set of Serre weights $F\Big(\mathfrak{Tr}_{\lambda}\Big(t_{{\mu+\eta-\lambda}}s\big(r(\Sigma)\big)\Big)\Big)$, where $r(\Sigma)$ is defined by swapping the digits of $\un{a}\in\cA$ in the elements $(\eps,\un{a})\in\Sigma$.

\subsubsection{The covering order}
\label{subsub:CO}
\begin{defn}\label{defn:cover}
We say that a $3$-deep {Serre} weight $\sigma_0$ \emph{covers} $\sigma$ if 
\[
\sigma \in \underset{\substack{R \, 1\textrm{-generic}, \\ \sigma_0 \in \JH(\overline{R})}}{\bigcap} \JH(\overline{R})
\]
(where $R$ runs over $1$-generic Deligne--Lusztig representations).
\end{defn}

\begin{lemma}\label{lemma:coverup}
A Serre weight $\sigma_1$ covers $\sigma_2$ if and only if $\sigma_2 \uparrow \sigma_1$.
\end{lemma}
\begin{proof}
This follows from \cite[Proposition 2.3.12(4)]{MLM} where the slightly weaker genericity hypotheses come from \S \ref{sec:reduction}.
\end{proof}

\subsubsection{$L$-parameters}
\label{subsub:Lp}
We now assume that $S_p$ has arbitrary finite cardinality.
An $L$-parameter (over $E$) is a $\un{G}^\vee(E)$-conjugacy class of a continuous homomorphism $\rho:G_{\Qp}\ra{}^L\un{G}(E)$ which is compatible with the projection to $\Gal(E/\Qp)$ (such homomorphism is called \emph{$L$-homomorphism}).
An inertial $L$-parameter is a $\un{G}^\vee(E)$-conjugacy class of {a} homomorphism $\tau:I_{\Qp}\ra\un{G}^\vee(E)$ with open kernel, and which admits an extension to an $L$-homomorphism.
An inertial $L$-parameter is \emph{tame} if some (equivalently, any) representative in its equivalence class factors through the tame quotient of $I_{\Qp}$.

Fixing isomorphisms $\ovl{F^+_v}\stackrel{\sim}{\ra}\ovl{\Q}_p$ for all $v\in S_p$, we have a bijection between $L$-parameters (resp.~tame inertial $L$-parameters) and collections of the form $(\rho_v)_{v\in S_p}$ (resp.~of the form $(\tau_v)_{v\in S_p}$) where for all $v\in S_p$ the element $\rho_v:G_{F^+_v}\ra\GL_3(E)$ is a continuous Galois representation (resp.~the element $\tau_v:I_{F^+_v}\ra\GL_3(E)$ is a tame inertial type for $F^+_v$).

We have similar notions when $E$ is replaced by $\F$.
Again {we} will often abuse terminology, and identify an $L$-parameter (resp.~a tame inertial $L$-parameter) with a fixed choice of a representative in its class.
This shall cause no confusion, and 
nothing in what follows will depend on this choice.

The definitions and results of \S \ref{subsub:TIT:DL}--\ref{subsub:CO} generalize in the evident way for tame inertial $L$-parameters and $L$-homomorphism.
(In the case of the inertial local Langlands correspondence of \S \ref{subsub:ILLC}, given a tame inertial $L$-parameter $\tau$ corresponding to the collection of tame inertial types $(\tau_v)_{v\in S_p}$, we let $\sigma(\tau)$ be the irreducible smooth $E$-valued representation of $G_0(\Z_p)$ given by $\otimes_{v\in S_p} \sigma(\tau_v)$.)

\subsection{Breuil--Kisin modules}
\label{subsub:BKM}
We recall some background on Breuil--Kisin modules with tame descent data.
We refer the reader to \cite[\S 3.1,3.2]{LLLM2} and \cite[\S 5.1]{MLM} for further detail, with the caveat that we are following the conventions of the latter on the labeling of embeddings for tame inertial types and Breuil--Kisin modules (see~\emph{loc.~cit}.~Remark 5.1.2).

Let $\tau=\tau(s,\mu+\eta)$ be a tame inertial type with lowest alcove presentation $(s, \mu)$ which we fix throughout this section (recall that $\mu$ is $1$-deep in $\un{C}_0$).
Recall that $r\in\{1,2,3\}$ is the order of ${s_0 s_{1} s_{2} \cdots s_{f-1} \in W}$.
We write $K'/K$ for the unramified extension of degree $r$ contained in $\ovl{K}$ and set  $f'\defeq fr$, $e' \defeq p^{f'}-1$.
We identify $\Hom_{\Qp}(K',E)$ with $\Z/f' \Z$ via $\sigma_{j'} \defeq \sigma'_0\circ \phz^{-j'} \mapsto j'$ where $\sigma'_0:K' \iarrow E$ is a fixed choice for an embedding extending $\sigma_0:K \iarrow E$. %
In this way, restriction of embeddings corresponds to reduction modulo $f$ in the above identifications.

Let $\pi'\in \ovl{K}$ be an $e'$-th root of $-p$, let $L' \defeq K'(\pi')$ and $\Delta' \defeq \Gal(L'/K') \subset \Delta \defeq \Gal(L'/K)$. 
We define the $\cO_{K'}^{\times}$-valued character $\omega_{K'}(g) \defeq  \frac{g(\pi')}{\pi'}$ for $g \in \Delta'$ (this does not depend on the choice of $\pi'$).
Given an $\cO$-algebra $R$, we set $\fS_{L', R} \defeq (W(k') \otimes_{\Zp} R)[\![u']\!]$. 
The latter is endowed with an endomorphism $\varphi:\fS_{L', R} \ra \fS_{L', R}$ acting as Frobenius on $W(k')$, trivially on $R$, and sending $u'$ to $(u')^{p}$.  
It is endowed moreover with an action of $\Delta$ as follows: for any $g$ in $\Delta'$, $g(u') = \frac{g(\pi')}{\pi'} u' = \omega_{K'}(g) u'$   
and $g$ acts trivially on the coefficients; if $\sigma^f \in\Gal(L'/K)$ is the lift of $p^f$-Frobenius on $W(k')$ which fixes $\pi'$, then $\sigma^f$ is a generator for $\Gal(K'/K)$, acting in natural way on $W(k')$ and trivially on both $u'$ and $R$. 
Set $v = (u')^{e'}$, 
\[
\fS_R \defeq (\fS_{L', R})^{\Delta = 1} = (W(k) \otimes_{\Zp} R)[\![v]\!]
\]
and $E(v) \defeq v + p = (u')^{e'} + p$.

Let $Y^{[0, 2]}(R)$ be the groupoid of Breuil--Kisin modules of rank $3$ over $\fS_{L', R}$, height in $[0,2]$ and descent data of type $\tau$ (cf.~\cite[\S 3]{CL}, \cite[Definition 3.1.3]{LLLM2}, \cite[Definition 5.1.3]{MLM}):
\begin{defn}
An object of $Y^{[0, 2],\tau}(R)$ is the datum of:
\begin{enumerate} 
\item  a finitely generated projective $\fS_{L', R}$-module $\fM$ which is locally free of rank $3$;
\item an injective $\fS_{L', R}$-linear map $\phi_\fM:\phz^*(\fM)\ra\fM$ whose cokernel is annihilated by $E(v)^2$; and
\item a semilinear action of $\Delta$ on $\fM$ which commutes with $\phi_{\fM}$, and such that, for each $j' \in \Hom_{\Qp}(K',E)$, 
\[
(\fM\otimes_{W(k'),\sigma_{j'}}R) \mod u' \cong \tau^{\vee} \otimes_{\cO} R 
\]  
as $\Delta'$-representations.
\end{enumerate}
\end{defn}
Note that $\fM^{(j')}\defeq \fM\otimes_{W(k'),\sigma_{j'}}R$ is a $R[\![u']\!]$-submodule of $\fM$ in a standard way, endowed with a semilinear action of $\Delta'$ and the Frobenius $\phi_{\fM}$ induces $\Delta'$-equivariant morphisms $\phi_\fM^{(j')}:\phz^*(\fM^{(j'-1)})\ra \fM^{(j')}$.    
In particular, by letting $\tau'$ denote the tame inertial type for $K'$ obtained from $\tau$ via the identification $I_{K'}=I_K$ induced by the inclusion $K'\subseteq \ovl{K}$, the semilinear action of $\Delta$ induces an isomorphism $\iota_{\fM}:(\sigma^f)^*(\fM) \cong \fM$ (see \cite[\S 6.1]{LLLM}) as elements of $Y^{[0,2], \tau'}(R)$.

Let $\fM \in Y^{[0, 2], \tau}(R)$.
Recall that an \emph{eigenbasis} of $\fM$ is a collection of bases ${\beta}^{(j')}=(f_1^{(j')},f_2^{(j')},f_3^{(j')})$ for each $\fM^{(j')}$ such that $\Delta'$ acts on $f_i^{(j')}$ via the character $\omega_{f'}^{-\mathbf{a}^{{\prime}\,(0)}_{i}}$ {(see \S \ref{subsub:TIT:DL} for the definition of $\mathbf{a}^{\prime\,(0)}\in \Z^3$)} and such that $\iota_{\fM} ((\sigma^f)^*(\beta^{(j')})) = \beta^{(j' + f)}$ for all $j'\in \Hom_{\Qp}(K',E)$.
Given an eigenbasis $\beta$ for $\fM$, we let $C^{(j')}_{\fM, \beta}$ be the matrix of $\phi_\fM^{(j')}:\phz^*(\fM^{(j' - 1)})\ra \fM^{(j')}$ with respect to the bases $\phz^*(\beta^{(j'-1)})$ and $\beta^{(j')}$ and set
\[
A^{(j')}_{\fM,\beta}\defeq 
\Ad\left(
(s'_{\orient,j'})^{-1}  (u^{\prime})^{-\bf{a}^{\prime\,(j')}}
\right)(C^{(j')}_{\fM,\beta})
\]
for $j'\in \Hom_{\Qp}(K',E)$.
It is an element of $\GL_3(R(\!(v+p)\!))$ with coefficients in $R[\![v+p]\!]$, is upper triangular modulo $v$ and only depends on the restriction of $j'$ to $K$ (see \cite[\S 5.1]{MLM}).

Let $\Iw(\F)$ denote the Iwahori subgroup of $\GL_3(\F(\!(v)\!))$ relative to the Borel of upper triangular matrices.
We define the \emph{shape} of a mod $p$ Breuil--Kisin module $\fM \in Y^{[0,2], \tau}(\F)$ to be the element $\tld{z} = (\widetilde{z}_j) \in \widetilde{\un{W}}^{\vee}$ such that for any eigenbasis $\beta$ and any $j \in \cJ$, the matrix $A^{(j)}_{\fM, \beta}$ lies in $\Iw(\F) \widetilde{z}_j \Iw(\F)$.
This notion doesn't depend on the choice of eigenbasis, see \cite[Proposition 5.1.8]{MLM} (but it \emph{does} depend on the lowest alcove presentation of $\tau$, see \emph{ibid}.~Remark 5.1.5).

\section{Local models in mixed characteristic and the Emerton--Gee stack}
\label{sec:LM:EG}

We assume throughout this section that $S_p=\{v\}$ so $\cO_p=\cO_K$ is the ring of integer of a finite unramified extension $K$ of $\Qp$.  We identify $\cJ=\Hom_{\Qp}(K,E)$ with $\Z/f\Z$ via $\sigma_{j} \defeq \sigma_0 \circ \phz^{-j} \mapsto j$.

\subsection{Local models in mixed characteristic}\label{sec:localmodel}

We now define the mixed characteristic local models which are relevant to our paper.
We follow closely \cite[\S 4]{MLM} and the notation therein.

For any Noetherian $\cO$-algebra $R$, define 
\begin{align*}
L \cG_{\cO}(R)&\defeq \{A \in \GL_3(R(\!(v+p)\!)),\ \text{ $A$ is upper triangular modulo $v$}\};\\
L^+\cG_{\cO}(R)&\defeq \{A \in \GL_3(R[\![v+p]\!]),\ \text{ $A$ is upper triangular modulo $v$}\};
\\
L^{[0,2]} \cG_{\cO}(R)&\defeq \left\{
A \in L \cG_{\cO}(R),
\begin{array}{l}
\text{$A$, $(v+p)^2A^{-1}$ are elements of $\Mat_3(R[\![v+p]\!])$}
\\ 
\text{and are upper triangular modulo $v$}
\end{array}\right\}.
\end{align*}
The fpqc quotients $L^+\cG_{\cO} \backslash L^{[0,2]}\cG_{\cO}\into L^+\cG_{\cO}\backslash L\cG_{\cO}$ induced from inclusions $L^+\cG_{\cO}(R)\subseteq L^{[0,2]} \cG_{\cO}(R)\subseteq L \cG_{\cO}(R)$ are representable by a projective scheme $\Gr^{[0,2]}_{\cG,\cO}$ and an ind-projective ind-scheme $\Gr_{\cG,\cO}$ respectively.

For any $\tld{z}=(\tld{z}_j)_{j\in\cJ}\in \tld{\un{W}}^{\vee}$ and any Noetherian $\cO$-algebra $R$, define 
\[
\tld{U}(\tld{z})(R)=\big(\tld{U}(\tld{z}_j)(R)\big)_{j\in\cJ}\subseteq (L\cG_{\cO}(R))^{\cJ}
\]
to be the set of $f$-tuples of matrices $(A^{(j)})_{j\in\cJ}\in (L\cG_{\cO}(R))^{\cJ}$ such that for all $1\leq i,k\leq 3$ and $j\in\cJ$,
\begin{itemize}
\item $A^{(j)}_{ik}\in v^{\delta_{i>k}}R\left[v+p, \frac{1}{v+p}\right]$; 
\item $\deg_{v+p}(A^{(j)}_{ik})\leq \nu_{j,k}-\delta_{i<z_j(k)}$; and
\item
$\deg_{v+p}(A^{(j)}_{z_j(k)k})=\nu_{j,k}$ and the coefficient of the leading term is a unit of $R$;
\end{itemize}
where we have written $\tld{z}=z t_\nu$ and $\nu=(\nu_{j,1},\nu_{j,2},\nu_{j,3})_{j\in \cJ}$ where we have written $\tld{z}=z t_\nu$ and $\nu=(\nu_{j,1},\nu_{j,2},\nu_{j,3})_{j\in \cJ}$ {(and recall from \S \ref{subsec:notations} the notation for the Kronecker deltas $\delta_{i>k}$, $\delta_{i<z_j(k)}$).}
We set $\tld{U}^{[0,2]}(\tld{z})(R)\defeq \tld{U}(\tld{z})(R)\cap (L^{[0,2]}\cG(R))^{\cJ}$.
Note that both $\tld{U}^{[0,2]}(\tld{z})$ and $\tld{U}(\tld{z})$ are endowed with a $\un{T}^{\vee}_{\cO}$-action induced by left multiplication of matrices. 
It follows from \cite[Lemmas 3.2.2 and 3.2.7]{MLM} that the natural map $\tld{U}(\tld{z})\rightarrow \Gr^{\,\cJ}_{\cG,\cO}$ (resp.~$\tld{U}^{[0,2]}(\tld{z})\rightarrow \Gr^{[0,2],\,\cJ}_{\cG,\cO}$) factors as a $\un{T}^{\vee}_{\cO}$-torsor map followed by an open immersion. 
(We have written $\Gr^{\cJ}_{\cG,\cO}$ for the product, over $\Spec \cO$, of $f$-copies of $\Gr_{\cG,\cO}$ indexed over elements $j\in\cJ$, and $\Gr^{[0,2],\cJ}_{\cG,\cO}$ is defined similarly.)

\vspace{3mm}

We now compare the objects above with groupoids of Breuil--Kisin modules with tame descent.
Let $(s,\mu)\in \un{W}\times X^*(\un{T})$ be a lowest alcove presentation for the tame inertial type $\tau\defeq \tau(s,\mu+\eta)$.
We have the \emph{twisted shifted} conjugation action of $\un{T}_{\cO}^\vee$ on $\tld{U}^{[0,2]}(\tld{z})$ given by
\[
A^{(j)}\mapsto t_j A^{(j)} \mathrm{Ad}(s^{-1}_j)(t_{j-1}),
\]
which is exactly the restriction to $\un{T}_{\cO}^\vee$ of the $(s,\mu)$-twisted $\phz$-conjugation in \cite[\S 5.2]{MLM}.
By \cite[Corollary 5.2.3]{MLM} the quotient of $\tld{U}^{[0,2]}(\tld{z})_{\F}$ by this action is isomorphic to an open substack $Y^{[0,2],\tau}_{\F}(\tld{z})$ of $Y^{[0,2],\tau}_{\F}$.
We denote by $Y^{[0,2],\tau}(\tld{z})$ the open substack of $Y^{[0,2],\tau}$ induced by $Y^{[0,2],\tau}_{\F}(\tld{z})$ (see \cite[Definition 5.2.4]{MLM}).

By \cite[Theorem 5.3.1]{MLM}, whenever $\mu$ is $3$-deep in $\un{C}_0$, we have a morphism of $p$-adic formal algebraic stacks over $\cO$:
\[
\tld{U}^{[0,2]}(\tld{z})^{\wedge_p}\rightarrow  Y^{[0,2],\tau}(\tld{z})\into Y^{[0,2],\tau}
\]
where the left map is a $\un{T}^{\vee}_{\cO}$-torsor for the twisted shifted conjugation action on the source and the second map is an open immersion.

We finally consider Breuil--Kisin modules with height bounded by the cocharacter $\eta\in X_*(\un{T}^\vee)$.
The fiber $\Gr^{\cJ}_{\cG,E}$ of $\Gr^{\cJ}_{\cG,\cO}$ over $E$ is the affine Grassmannian of $\GL_3$ and we let $M_{\cJ}(\leqeta)$ be the Zariski closure in $\Gr^{\cJ}_{\cG,\cO}$ of the open affine Schubert cell associated to $(v+p)^\eta$ in $\Gr^{\cJ}_{\cG,E}$.
Let $\tld{U}(\tld{z},\leqeta)$ be the pull back of $\tld{U}^{[0,2]}(\tld{z})$ along the closed immersion $M_{\cJ}(\leqeta)\into \Gr^{\cJ}_{\cG,\cO}$.

Let $Y^{\leqeta,\tau}$ denote the closed $p$-adic formal substack of $Y^{[0,2],\tau}$ appearing in \cite[\S 5]{CL} (and denoted $Y^{\eta,\tau}$ in loc.~cit.).
The computations of \cite[\S 4]{LLLM} give the following:
\begin{prop}
\label{prop:loc:mod:diag:1}
Let $\tld{z}\in\Adm^{\vee}(\eta)$. 
Then $\tld{U}(\tld{z},\leqeta)$ is an affine scheme over $\cO$, with presentations $\bigotimes_{j=0}^{f-1}\cO(\tld{U}(\tld{z}_j, \leq\!\eta_j))$ where the $\cO$-algebras $\cO(\tld{U}(\tld{z}_j, \leq\!\eta_j))$ are given in Table \ref{table:coord1}.
Moreover, let $(s,\mu)\in \un{W}\times (X^*(\un{T})\cap \un{C}_0)$ be a lowest alcove presentation of $\tau\defeq \tau(s,\mu+\eta)$, with $\mu$ being $3$-deep in $\un{C}_0$ and denote by $Y^{\leqeta,\tau}(\tld{z})$ the pull back of $Y^{[0,2],\tau}(\tld{z})$ along the closed immersion $Y^{\leqeta,\tau}\into Y^{[0,2],\tau}$.

Then we have a morphism of $p$-adic formal algebraic stacks over $\cO$
\begin{equation}
\label{eq:LMD:1}
\tld{U}(\tld{z},\leqeta)^{\wedge_p}\rightarrow  Y^{\leqeta,\tau}(\tld{z})\into Y^{\leqeta,\tau}.
\end{equation}
where the left map is a $\un{T}^{\vee}_{\cO}$-torsor for the twisted shifted conjugation action on the source and the second map is an open immersion.
\end{prop} 

\subsection{Monodromy condition}
\label{sub:Mon:Cond}

We introduce closed subspaces of $Y^{\leqeta,\tau}$, $Y^{\leqeta,\tau}(\tld{z})$ and $\tld{U}(\tld{z},\leqeta)^{\wedge_p}$, and compare them with potentially crystalline substacks of the Emerton--Gee stack (introduced below).

{Recall the element $\eta\in X^*(\un{T})$ we fixed in \S \ref{subsec:notations}.} 
Let $\tau\defeq \tau(s,\mu+\eta)$ be a tame inertial type with lowest alcove presentation $(s, \mu)$, where $\mu$ is $1$-deep in alcove $\un{C}_0$.
By \cite[Proposition 7.1.6]{MLM} the datum of a $p$-adically complete, topologically finite type flat $\cO$-algebra $R$, and a morphism $f:\Spf\, R \to Y^{[0,2],\tau}$ (corresponding to an element $\fM\in Y^{[0,2],\tau}(R)$) defines a $p$-saturated ideal $I_{\fM,\nabla_{\infty}}\subset R$ which is compatible with flat base change. 
(In the terminology of \emph{loc.~cit}., and in the case where $\fM$ is free over $\fS_{L',R}$, the morphism $f:\Spf\, R \to Y^{[0,2],\tau}$ factors through $\Spf\, R\ra \Spf\, (R/I_{\fM,\nabla_{\infty}})$ if and only if $\fM$ satisfies the \emph{monodromy condition} \cite[Definition 7.1.2]{MLM}.)
This gives rise to a $\cO$-flat closed substack $Y^{[0,2],\tau,\nabla_\infty}\into Y^{[0,2],\tau}$ (cf.~\cite[\S 7.2]{MLM}) characterized by the property that for any $\fM\in Y^{[0,2],\tau}(R)$ corresponding to $\Spf\, R \to Y^{[0,2],\tau}$, the pullback of the substack $Y^{[0,2],\tau,\nabla_\infty}$ is $\Spf(R/I_{\fM,\nabla_\infty})$.
We finally define $Y^{\leqeta,\tau,\nabla_\infty}$ as the pullback of $Y^{[0,2],\tau,\nabla_\infty}\into Y^{[0,2],\tau}$ along the closed immersion $Y^{\leqeta,\tau}\into Y^{[0,2],\tau}$.

\vspace{3mm}

Let $\cX_{K,3}$ be the Noetherian formal algebraic stack over $\Spf\,\cO$ defined in \cite[Definition 3.2.1]{EGstack}.
It has the property that for any complete local Noetherian $\cO$-algebra $R$ with finite residue field, the groupoid $\cX_{K,3}(R)$ is equivalent to the groupoid of rank $3$ projective $R$-modules equipped with a continuous $G_K$-action, see \cite[\S 3.6.1]{EGstack}.
(In particular, we will consider closed points of $\cX_{K,3}(\F)$ as continuous Galois representations $\rhobar:G_K\rightarrow \GL_3(\F)$, and conversely.)
Moreover, by \cite[Theorem 4.8.12]{EGstack}, there is a unique $\cO$-flat closed formal substack $\cX^{\eta,\tau}$ of $\cX_{K,3}$ which parametrizes, over finite flat $\cO$-algebras, those $G_K$-representations which after inverting $p$ are 
potentially crystalline with Hodge-Tate weight $\eta$ and inertial type $\tau$. We define $\cX^{\leqeta,\tau}$ in the same way, except that the condition on Hodge-Tate weights becomes $\leqeta$. In particular, $\cX^{\leqeta,\tau}$ is the scheme theoretic union of the substacks $\cX^{\lambda,\tau}$ for $\lambda$ dominant and $\lambda\leq \eta$.

The substacks $\cX^{\eta,\tau}$, $\cX^{\leqeta,\tau}$ have the following fundamental properties:
\begin{thm}[\cite{EGstack}, Theorem 4.8.12, Proposition 4.8.10]
\label{thm:EG_basic_properties}
Let $\cX^{?,\tau}$ denote either the substack $\cX^{\eta,\tau}$ or $\cX^{\leqeta,\tau}$.
Then
\begin{enumerate}
\item the stack $\cX^{?,\tau}$ is a $p$-adic formal algebraic stack, flat and topologically of finite type over $\Spf\, \cO$ %
\item $\cX^{\eta,\tau}$ is equidimensional of dimension $1+3f$, while for $\lambda$ dominant and $\lambda< \eta$, $\cX^{\lambda,\tau}$ is equidimensional of dimension $<1+3f$; 
\item let $\rhobar \in \cX^{\eta,\tau}(\bF)$ corresponding to a mod $p$ representation of $G_K$. 
Then the potentially crystalline deformation ring $R_{\rhobar}^{\eta,\tau}$ is a versal ring to $\cX^{\eta,\tau}$ at $\rhobar$; and
\item for any smooth map $\Spf\, R\ra \cX^{?,\tau}$ with $R$ being a $p$-adically complete, topologically of finite type $\cO$-algebra, the ring $R$ is reduced and $R[1/p]$ is regular.
\end{enumerate}
\end{thm}

Using Proposition \ref{prop:loc:mod:diag:1}, we can finally relate the objects introduced so far:
\begin{thm}
\label{prop:loc:mod:diag:2}
Let $\tld{z}\in\Adm^{\vee}(\eta)$ and assume that the character $\mu$ (appearing in the lowest alcove presentation $(s,\mu)$ of $\tau$) is $4$-deep.
We have a commutative diagram of $p$-adic formal algebraic stacks over $\Spf\,\cO$:
\[
\xymatrix{
\tld{U}(\tld{z},\leqeta)^{\wedge_p}\ar[r]& Y^{\leqeta,\tau}(\tld{z})\ar@{^{(}->}[r]& Y^{\leqeta,\tau}\\
\tld{U}(\tld{z},\leqeta,\nabla_{\tau,\infty})\ar[r]\ar@{^{(}->}[u]& Y^{\leqeta,\tau,\nabla_\infty}(\tld{z})\ar@{^{(}->}[r]\ar@{^{(}->}[u]& Y^{\leqeta,\tau,\nabla_\infty}
\ar@{^{(}->}[u]
\\
\tld{\cX}^{\leq \eta,\tau}(\tld{z})
\ar[r]\ar[u]^{\cong}& \cX^{\leq \eta,\tau}(\tld{z})\ar@{^{(}->}[r]\ar[u]^{\cong}&  \cX^{\leq \eta,\tau}
\ar[u]^{\cong}\\
\tld{\cX}^{\eta,\tau}(\tld{z})
\ar[r]\ar@{^{(}->}[u]& \cX^{\eta,\tau}(\tld{z})\ar@{^{(}->}[r]\ar@{^{(}->}[u]&  \cX^{\eta,\tau}
\ar@{^{(}->}[u]}
\]
where:
\begin{itemize}
\item all the stacks appearing in the left column and in the central column 
are \emph{defined} so that all the squares in the diagram are cartesian; 
\item the hooked horizontal arrows are open immersion;
\item the left horizontal arrows are $\un{T}^{\vee}_{\cO}$-torsor for the twisted shifted conjugation action on the source (induced by the twisted shifted conjugation action on  $\tld{U}(\tld{z},\leqeta)^{\wedge_p}$);
\item the vertical hooked arrows are closed immersions and the vertical arrows decorated with ``$\cong$'' are isomorphisms.
\end{itemize}
In particular, $\tld{U}(\tld{z},\leqeta,\nabla_{\tau,\infty})$ is an affine $p$-adic formal scheme over $\Spf\,\cO$, topologically of finite type. 
Furthermore, if $\ell(\tld{z}_j) \geq 2$ for all $j$, then $\tld{U}(\tld{z},\leqeta,\nabla_{\tau,\infty})$ is the $p$-adically completed tensor product over $\cO$ of the rings in Table \ref{Table:integraleqs}. 
\end{thm}
\begin{proof} This is \cite[Proposition 7.2.3]{MLM}. 
The last assertion follows from the computations in \cite[\S 5.3]{LLLM} noting that the whole discussion there applies to the $p$-adic completion (as opposed to completions at closed points), and that the computations of \emph{loc.~cit}.~can be performed with less stringent genericity assumptions (see the proof of Theorem \ref{thm:local_model_cmpt} below for the precise genericity).
\end{proof}
\begin{rmk}
\label{rmk:=eta}
Note that $\cX^{\eta,\tau}\subseteq \cX^{\leqeta,\tau}$ can be characterized as the union of the $(1+3f)$-dimensional irreducible components (which is the maximal possible dimension).
In particular, by letting $\tld{U}(\tld{z},\eta,\nabla_{\tau,\infty})$ denote the maximal reduced $\cO$-flat $(1+6f)$-dimensional closed $p$-adic formal subscheme of $\tld{U}(\tld{z},\leqeta,\nabla_{\tau,\infty})$, Theorem \ref{prop:loc:mod:diag:2} gives an identification of $\tld{\cX}^{\eta,\tau}(\tld{z})$ with $\tld{U}(\tld{z},\eta,\nabla_{\tau,\infty})$.
\end{rmk}

\subsection{Special fibers}
\label{subsub:SpFi}
Let $\Fl$ denote the affine flag variety over $\F$ for ${\GL_3}_{/\F}$ (with respect to the Iwahori relative to the upper triangular Borel),
identified with the special fiber of $\Gr_{\cG,\cO}$.
As in \cite[(4.7)]{MLM}, we define the closed sub-ind scheme $\Fl^{\nabla_0}\into \Fl$.

\subsubsection{Labeling components of $(\Fl^{\nabla_0})^{\cJ}$}
\label{subsub:Label:FL}

\begin{defn}
Let $(\tld{w}_1,\omega)\in \tld{\un{W}}_1^+\times (X^*(\un{T})\cap \un{C}_0+\eta)$ %
and write $S_\F^\circ(\tld{w}_1^*w_0^*)$ for the open affine Schubert cell associated to $\tld{w}_1^*w_0^*\in \tld{\un{W}}^\vee$. 
We define $C_{(\tld{w}_1,\omega)}$ to be the Zariski closure in $(\Fl^{\nabla_0})^{\cJ}$ of $S_\F^\circ(\tld{w}_1^*w_0^*)t_{\omega^*} \cap (\Fl^{\nabla_0})^{\cJ}$.
It is an irreducible subvariety of $(\Fl^{\nabla_0})^{\cJ}$ of dimension $3f$ when $\omega-\eta$ is $2$-deep (see \cite[Proposition 4.3.3]{MLM}). 
It does not depend on the equivalence class of the pair $(\tld{w}_1,\omega)$ defined in \S \ref{subsubsec:SW}.
\end{defn}
\subsubsection{$\un{T}_\F^\vee$-torsors} Replacing the Iwahori with the pro-$v$ Iwahori in the construction of $\Fl$ yields a $\un{T}_\F^\vee$-torsor $\tld{\Fl}^{\cJ}\ra \Fl^{\cJ}$. 
We use $\tld{\cdot}$ to denote the pullback via this $\un{T}_\F^\vee$-torsor of objects introduced so far (e.g.~$\tld{C}_{(\tld{w}_1,\omega)} \subset (\tld{\Fl}^{\nabla_0})^{\cJ}$).
Let $\tld{\Fl}^{[0,2]}\subset \tld{\Fl}$ be the pullback of the special fibre of $\Gr^{[0,2]}_{\cG,\cO}$.

On the other hand, $\tld{\Fl}^{\cJ}$ is also endowed with a $\un{T}_\F^\vee$-action by twisted shifted conjugation induced, at level of matrices, by $(A^{(j)})_{j \in \cJ} \mapsto (t_{j} A^{(j)} t^{-1}_{j-1})_{j \in \cJ}$ for $(t_j)_{j \in \cJ} \in \un{T}_\F^\vee$.

\subsubsection{Labeling by Serre weights}
Recall from \cite[Lemma 2.2.3]{MLM} the bijection $(\tld{w}_1,\omega) \mapsto (F_{(\tld{w}_1,\omega)},t_{\omega-\eta}\tld{w}_1 \un{W}_a/\un{W}_a)$ between lowest alcove presentations $(\tld{w}_1,\omega)$ of $0$-deep Serre weights and $0$-deep Serre weights $\sigma$ with the choice of an algebraic central character $\zeta\in X^*(\un{Z})$ lifting the central character of $\sigma$. 
If $(\tld{w}_1,\omega)$ maps to $(\sigma,\zeta)$ under this bijection, then we set $C^\zeta_{\sigma}\defeq C_{(\tld{w}_1,\omega)}$. 
If the algebraic central character $\zeta\in X^*(\un{Z})$ is understood, we will simply write $C_{\sigma}$. 

\subsubsection{The local model diagram in characteristic $p$}
As explained in \cite[\S 7.4]{MLM} there is a bijection $\sigma\mapsto \cC_\sigma$  between Serre weights and the top dimensional (namely, $3f$-dimensional) irreducible components of $\cX_{K,3}$.
(This is a relabeling of the bijection of \cite[Theorem 6.5.1]{EGstack}.)
The main result of this section describes sufficiently generic $\cC_\sigma$ in terms of the coordinate charts of Theorem \ref{prop:loc:mod:diag:2}.

Recall that $\Phi\text{-}\Mod^{\text{\emph{\'et}},n}_{K, \F}$ denotes the fppf stack over $\F$ whose $R$-points, for a finite type $\F$-algebra $R$, parametrize projective rank $n$ \'etale $(\phz, \cO_{\cE,K}\otimes_{\Fp}R)$-modules {(recall that $\cO_{\cE,K}$ denotes the $p$-adic completion of $(W(k)[\![v]\!])[1/v]$)}.
We have a morphism $(\cX_{K,3})_{\F}\rightarrow \Phi\text{-}\Mod^{\text{\emph{\'et}},n}_{K, \F}$ corresponding to ``restriction to $G_{K_\infty}$'' (cf.~\cite[\S 3.2]{EGstack}).

\begin{thm}
\label{thm:local_model_cmpt}
Let $(s,\mu)$ be a $\zeta$-compatible $4$-generic lowest alcove presentation of a tame inertial type $\tau$. 
Let $\tld{z}\in\Adm^{\vee}(\eta)$ and $\sigma \in \JH(\ovl{\sigma}(\tau))$.

We have a commutative diagram 

\begin{equation}
\label{diag:cmpts}
\begin{tikzcd}[column sep=small]
&&&\tld{C}^{\zeta}_\sigma\ar[hook]{d}
\\
\tld{\fP}_{\sigma,\tld{z}}\ar[hook]{r}\ar[bend left=20,hook,open]{rrru}&\tld{U}(\tld{z},\eta,\nabla_{\tau,\infty})_{\F}
\ar[hook]{r}&\tld{U}(\tld{z},\leq \eta)_{\F}\ar{ddd}\ar{r}&\tld{\Fl}^{[0,2]}_{\cJ}\cdot s^*t_{\mu^*+\eta^*}\ar{ddd}{T^{\vee,\cJ}_{\F}}
\\
\tld{\cC}_\sigma(\tld{z})\ar{u}{\cong}\ar[hook]{r}\ar[swap]{d}{T^{\vee,\cJ}_{\F}}&\tld{\cX}^{\eta,\tau}(\tld{z})_\F\ar{d}{T^{\vee,\cJ}_{\F}}\ar[swap]{u}{\cong}\ar{u}{\text{\tiny{\emph{Rmk~\ref{rmk:=eta}}}}}&&
\\
\cC_\sigma(\tld{z})\ar[hook]{r}\ar[open,hook]{d}&\cX^{\eta,\tau}(\tld{z})_\F\ar[open,hook]{d}&&\\
\cC_\sigma\ar[hook]{r}&\cX^{\eta,\tau}_\F\ar[hook]{r}\ar[hook]{d}&Y^{\leq\eta,\tau}_\F\ar[hook]{r}{\iota_{(s,\mu)}}&\Big[\tld{\Fl}^{[0,2]}_{\cJ}\cdot s^*t_{\mu^*+\eta^*}/T^{\vee, \cJ}_{\F}\text{\emph{-sh.cnj}}\Big]\ar[hook]{d}{\iota_0}\\
&(\cX_{K,3})_{\F}\ar{rr}&&\Phi\text{-}\Mod^{\text{\emph{\'et}},n}_{K, \F}
\end{tikzcd}
\end{equation}
where 
\begin{itemize}
\item
All the squares are cartesian (this defines the previously undefined objects $\cC_\sigma(\tld{z})$, $\tld{\cC}_\sigma(\tld{z})$ and $\tld{\fP}_{\sigma,\tld{z}}$).
\item All the hooked arrows decorated with a circle are open immersions; all the hooked undecorated arrows are monomorphisms and, except $\iota_0$, are moreover closed immersions; all the arrow decorated with $T^{\vee, \cJ}_{\F}$ are $T^{\vee, \cJ}_{\F}$-torsors.
\item $\tld{U}(\tld{z},\leq \eta)_{\F}\ra Y^{\leq\eta,\tau}_{\F}$ is an open immersion followed by a $T^{\vee, \cJ}_{\F}$-torsor map. The map $\tld{U}(\tld{z},\leq \eta)_{\F}\ra \tld{\Fl}^{[0,h]}_{\cJ}\cdot s^*t_{\mu^*+\eta^*}$ is given by the formula 
\[ (A^{(j)})_j \mapsto  (A^{(j)})_js^*t_{\mu^*+\eta^*}\] 
and is a locally closed immersion.
\item
The map $\iota_{(s,\mu)}$ is given, fpqc-locally, by $\fM\mapsto (A^{(j)}_{\fM,\beta})_js^*t_{\mu^*+\eta^*}$, for any choice of eigenbasis $\beta$ for $\fM\in Y^{\leq\eta,\tau}_\F(R)$.
\item
The map $\iota_0$ is defined, fpqc locally, by sending the class of a tuple $(A^{(j)})_{j\in\cJ}$ to the free rank $n$ \'etale $\phz$-module with Frobenius given by $(A^{(j)})_{j\in\cJ}$ in the standard basis.
\item
If $\sigma = F({\mathfrak{Tr}_{\mu+2\eta}(s(\eps),a)})$, then the closed immersion $\tld{\fP}_{\sigma,\tld{z}}\into \tld{U}(\tld{z},\eta,\nabla_{\tau,\infty})_{\F}$ corresponds to the ideal $\sum_j \tld{\fP}_{(\eps_j,a_j),\tld{z}_j}\cO(\tld{U}(\tld{z},\eta,\nabla_{\tau,\infty})_{\F})$ with $\tld{\fP}_{(\eps_j,a_j),\tld{z}_j}$ in Table \ref{Table:intsct} (or the unit ideal if $\tld{\fP}_{(\eps_j,a_j),\tld{z}_j}$, up to symmetry, does not appear in Table \ref{Table:intsct} for some $j$); in particular $\tld{U}(\tld{z},\eta,\nabla_{\tau,\infty})_{\F}$ is reduced.
\item
The bottom horizontal map identifies  $\cX^{\eta,\tau}_{\F}$ with the reduced union of $\Big[\tld{C}^\zeta_{\sigma'}/T_\F^{\vee,\cJ}\text{\emph{-sh.cnj}}\Big]$ for $\sigma'\in \JH(\ovl{\sigma}(\tau))$. %

\end{itemize}
\end{thm}
\begin{proof}
 Theorem \ref{prop:loc:mod:diag:2} (together with Remark \ref{rmk:=eta}) and \cite[Proposition 5.4.7, Theorem 7.4.2]{MLM},
imply the existence of the portion of diagram (\ref{diag:cmpts})  which excludes the leftmost vertical column, the top triangle, and the identification of $\tld{U}(\tld{z},\eta,\nabla_{\tau,\infty})_\F$ with entries of Table \ref{Table:intsct}.
(In the notation of \cite[Proposition 5.4.7]{MLM} the monomorphism $\iota_0$ would be denoted as $\iota_{s^*t_{\mu^*+\eta^*}}$, the morphism $\iota_{(s,\mu)}$ would be the diagonal arrow.) 
Furthermore, all stated properties of this portion of the diagram are already known to hold, except possibly for the last item. We now explain how to fill in the missing parts with all the desired properties except for the last item.
\begin{enumerate}
\item
\label{it:pf:diag:0}
We first deal with the case $\ell(\tld{z}_j)\geq 2$ for all $0\leq j\leq f-1$.
In this situation, the computations in \cite[\S 5.3.1]{LLLM} show that $\tld{U}(\tld{z},\eta,\nabla_{\tau,\infty})_\F\into \tld{U}(\tld{z},\leq \eta)_{\F}$ identifies with the scheme given by Table \ref{Table:intsct}. Indeed, we note that:
\begin{itemize}
\item The computations in \cite[\S 5.3.1]{LLLM} of various completions of $\tld{U}(\tld{z},\eta,\nabla_{\tau,\infty})_\F$ are in fact valid for $\tld{U}(\tld{z},\eta,\nabla_{\tau,\infty})_\F$.
\item The computations are performed with an unnecessary strong genericity condition: indeed, by using the ``(1,3)-entry'' of the leading term in the monodromy condition, one recovers the last displayed equation at page 59 of \emph{loc.~cit}.~with $n-3$ replaced by $n-1$.
\end{itemize}
Choosing the closed subscheme $\tld{\mathfrak{P}}_{\sigma,\tld{z}}$ according to Table \ref{Table:intsct}, we have now constructed the top horizontal arrow of diagram \ref{diag:cmpts}. This uniquely induces the leftmost vertical column of the diagram for \emph{some} choice of irreducible component $\cC_{\kappa}$ of $\cX^{\eta,\tau}$. We need to show that:

\begin{enumerate}
\item 
\label{it:pf:diag:1}
The composite of the top horizontal arrows identifies $\tld{\fP}_{\sigma,\tld{z}}$ with an non-empty open subscheme of $\tld{C}_{\sigma}^{\zeta}$; and
\item
\label{it:pf:diag:2}
$\cC_\kappa=\cC_\sigma$.
\end{enumerate}
For item {(a)}, we use \cite[Proposition 4.3.4, Proposition 4.3.5]{MLM}: one checks that according to Table \ref{Table:intsct}, the image of $\tld{\fP}_{\sigma,\tld{z}}$ along the top horizontal arrows is, in the notation of \emph{loc.~cit}., a $T_{\F}^{\vee,\cJ}$-torsor over 
\[
S_{\F}^{\nabla_0}(\tld{w}_1,\tld{w}_2,\tld{s})
\]
for suitable choices of $\tld{w}_1,\tld{w}_2$, and for $\tld{s}$ taken to be $t_{\mu+\eta}s$, and this is exactly an open subscheme of $\tld{C}_{\sigma}^\zeta$. 

For item {(b)}, we use the just established item {(a)}, and then use the same argument as in the third paragraph in the proof of \cite[Theorem 7.4.2]{MLM} to recognize that $\cC_{\kappa}$ is actually $\cC_{\sigma}$.
\item
\label{it:pf:diag:3}
We deal with the general case.
The computations in \cite[\S 8]{LLLM}, namely Propositions 8.3, 8.11 and 8.13, as well as those 
of\footnote{The computations of \cite[\S 5.3.2, 5.3.3]{LLLM} are also performed with unnecessary strong genericity conditions: again using the ``(1,3)-entries'' of the leading term in the monodromy condition recovers, in \cite[\S 5.3.2, 5.3.3]{LLLM} respectively, the equations
\begin{align*}
&c_{11}((a-b)c_{32}c^*_{23}-(a-c)c'_{22}c'_{33})-p(e-a+c)c^*_{12}c^*_{23}c^*_{31}+O(p^{n-1})\\
&c_{12}((a-b)c_{31}c'_{23}+(b-c)c^*_{21}c'_{33})-p(e-a+c)c^*_{21}c^*_{32}c^*_{13}+O(p^{n-1})
\end{align*}
(cf.~the last displayed equations at page 60 and 61 of \cite{LLLM})
obtaining the ``monodromy equations'' claimed in \emph{loc.~cit}.~as soon as $\mu$ is $4$-deep in alcove $\un{C}_0$.} \emph{loc.~cit}.~\S 5.3.2, 5.3.3, show that the closed embedding of $\tld{U}(\tld{z},\eta,\nabla_{\tau,\infty})_{\F}\into \tld{U}(\tld{z},\leq \eta)_{\F}$ factors through $\tld{U}(\tld{z},\eta,\nabla_{\tau,\infty})_{\mathrm{table},\F}$, where we temporarily write $\tld{U}(\tld{z},\eta,\nabla_{\tau,\infty})_{\mathrm{table},\F}$ for the scheme defined in the second column of Table \ref{Table:intsct}.

We have to prove that this closed immersion is actually an isomorphism. 
Let $n_{\tld{z}}$ be $\# \big(W^?(\taubar(sz^*,\mu+s(\nu^*)+\eta))\cap \mathrm{JH}(\ovl{R_s(\mu+\eta)})\big)$.
Note that the arguments of \cite[\S 8]{LLLM} show that $\tld{U}(\tld{z},\eta,\nabla_{\tau,\infty})_{\mathrm{table},\F}$ is \emph{reduced}, and that its number of irreducible component is $n_{\tld{z}}$.
We will show that there are at least $n_{\tld{z}}$ irreducible components of $\cX^{\eta,\tau}_\F$ which intersect the open substack $\cX^{\eta,\tau}(\tld{z})_\F$.

Now, from the previously established cases of diagram (\ref{diag:cmpts}), we see that $\cX^{\eta,\tau}$ must contain all the $\cC_{\sigma'}$ which occurs in diagram \ref{diag:cmpts} for $\tld{z}'$ such that $\ell(\tld{z}'_j)\geq 3$ for all $0\leq j \leq f-1$. In particular, $\cX^{\eta,\tau}_\F$ contains all $\cC_{\sigma'}$ such that $\sigma' \in \mathrm{JH}(\ovl{R_s(\mu+\eta)})$.  %
Note that by definition, $\cX^{\eta,\tau}(\tld{z})_{\F}=\cX^{\eta,\tau}_\F\cap \Big[\tld{U}(\tld{z},\leq \eta)_\F s^*t_{\mu^*+\eta^*}/T^{\vee,\cJ}_\F\text{-sh.cnj}\Big]$. We are thus reduced to showing that there are at least $n_{\tld{z}}$ choices of $\sigma'$ as above such that $\cC_{\sigma'} \cap \Big[\tld{U}(\tld{z},\leq \eta)_\F s^*t_{\mu^*+\eta^*}/T^{\vee,\cJ}_\F\text{-sh.cnj}\Big]\neq \emptyset$. But this last condition is equivalent to $\tld{C}^\zeta_{\sigma'}\cap \tld{U}(\tld{z},\leq \eta)_\F s^*t_{\mu^*+\eta^*}\neq \emptyset$, and in turn equivalent to $\tld{z}s^*t_{\mu^*+\eta^*}\in \tld{C}^{\zeta}_{\sigma'}$. To summarize, we need to show the combinatorial statement that the number of $\tld{C}^{\zeta}_{\sigma'}$ which contain $\tld{z}s^*t_{\mu^*+\eta^*}$ is exactly $n_{\tld{z}}$.
But this is the same combinatorial statement as \cite[Theorem 4.7.6]{MLM}, and we observe that the conclusion of that Theorem holds in our current setup: this follows from the invariance property \cite[Proposition 4.3.5]{MLM} of $\tld{C}^{\zeta}_{\sigma'}$, as well as the fact that $\tld{z}'s^*t_{\mu^*+\eta^*}\in \tld{C}^{\zeta}_{\sigma'}$ whenever $\tld{C}^{\zeta}_{\sigma'}$ occurs in diagram (\ref{diag:cmpts}) for $\tld{z}'$ such that $\ell(\tld{z}'_j)\geq 3$ for $0\leq j \leq f-1$.

At this point, we have shown that $\tld{U}(\tld{z},\eta,\nabla_{\tau,\infty})_{\F}$ identifies with $\tld{U}(\tld{z},\eta,\nabla_{\tau,\infty})_{\mathrm{table},\F}$, and thus we establish the top horizontal arrow of diagram (\ref{diag:cmpts}). The rest of the proof now is exactly the same as in the previous case.

\end{enumerate}
Finally, it remains to check the last item in the theorem. But the reducedness follow from the reducedness of each $\tld{U}(\tld{z},\eta,\nabla_{\tau,\infty})_{\F}$, while the identification of the irreducible components was already established in the arguments above.
\end{proof}

\begin{cor}
\label{thm:local_model_main} 
Let $\tld{z}\in\Adm^{\vee}(\eta)$ and assume that the character $\mu$ (appearing in the lowest alcove presentation $(s,\mu)$ of $\tau$) is $4$-deep. 
Then:
\begin{itemize}
\item $\cX^{\eta,\tau}_{\bF}$ is reduced; 
\item $\cO\big(\tld{U}(\tld{z},\eta,\nabla_{\tau,\infty})\big)$ is a normal domain; and
\item
for any $\rhobar:G_K\ra\GL_3(\F)$ the ring $R^{\eta,\tau}_{\rhobar}$ is either $0$ or a normal domain. 
\end{itemize}
\end{cor}

\begin{rmk}
\label{rmk:CM}
It can be showed that both $\cO\big(\tld{U}(\tld{z},\eta,\nabla_{\tau,\infty})\big)$ and $R^{\eta,\tau}_{\rhobar}$ are Cohen--Macaulay. This can be done by either explicit inspection of the schemes occurring in Table \ref{Table:intsct} as in \cite[\S 8]{LLLM}, or by using the cyclicity of patched modules proven in Theorem \ref{thm:cyclic} below. 
\end{rmk}

\begin{proof}
 The fact that $\cO\big(\tld{U}(\tld{z},\eta,\nabla_{\tau,\infty})\big)$ is a normal domain follows from the fact that its special fiber is reduced, as in the proof of \cite[Proposition 8.5]{LLLM}.
The statement for $R^{\eta,\tau}_{\rhobar}$ follows in the same way, noting that $R^{\eta,\tau}_{\rhobar,\F}$, being (equisingular to) a completion of the excellent reduced ring $\cO\big(\tld{U}(\tld{z},\eta,\nabla_{\tau,\infty})_\F \big)$, is reduced.
\end{proof}

We can finally introduce the notion of Serre weights attached to a continuous Galois representation $\rhobar:G_K\ra\GL_3(\F)$.

\begin{defn}
\label{def:gen:W}
Let $\rhobar:G_K\ra\GL_3(\F)$ be a continuous Galois representation.
We define $W^g(\rhobar)$ to be the set of Serre weights $\sigma$ such that $\rhobar\in \cC_\sigma(\F)$ (cf.~\cite[Definition 9.1.2]{MLM}).

If $\rhobar|_{I_K}$ is \emph{tame} so that $\rhobar|_{I_K}$ is isomorphic to a tame inertial $\F$-type $\tau(w,\mu+\eta)$, then we define $W^?(\rhobar)$ to be $W^?(\taubar(w,\mu+\eta))$ (cf.~subsubsection {\ref{sec:reduction}} for the latter).   We also say that $\rhobar$ is \emph{$N$-generic} if $\tau(w,\mu+\eta)$ is.

Finally, we say that a Serre weight is \emph{generic} if it is a Jordan--H\"older constituent of a $4$-generic Deligne--Lusztig representation. 
(By equation (\ref{eq:SW:LAP}) and \S \ref{sec:reduction}, a Serre weight is generic if and only if it admits a lowest alcove presentation $(\tld{w}_1,\omega)$ such that $(\omega,\tld{w}_1\cdot \un{C}_0)\in t_{\mu+\eta}s(\Sigma)$ for some $\mu\in \un{C}_0\cap X^*(\un{T})$ which is $4$-deep.)
A generic Serre weight is necessarily $2$-deep by \cite[Proposition 2.3.7]{MLM}. 
We let $W^g_{\mathrm{gen}}(\rhobar)$ and $W^?_{\mathrm{gen}}(\rhobar)$ denote the subsets of generic Serre weights of $W^g(\rhobar)$ and $W^?(\rhobar)$, respectively. 
\end{defn}
\begin{rmk}
When restricted to the present setting, the subset $W^g_{\mathrm{gen}}(\rhobar)\subseteq W^g(\rhobar)$ defined in \cite[Definition 9.1.2]{MLM} (consisting of $8$-deep Serre weights) is a subset of the set defined above. 
\end{rmk}
\begin{cor}\label{cor:sspts}
Let $\rhobar: G_K \ra \GL_3(\F)$ be semisimple and $4$-generic.
Then $W^g_{\mathrm{gen}}(\rhobar) = W_{\mathrm{gen}}^?(\rhobar)$.
\end{cor}
\begin{proof}
Let $\sigma$ be a generic Serre weight so that $\sigma\in \JH\big(\ovl{R_{s}(\mu+\eta)}\big)$ for some $4$-deep $\mu\in\un{C}_0$ and $s\in \un{W}$. 
By Theorem \ref{thm:local_model_cmpt}, if $\rhobar\notin \cX^{\eta,\tau}_{\F}$ then $\sigma\notin W^g(\rhobar)$. 
Else $\rhobar$, being semisimple, corresponds to a point in $T_{\F}^{\vee,\cJ}\tld{z}\in\tld{U}(\tld{z},\leq\eta)_{\F}$ in the diagram \eqref{diag:cmpts}.
The proof of Theorem \ref{thm:local_model_cmpt}~(precisely, the end of the third paragraph in the proof of item (\ref{it:pf:diag:3}) of \emph{loc.~cit}.) shows that $\rhobar\in \tld{\cC}_\sigma$ if and only if $\sigma\in W^?(\rhobar)$. 
\end{proof}

\begin{prop}\label{prop:gWE}
If $\rhobar: G_K \ra \GL_3(\F)$ is a Galois representation such that $\rhobar^{\semis}|_{I_K}$ is $6$-generic, then every Serre weight in $W^g(\rhobar)$ is generic i.e.~$W^g_{\mathrm{gen}}(\rhobar)= W^g(\rhobar)$.
\end{prop}
\begin{proof}
Since the proof uses methods which are now well-known (see e.g.~\cite[\S 3]{GHLS}) but orthogonal to those of this section, we will be brief. 
Let $\rhobar$ be as in the statement of the proposition. 
Then in particular, $p > 6$. 
Suppose that $F(\lambda) \in W^g(\rhobar)$. 
Let $\rhobar': G_K \ra \GL_3(\F)$ be a maximally nonsplit Galois representation lying {only} on $\cC_{F(\lambda)}$. 
Then $\rhobar'$ has an ordinary crystalline lift of weight $\lambda+\eta$ by \cite[Lemma 5.5.4]{EGstack}. 
Then by \cite[Corollary A.7]{EG} there is an automorphic globalization $\rbar': G_{F^+} \ra \cG_3(\F)$ (where $\cG_3$ is the algebraic group defined in \cite[\S 2.1]{CHT} with $n=3$) of $\rhobar'$ which is potentially diagonalizably automorphic in the sense of \cite[Theorem 4.3.1]{LLL}. 
The proof of \cite[Theorem 4.3.8]{LLL} implies that $\Hom_{G(\cO_{F^+_p})}(\otimes_{v|p} \ovl{W(\lambda)}^\vee,S(U))^{\mathrm{ord}} \neq 0$ in the notation of \emph{loc.~cit}. 
Since the natural map $\Hom_{G(\cO_{F^+_p})}(\otimes_{v|p} F(\lambda)^\vee,S(U))^{\mathrm{ord}} \ra \Hom_{G(\cO_{F^+_p})}(\otimes_{v|p} \ovl{W(\lambda)}^\vee,S(U))^{\mathrm{ord}}$ is an isomorphism by \cite[Lemma 6.1.3]{gee-geraghty}, we conclude that $\Hom_{G(\cO_{F^+_p})}(\otimes_{v|p} F(\lambda)^\vee,S(U))$ is nonzero. 
This implies that $\Hom_{G(\cO_{F^+_p})}(\otimes_{v|p} \ovl{R}_1(\lambda)^\vee,S(U)) \neq 0$ by \cite[Lemma 2.3]{florian-inv}. 
In particular, $\rhobar'$ has a potentially semistable lift of type $(\eta,\tau)$ for $\tau = \tau(1,\lambda)$. 
Since $W^g(\rhobar') = \{F(\lambda)\}$, we conclude that $\cC_{F(\lambda)}$ is a subset of the substack of $\cX_{K,3}$ corresponding to potentially semistable representations of type $(\eta,\tau)$. 
In particular, since $\rhobar \in \cC_{F(\lambda)}(\F)$, $\rhobar$ has a potentially semistable lift of type $(\eta,\tau)$. 
By \cite[Lemma 5]{Enns}, $\rhobar^\semis$ has a semistable lift of type $(\eta,\tau)$. 
Then $R_1(\lambda)$ is $1$-generic by \cite[Proposition 7]{Enns} (the proof of \cite[Theorem 8]{Enns} shows that $R_1(\lambda)$ is $2$-generic in the sense of \cite[Definition 2]{Enns} where $\delta = 6$ and $n+1 = 4$ here so that $R_1(\lambda)$ is $1$-generic by \cite[Remark 2.2.8]{LLL}). 
In particular, any potentially semistable lift of type $(\eta,\tau)$ is potentially crystalline. 
By the proof of \cite[Proposition 3.3.2]{LLL}, $\tau$ is $4$-generic (the proof shows that \emph{any} lowest alcove presentation of $\tau$ is $3$-generic, but one compatible with a $6$-generic lowest alcove presentation of $\rhobar^{\semis}_{I_K}$ is $4$-generic). 
We conclude that $R_1(\lambda)$ is $4$-generic and that $F(\lambda)$ is generic. 
\end{proof}

\begin{cor}\label{cor:g=?}
Let $\rhobar: G_K \ra \GL_3(\F)$ be semisimple and $6$-generic.
Then $W^g(\rhobar) = W^?(\rhobar)$.
\end{cor}
\begin{proof}
This follows from Corollary \ref{cor:sspts}, Proposition \ref{prop:gWE}, and the fact that $W^?_{\mathrm{gen}}(\rhobar)= W^?(\rhobar)$ in this case. 
\end{proof}

\begin{rmk}
\label{rmk:products:1}
The notions and the results of this section hold true, \emph{mutatis mutandis}, when the set $S_p$ has arbitrary finite cardinality, and $\tau$, $\rhobar$ are a tame inertial $L$-parameter and a continuous $\F$-valued $L$-homomorphism respectively.
In this case $\tld{U}(\tld{z})$, $\tld{U}^{[0,2]}(\tld{z})$, etc.~are fibered products, over $\Spec\cO$ and over the elements $v\in S_p$, of objects of the form $\tld{U}(\tld{z}_v)$, $\tld{U}^{[0,2]}(\tld{z})$, etc.~for $\tld{z}_v\in (\tld{W}^{\vee})^{\Hom_{\Qp}(F^+_v,E)}$; the twisted shifted $\un{T}_{\cO}^\vee$-conjugation action is now induced by $A^{(j)}\mapsto t_j A^{(j)} s^{-1}_j(t_{\pi^{-1}(j)})$ for $j\in \Hom_{\Qp}(F_p,E)$.
Analogously, the algebraic stacks $Y^{[0,2],\tau}$, $\cX_{K,3}$, $\cX^{\eta,\tau}$ etc.~are fibered products, over $\Spf \, \cO$ and over the elements $v\in S_p$, of $Y^{[0,2],\tau_v}$, $\cX_{F_v^+,3}$, $\cX^{\eta_v,\tau_v}$ etc.
The results of this section hold true in this more general setting.
\end{rmk}

\section{Geometric Serre weights}
\label{sec:GSW}
The irreducible components of $\cX_{K,n}$ from \cite[Definition 3.2.1]{EGstack} give rise to a partition of $\cX_{K,n}$ with locally closed parts \[\underset{\sigma \in W^+}{\bigcap} \cC_\sigma \setminus \underset{\sigma \notin W^+}{\bigcup} \cC_\sigma\] indexed by sets $W^+$ of Serre weights. 
It is of interest to determine the geometric properties of these pieces e.g.~when they are nonempty. 
In principle, one can directly study 
\begin{equation}
\label{eq:parts}
\underset{\sigma \in W^+}{\bigcap} \cC_\sigma \setminus \underset{\substack{\sigma \notin W^+\\\sigma\text{ generic }}}{\bigcup} \cC_\sigma
\end{equation}
for a set $W^+$ of \emph{generic} Serre weights using the relationship between $\cX_{K,n}$ and $\Fl^{\nabla_0}$, but this seems to be complicated even when $n=4$. 
In this section, we determine when \eqref{eq:parts} is nonempty in generic cases when $n=3$ using a notion of \emph{obvious weights} for wildly ramified representations. 

\subsection{Intersections of generic irreducible components in $\Fl^{\nabla_0}$}
\label{subsec:geometricFL}

We first study the geometry of $\Fl^{\nabla_0}$. 
The set $\cJ$ will be a singleton, and so we will omit it from the notation. 
For $n\in \N$, let $\cC_{n\text{-deep}}$ be the set of $\omega\in X^*(T)$ such that $\omega-\eta$ is $n$-deep in $C_0$.
Recall from \S \ref{subsub:Label:FL} that given $(\tld{w},\omega) \in \tld{W}_1 \times \cC_{2\text{-deep}}$, we have the irreducible subvariety $C_{(\tld{w},\omega)}$ of $\Fl^{\nabla_0}$.
We define $\Fl_{2\textrm{-deep}}^{\nabla_0}$ as the union of the $C_{(\tld{w},\omega)}$ with $(\tld{w},\omega)\in \tld{W}_1 \times \cC_{2\text{-deep}}$ (in particular, these $C_{(\tld{w},\omega)}$ are its irreducible components). 
The action of $T_{\F}^\vee$ (resp.~$\bG_m$) on $\Fl^{\nabla_0}$ induced by right multiplication (resp.~loop rotation $t\cdot v = t^{-1}v$) preserves $\Fl_{2\textrm{-deep}}^{\nabla_0}$ and its irreducible components. 
We let $\tld{T}_{\F}^\vee$ be the extended torus $T_{\F}^\vee \times\bG_m$. 
We write $\Fl_{2\textrm{-deep}}^{\nabla_0,T^\vee}$ for the set of $T_{\F}^\vee$-fixed points (or equivalently $\tld{T}_{\F}^\vee$-fixed points) of $\Fl_{2\textrm{-deep}}^{\nabla_0}$ under right translation.

\begin{defn}
For $x^* \in \Fl_{2\textrm{-deep}}^{\nabla_0}(\F)$, let $W_{2\textrm{-deep}}^g(x^*)$ be the set 
\[
\{(\tld{w},\omega) \in (\tld{W}_1 \times \cC_{2\text{-deep}}) \mid x^* \in C_{(\tld{w},\omega)}(\F)\}/\sim.
\] 
Recall that the equivalence relation is given by $(\tld{w}, \omega) \sim (t_{\nu} \tld{w}, \omega - \nu)$ for any $\nu \in X^0(T)$.    
\end{defn}

The main result of this section classifies the sets $W^g_{2\textrm{-deep}}(x^*)$ for $x^*\in \Fl_{2\textrm{-deep}}^{\nabla_0}(\F)$. 
Combining the proof of Theorem \ref{thm:local_model_cmpt} (namely, the combinatorial statement in the proof of item (\ref{it:pf:diag:3}) of \emph{loc.~cit}.) and Corollary \ref{cor:sspts} (see also \cite[Prop.~2.6.2]{MLM}), we obtain the following description of $W^g_{2\textrm{-deep}}(x^*)$ for $x^* \in \Fl_{2\textrm{-deep}}^{\nabla_0,T^\vee}(\F)$

\begin{thm}\label{thm:Tfixed}
For $x^* \in \Fl_{2\textrm{-deep}}^{\nabla_0,T^\vee}$, $W^g_{2\textrm{-deep}}(x^*)$ is the set
\[
\{(\tld{w},x\tld{w}_2^{-1}(0)) \in \tld{W}_1^+ \times \cC_{2\text{-deep}} \mid \tld{w}_2 \uparrow \tld{w}\}/\sim.
\]
\end{thm}

A first step towards understanding $W^g_{2\textrm{-deep}}(x^*)$ will be the determination of the subset $W_{\mathrm{obv}}(x^*) \subset W^g_{2\textrm{-deep}}(x^*)$ of \emph{obvious weights}. It is defined as follows (cf.~\cite{OBW}).

\begin{defn}\label{defn:OW}
\begin{enumerate}
\item
Let $y\in \tld{W}$ and $\tld{w}\in \tld{W}_1$ be such that $y\tld{w}^{-1}(0)\in\cC_{2\text{-deep}}$. 
We define $C_{(\tld{w},y\tld{w}^{-1}(0))}(y^*)$ to be the intersection $C_{(\tld{w},y\tld{w}^{-1}(0))}\cap T_\F^\vee\backslash \tld{U}(y^*)_{\F}$ in $\Fl_{2\textrm{-deep}}^{\nabla_0}$.
\item
Let $x^* \in \Fl_{2\textrm{-deep}}^{\nabla_0}(\F)$. %
Then define $SP(x^*)$ to be the set 
\[
\{(y,(\tld{w},y\tld{w}^{-1}(0))) \in \tld{W} \times (\tld{W}_1 \times \cC_{2\text{-deep}})/\sim\, \mid x^* \in C_{(\tld{w},y\tld{w}^{-1}(0))}(y^*)(\F)\}
\]
and $S(x^*)$ and $W_{\obv}(x^*)$ be the images of $SP(x^*)$ under the projections of $\tld{W} \times (\tld{W}_1 \times \cC_{2\text{-deep}})/\sim$ to $\tld{W}$ and $(\tld{W}_1 \times \cC_{2\text{-deep}})/\sim$, respectively. 
\end{enumerate}
\end{defn}

We call elements in $S(x^*)$ specializations of $x^*$.  The set $SP(x^*)$ is the set of specialization pairs consisting of a specialization and an obvious weight.

\begin{lemma}\label{lemma:obvgeom}
For $x^* \in \Fl_{2\textrm{-deep}}^{\nabla_0,T^\vee}(\F)$ and $\tld{w}\in \tld{W}_1$ such that $x\tld{w}^{-1}(0)\in \cC_{2\text{-deep}}$, $x^* \in C_{(\tld{w},x\tld{w}^{-1}(0))}$.
\end{lemma}
\begin{proof}
For any $\tld{w} \in \tld{W}_1$, since $x^* \in S^\circ_{\F}(\tld{w}^* w_0^*) (x \tld{w}^{-1} w_0)^*$, $x^* \in S_\F^{\nabla_0}(\tld{w},e,x\tld{w}^{-1}w_0) = C_{(\tld{w},x\tld{w}^{-1}(0))}$ (see \cite[(4.9) and Prop.~4.3.5]{MLM}).
\end{proof}

\begin{rmk}\label{rmk:obv}
\begin{enumerate}
\item 
Let $x^* \in \Fl_{2\textrm{-deep}}^{\nabla_0,T^\vee}$. Since $x^* $ is the unique $T^{\vee}_\F$-fixed point of $C_{(\tld{w},x\tld{w}^{-1}(0))}(x^*)$,  $SP(x^*) = \{(x,(\tld{w},x\tld{w}^{-1}(0))) \in \tld{W} \times (\tld{W}_1 \times \cC_{2\text{-deep}})/\sim\}$ by  Lemma \ref{lemma:obvgeom}. In particular, $S(x^*) = \{x\}$. Moreover, for all tame inertial $\F$-types $\taubar$, $(\tld{w},\omega) \in W_{\obv}(\tld{w}(\taubar)^*)$ if and only if $F_{(\tld{w},\omega)} \in W_{\obv}(\taubar)$ in the sense of \cite[Definition 2.6.3]{MLM}. 
\item \label{item:obvLB} For $x^* \in \Fl_{2\textrm{-deep}}^{\nabla_0}(\F)$, clearly, $W_{\obv}(x^*) \subset W^g_{2\textrm{-deep}}(x^*)$.
\end{enumerate}
\end{rmk}

To determine $W^g_{2\textrm{-deep}}(x^*)$, we first determine $SP(x^*)$. 
The idea is that $W_{\obv}(x^*)$ gives a lower bound for $W^g_{2\textrm{-deep}}(x^*)$ (Remark \ref{rmk:obv}\eqref{item:obvLB}) while $S(x^*)$ gives an upper bound by Lemma \ref{lemma:semicont}\eqref{item:SWineq} and Theorem \ref{thm:Tfixed}, and $SP(x^*)$ combines these invariants into a more uniformly behaved set (see Corollary \ref{cor:surjtheta}). 

The following results are key to our analysis of $SP(x^*)$. 
For $x^* \in \Fl_{2\textrm{-deep}}^{\nabla_0}$, let $\theta_{x^*}: SP(x^*) \ra W$ be the map that takes $(y,(\tld{w},\omega))$ to the image of $y \tld{w}^{-1}$ in $W$.

\begin{prop}\label{prop:injtheta}
For $x^* \in \Fl_{2\textrm{-deep}}^{\nabla_0}$, $\theta_{x^*}: SP(x^*) \ra W$ is injective.
\end{prop}
\begin{proof}
This is \cite[Prop.~3.6.4]{OBW}.
(It can also be proven by direct computation in the case of $\GL_3$.)
\end{proof}

Let $\cI\defeq L^+\cG_{\F}$ be the Iwahori group scheme.
For $\tld{w}_\tau \in \tld{W}$ and $x^* \in \Fl$, let $\tld{w}(x^*,\tld{w}_\tau) \in \tld{W}$ be the unique element such that $x^* \in \cI\backslash \cI \tld{w}(x^*,\tld{w}_\tau)^* \cI \tld{w}_\tau^*$.

\begin{lemma} \label{lemma:semicont}
Suppose that $y \in S(x^*)$. 
Then 
\begin{enumerate}
\item \label{item:semicont} $\tld{w}(y^*,\tld{w}_\tau) \leq \tld{w}(x^*,\tld{w}_\tau)$; and
\item \label{item:SWineq} $W^g_{2\textrm{-deep}}(x^*) \subset W^g_{2\textrm{-deep}}(y^*)$. 
\end{enumerate}
\end{lemma}
\begin{proof}

That $y \in S(x^*)$ implies that $x^* \in T_{\F}^{\vee}\backslash\tld{U}(y^*)$ or equivalently that $y^*$ is in the $\tld{T}_{\F}^{\vee}$-orbit closure of $x^*$. 
For \eqref{item:semicont}, $y^*$ is in the ($\tld{T}_{\F}^{\vee}$-orbit) closure of $\cI\backslash \cI \tld{w}(x^*,\tld{w}_\tau)^* \cI \tld{w}_\tau^*$ which implies the desired inequality. 
For \eqref{item:SWineq}, if $x^* \in C_{(\tld{w},\omega)}$, then $y^* \in C_{(\tld{w},\omega)}$ since $C_{(\tld{w},\omega)}$ is $\tld{T}_{\F}^{\vee}$-stable and closed. 
\end{proof}

The following result provides a method to start with an element of $SP(x^*)$ and produce another using a simple reflection in $W$.

\begin{prop}\label{prop:walk}
Let $x^*\in \Fl_{2\textrm{-deep}}^{\nabla_0}$ and $s \in W$ be a simple reflection. 
Suppose that for some $y\in \tld{W}$, 
\begin{enumerate}
\item \label{item:2-deep} $y \tld{w}^{-1}s t_{-\eta} \tld{w}(0)-\eta$, $y \tld{sw}^{-1}(0)-\eta$, and $y t_{-(sw)^{-1}(\eta)}(0)-\eta$ are $2$-deep; and 
\item $(y,(\tld{w},y\tld{w}^{-1}(0))) \in SP(x^*)$, 
\end{enumerate}
where $\tld{sw} \in \tld{W}_1$ is the unique element up to $X^0(T)$ such that $\tld{sw}\tld{w}^{-1}s^{-1} \in X^*(T)$. 
Then either $(y\tld{w}^{-1}s\tld{w},(\tld{w},y\tld{w}^{-1}(0))) \in SP(x^*)$ or $(y,(\tld{sw},y\tld{sw}^{-1}(0)))$ is in $SP(x^*)$. 
\end{prop}
\begin{proof}
Let $\tld{w}_\tau$ be $y (\tld{w}^{-1}\tld{w}_h^{-1} w_0 s \tld{w})^{-1}$ so that $\tld{w}(y^*,\tld{w}_\tau) = \tld{w}^{-1}\tld{w}_h^{-1} w_0 s \tld{w}$. 
In Galois-theoretic language, this corresponds to the choice of the inertial type $\tau$ in \cite[Proposition 7.16(3)]{LLLM}. 
We will see that there are only two possibilities for $\tld{w}(x^*,\tld{w}_\tau)$.
First, $\tld{w}^{-1}\tld{w}_h^{-1} w_0 s \tld{w} \leq \tld{w}(x^*,\tld{w}_\tau)$ by Lemma \ref{lemma:semicont}\eqref{item:semicont}. 

Let $M(\leq\eta)_{\F} \subset \Fl$ be the reduced closure of $\cup_{w\in W} \cI\backslash \cI t_{w^{-1}(\eta)} \cI$ (this is compatible with the notation in \S \ref{sec:localmodel}). 
Let $M(\eta,\nabla_{\tld{w}_\tau(0)})_{\F}\tld{w}_\tau^*$ be the intersection $M(\eta)_{\F}\tld{w}_\tau^* \cap \Fl^{\nabla_0}$. 
For $\tld{z} \in \tld{W}^\vee$, let $M(\eta,\nabla_{\tld{w}_\tau(0)})_{\F}(\tld{z})$ denote the intersection $M(\eta,\nabla_{\tld{w}_\tau(0)})_{\F} \cap \big(T_{\F}^{\vee}\backslash\tld{U}(\tld{z})\big)$ which is isomorphic to $U(\tld{z},\eta,\nabla_{\tld{w}_\tau(0)})_{\F}\defeq T^{\vee}_{\F}\backslash\tld{U}(\tld{z},\eta,\nabla_{\tld{w}_\tau(0)})_{\F}$. 
Now, $x^*$ lies in $C_{(\tld{w},y\tld{w}^{-1}(0))}$ which is the closure of $$M(\eta,\nabla_{\tld{w}_\tau(0)})_\F(t_{w^{-1}(\eta)}^*)\tld{w}_\tau^*=\cI\backslash\cI t_{w^{-1}(\eta)} \cI\tld{w}^*_{\tau}\cap \Fl^{\nabla_0},$$ hence $\tld{w}(x^*,\tld{w}_\tau) \leq t_{w^{-1}(\eta)}$. 
Combining this with the last paragraph, we have $\tld{w}^{-1}\tld{w}_h^{-1} w_0 s \tld{w} \leq \tld{w}(x^*,\tld{w}_\tau)\leq t_{w^{-1}(\eta)}$. 
Since $\ell(\tld{w}^{-1}\tld{w}_h^{-1} w_0 s \tld{w}) = 3 = \ell(t_{w^{-1}(\eta)}) - 1$ (this is a consequence of a more general result in \cite{OBW}, but can be checked directly using \cite[Table 1]{LLLM}), we see that $\tld{w}(x^*,\tld{w}_\tau) = t_{w^{-1}(\eta)}$ or $\tld{w}^{-1}\tld{w}_h^{-1} w_0 s \tld{w}$. 

If $\tld{w}(x^*,\tld{w}_\tau) = t_{w^{-1}(\eta)}$, then $(y\tld{w}^{-1}s\tld{w},(\tld{w},y\tld{w}^{-1}(0))) \in SP(x^*)$ (this is represented by the red and blue parts in Figure \ref{BasicUL}). 
We claim that if $\tld{w}(x^*,\tld{w}_\tau) = \tld{w}^{-1}\tld{w}_h^{-1} w_0 s \tld{w}$, then %
$x^*\in C_{(\tld{sw},y\tld{sw}^{-1}(0))}(y^*)$
(this is represented by the arrows in Figure \ref{BasicUL}). 
It suffices to show that 
\begin{equation}\label{eqn:irreg}
\cI \backslash \cI (\tld{w}^{-1}\tld{w}_h^{-1} w_0 s \tld{w})^* \cI \tld{w}_\tau^* \cap \Fl^{\nabla_0} \subset \cI \backslash \cI t_{(sw)^{-1}(\eta)}^* \cI t_{-(sw)^{-1}(\eta)}^* y^* \cap \Fl^{\nabla_0}. 
\end{equation}
Using \eqref{item:2-deep}, \cite[Theorem 4.2.4]{MLM} shows that both
\begin{equation}
\label{eqn:2spaces} \cI \backslash \cI(\tld{w}^{-1}\tld{w}_h^{-1} w_0 s \tld{w})^* \cI \tld{w}_\tau^* \cap \Fl^{\nabla_0} %
\quad \textrm{and} \quad \cI \backslash \cI(w_0 s \tld{sw})^*\cI (\tld{sw}^{-1}\tld{w}_h^{-1})^* \tld{w}_\tau^* \cap \Fl^{\nabla_0}
\end{equation} 
are isomorphic to $\A^2_{\F}$. 
The equality $\tld{w}^{-1}\tld{w}_h^{-1} w_0 s \tld{w} = \tld{sw}^{-1} \tld{w}_h^{-1} w_0 s \tld{sw}$ implies that the latter space in  \eqref{eqn:2spaces} is contained in the former by the proof of \cite[Proposition 4.3.4]{MLM} (one can directly check that $(\tld{sw}^{-1}\tld{w}_h^{-1}) (w_0 s) \tld{sw}$ is a reduced factorization) so that the spaces in \eqref{eqn:2spaces} are equal. 
Finally, observe that 
\[
\cI \backslash \cI(w_0 s \tld{sw})^*\cI (\tld{sw}^{-1}\tld{w}_h^{-1})^* \tld{w}_\tau^* \subset \cI \backslash \cI(w_0 \tld{sw})^*\cI (w_0 s w_0^{-1})^*(\tld{sw}^{-1}\tld{w}_h^{-1})^* \tld{w}_\tau^*
\] 
and 
\begin{align*}
&\cI \backslash \cI(w_0 \tld{sw})^* \cI(w_0 s w_0^{-1})^*(\tld{sw}^{-1}\tld{w}_h^{-1})^* \tld{w}_\tau^* \cap \Fl^{\nabla_0} \\ = &\cI \backslash \cI(\tld{sw}^{-1} \tld{w}_h^{-1} w_0 \tld{sw})^*\cI (\tld{w}_h \tld{sw})^*(w_0 s w_0^{-1})^*(\tld{sw}^{-1}\tld{w}_h^{-1})^* \tld{w}_\tau^*\cap \Fl^{\nabla_0} \\
=& \cI \backslash \cI t_{(sw)^{-1}(\eta)}^* \cI t_{-(sw)^{-1}(\eta)}^* y^* \cap \Fl^{\nabla_0}
\end{align*} 
by (the proof of) \cite[Proposition 4.3.4]{MLM}. 
Putting this all together yields \eqref{eqn:irreg}. 

\end{proof}
\begin{figure}[h]
    \centering
     \includegraphics[scale=.1]{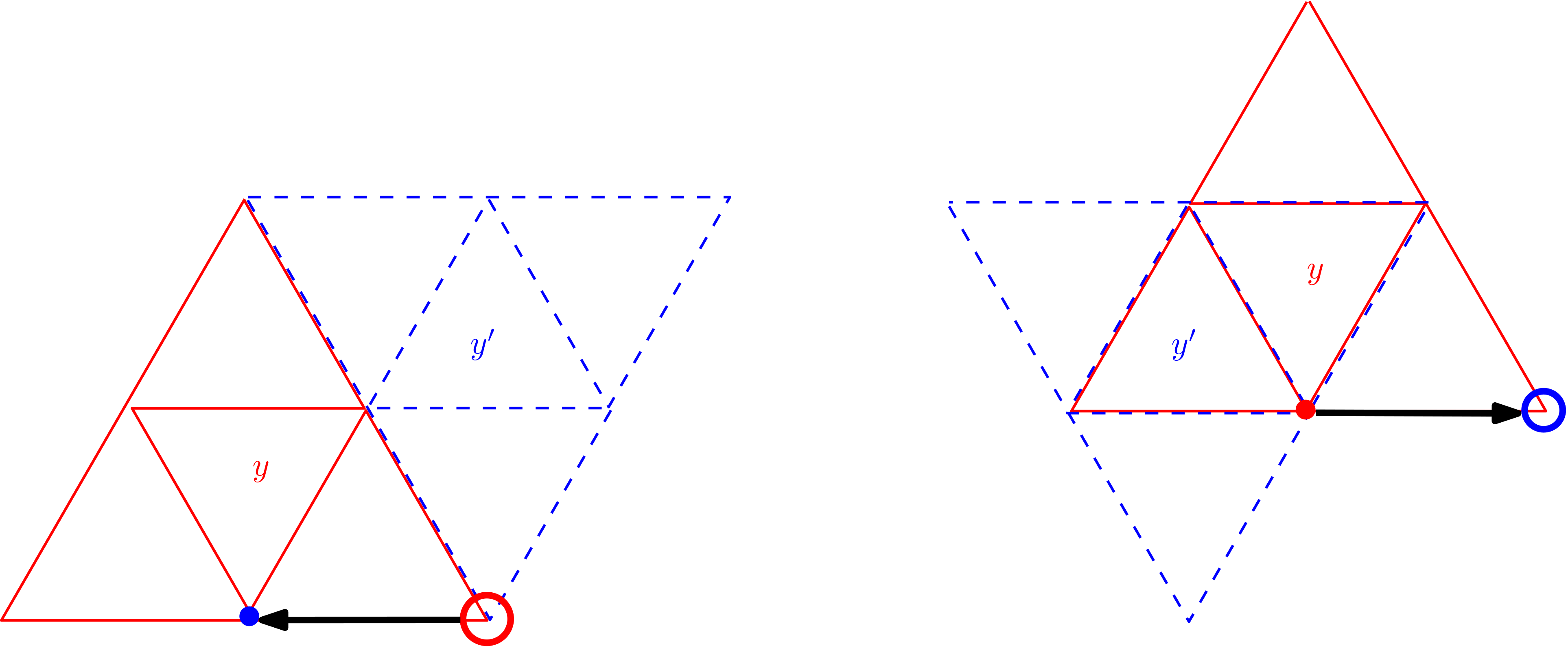} 
        \caption{We illustrate the dichotomy given by the last paragraph of the proof of Proposition \ref{prop:walk}. We represent the data $(y,(\tld{w},\omega))\in SP(x^*)$ by the alcove labeled by $y$ and the dot (resp.~circle) at $\omega\in X^*(T)/X^0(T)$ when $\tld{w}\cdot C_0$ is lower (resp.~upper) alcove (thus the left picture is the case where $\tld{w}\cdot C_0$ is upper alcove, and the right picture the case where $\tld{w}\cdot C_0$ is lower alcove). The starting pair $(y,(\tld{w},y\tld{w}^{-1}(0)))\in SP(x^*)$ is given by the red triangle (with vertexes labeled by the set $W_{\obv}(y^*)$) and the source of the arrow (labeled by the obvious weight $(\tld{w},y\tld{w}^{-1}(0))\in W_{\obv}(x^*)$).
        The dotted triangle represents the possible new specialization, while the tip of the arrow represents the new obvious weight.}
        \label{BasicUL}
\end{figure}

\begin{prop}\label{prop:SPnonempty}
For $x^* \in \Fl_{2\textrm{-deep}}^{\nabla_0}$, $SP(x^*)$ is nonempty. 
\end{prop}
\begin{proof}
By hypothesis, we have $x^*\in C_{(\tld{w},\omega)}$ for some $(\tld{w},\omega) \in \tld{W}_1 \times \cC_{2\text{-deep}}$. It follows from the definition of $C_{(\tld{w},\omega)}$ that it is a closed subscheme of $\ovl{S^{\circ}_\F(\tld{w}^*w_0^*)}t_{\omega^*} \cap \Fl^{\nabla_0}$.
Thus $C_{(\tld{w},\omega)}\subset \cup_{\tld{z} \leq \tld{w}^*w_0^*} S^{\circ}_\F(\tld{z})t_{\omega^*}$. Note that $S^{\circ}_\F(\tld{z})\subset  T_\F^\vee\backslash \tld{U}(\tld{z})_{\F}t_{\omega^*}= T_\F^\vee\backslash \tld{U}(\tld{z}t_{\omega^*})_{\F}$, for example cf.~\cite[Proposition 4.2.13]{MLM}. It follows that $C_{(\tld{w},\omega)}$ has an open cover $\cup_{\tld{z} \leq \tld{w}^*w_0^*}C_{(\tld{w},\omega)}(\tld{z}t_{\omega^*})$.

If $\tld{w}\cdot C_0=C_0$, then any $\tld{z}\leq \tld{w}^*w_0^*$ satisfies $\tld{z}^*\in W\tld{w}$. Choosing such a $\tld{z}$ with $x^*\in C_{(\tld{w},\omega)}(\tld{z}t_{\omega^*})$ gives $(t_{\omega}\tld{z}^*,(\tld{w},\omega))\in SP(x^*)$ and we are done.

Suppose now that $\tld{w}\cdot C_0$ is the upper $p$-restricted alcove. For $\tld{z}\leq  \tld{w}^*w_0^*$, either $\tld{z}^*\in W \tld{w}^*$, or $\tld{z}^*\in W\tld{w}^{\prime}$, where $\tld{w}^{\prime}$ is the unique element in $\tld{W}_1$ such that $\tld{w}^\prime < \tld{w}$. In particular $\tld{w}^\prime \cdot C_0=C_0$.
There are two cases:
\begin{itemize}
\item If $x^*\in C_{(\tld{w},\omega)}(\tld{z}t_{\omega^*})$ for some $\tld{z}^*\in W\tld{w}^*$, we get $(t_{\omega}\tld{z}^*,(\tld{w},\omega))\in SP(x^*)$ as above.
\item Otherwise, $x^*\in (\cup_{\tld{z}^*\in W\tld{w}^\prime} S^\circ_{\F}(\tld{z})t_{\omega^*})\cap \Fl^{\nabla_0}=C_{(\tld{w}^\prime,\omega)}$. Repeating our arguments with $(\tld{w},\omega)$ replaced by $(\tld{w}^\prime, \omega)$, we are also done in this case.
\end{itemize}
\end{proof}

\begin{cor}\label{cor:surjtheta}
Let $x^* \in \Fl_{2\textrm{-deep}}^{\nabla_0}$. 
If there exists $y_0\in S(x^*)$ with $y_0(0) \in \cC_{6\textrm{-deep}}$, then $\theta_{x^*}: SP(x^*) \ra W$ is bijective.
\end{cor}
\begin{proof}
By Proposition \ref{prop:injtheta}, it suffices to show that $\theta_{x^*}: SP(x^*) \ra W$ is surjective. 
By Proposition \ref{prop:SPnonempty}, $SP(x^*)$ is nonempty. 
If $(y,(\tld{w},y\tld{w}^{-1}(0))) \in SP(x^*)$ and $s\in W$ is a simple reflection, then either $(y\tld{w}^{-1}s\tld{w},(\tld{w},y\tld{w}^{-1}(0))) \in SP(x^*)$ or $(y,(\tld{sw},y\tld{sw}^{-1}(0))){\in} SP(x^*)$ by Proposition \ref{prop:walk} so that ${\theta_{x^*}}(y,(\tld{w},y\tld{w}^{-1}(0))) s$ is in the image of ${\theta_{x^*}}$. 
(The hypothesis that $y_0(0)\in \cC_{6\textrm{-deep}}$ guarantees that Proposition \ref{prop:walk}\eqref{item:2-deep} applies.)
Since simple reflections generate $W$, the result follows.
\end{proof}

\begin{lemma}\label{lemma:obvwts}
Suppose that $x^* \in \Fl_{2\textrm{-deep}}^{\nabla_0}$ such that there exists $y_0\in S(x^*)$ with $y_0(0) \in \cC_{6\textrm{-deep}}$. Then there exists $\lambda-\eta \in  C_0$ and $w\in W$ such that the image of $W_{\mathrm{obv}}(x^*)$ under \eqref{eq:bij:SW} is one of the following sets: 
\begin{enumerate}
\item
\label{case1} 
$w\{(0,0)\}$;
\item
\label{case2} 
$w\{(\eps_1-\eps_2,1)\}$;
\item
\label{case3} 
$w\{(0,0),(\eps_1-\eps_2,1)\}$;
\item 
\label{case4}
$w\{(0,0),(\eps_1-\eps_2,1), (\eps_2-\eps_1,1)\}$;
\item 
\label{case5}
$w\{(0,1),(\eps_1,0), (\eps_2,0), (\eps_1+\eps_2,1)\}$; and
\item 
\label{case6}
$w\{(0,0),(\eps_1-\eps_2,1), (\eps_2-\eps_1,1), (\eps_1,0),(\eps_2,0),(\eps_1+\eps_2,1)\}$.
\end{enumerate}
Moreover, every possibility arises.
Finally, with respect to the six above alternatives for $W_{\mathrm{obv}}(x^*)$ the image of $W^g_{2\textrm{-deep}}(x^*)$ under \eqref{eq:bij:SW} is contained in 
\begin{enumerate}[label=(\arabic*$'$)]
\item 
\label{case1'}
$w\{(0,0),(0,1)\}$;
\item 
\label{case2'}
$w\{(\eps_1-\eps_2,1)\}$;
\item 
\label{case3'}
$w\{(0,0),(0,1),(\eps_1-\eps_2,1)\}$;
\item 
\label{case4'}
$w\{(0,0),(0,1),(\eps_1-\eps_2,1), (\eps_2-\eps_1,1)\}$;
\item 
\label{case5'}
$w\{(0,1),(\eps_1,0),(\eps_1,1), (\eps_2,0),(\eps_2,1), (\eps_1+\eps_2,1)\}$; and
\item 
\label{case6'}
$w\{(0,0),(0,1),(\eps_1-\eps_2,1), (\eps_2-\eps_1,1), (\eps_1,0),(\eps_1,1),(\eps_2,0),(\eps_2,1),(\eps_1+\eps_2,1)\}$.
\end{enumerate}
\end{lemma}
\begin{rmk}
The sets in the second part of Lemma \ref{lemma:obvwts} are the minimal sets containing the corresponding sets in the first part closed under changing a $0$ in the second argument to a $1$. 
Since the set in the second part are obtained by taking intersections $\cap_{y\in S(x^*)} W^g_{2\textrm{-deep}}(y^*)$ which are closed under this operation, these sets are a natural upper bound for $W^g_{2\textrm{-deep}}(x^*)$. 
\end{rmk}
\begin{proof}
We will illustrate the proof with various figures, all of which follow the same graphic conventions as in Figure \ref{BasicUL}.

Recall that we have canonical isomorphisms $\tld{W}/W_a\stackrel{\sim}{\ra}X^*(Z)$ and $\pi_0(\Fl)\stackrel{\sim}{\ra}X^*(Z)$.
In this proof we will choose various $\lambda \in X^*(T)$ with the property that the image of $t_\lambda$ in $X^*(Z)$ is the same as the image of $x^*$, and use \eqref{eq:bij:SW} to identify $(\tld{W}_1 \times (X^*(T) \cap C_0 +\eta)^{\lambda-\eta|_{\un{Z}}})/\sim\}$ and $\Lambda_W^{\lambda} \times \cA$, since the latter set is more convenient to work with here.

By Corollary \ref{cor:surjtheta}, one obtains the elements of $W_{\mathrm{obv}}(x^*)$ by repeatedly applying the process described in Proposition \ref{prop:walk} which we call a \emph{simple walk}. 
We use the following two basic facts repeatedly. 
\begin{enumerate}[label=(\roman*)]
\item \label{item:adjacent} If $(\eps,a) \in W_{\mathrm{obv}}(x^*)$, then either $W_{\mathrm{obv}}(x^*) = \{(\eps,a)\}$ or there is a simple walk from some $(y,(\eps,a))$ giving another element $(\eps',a') \in W_{\mathrm{obv}}(x^*)$ in which case $a \neq a'$ and $\eps-\eps' \in W\{\eps_1,\eps_2\}$. 
\item \label{item:upperbound} $W^g_{2\textrm{-deep}}(x^*) \subset \cap_{y\in S(x^*)} W^g_{2\textrm{-deep}}(y^*)$ by Lemma \ref{lemma:semicont}\eqref{item:SWineq}.
\end{enumerate}

The analysis can be divided into a number of cases.

\begin{itemize}
\item Suppose that there is no element of the form $(\eps,1)$ in $W_{\mathrm{obv}}(x^*)$. 
Then $W_{\mathrm{obv}}(x^*)$ consists of a single element by \ref{item:adjacent}, which after changing $\lambda$, we assume to be $\{(0,0)\}$. 
Then by Corollary \ref{cor:surjtheta}, $S(x^*) = \{t_\lambda w \mid w \in W\}$. 
Then $W^g_{2\textrm{-deep}}(x^*) \subset \cap_{y\in S(x^*)} W^g_{2\textrm{-deep}}(y^*) = \{(0,0),(0,1)\}$ by \ref{item:upperbound}; see Figure \ref{CaseI}.
This gives (\ref{case1}) and \ref{case1'}. 
\begin{figure}[h]
    \centering
     \includegraphics[scale=.25]{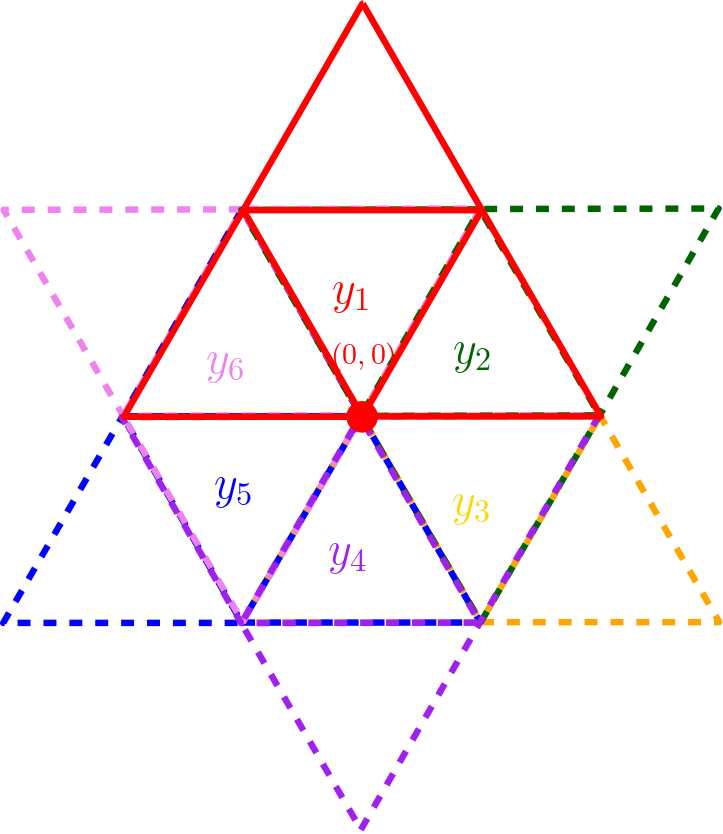} 
        \caption{Case where $W_{\mathrm{obv}}(x^*)=\{(0,0)\}$ and $\#S(x^*)=\#W$. The set $W_{\mathrm{obv}}(x^*)$ is pictured by the red dot and the set $S(x^*)$ by the $6$ triangles. The simple walks only produce new specializations.}
        \label{CaseI}
\end{figure}
\item 
Suppose now that $(\eps,1)\in W_{\mathrm{obv}}(x^*)$ for some $\eps \in \Lambda_W$. 
Say $(y_1,(\eps,1)) \in SP(x^*)$. 
If a simple walk produces $(y_2,(\eps,1)) \in SP(x^*)$ (i.e.~a new element of $S(x^*)$ rather than $W_{\obv}(x^*)$), then there exist $\lambda - \eta \in  C_0$ and $w\in W$ such that $y_1 = t_\lambda w$ and $y_2 = t_\lambda w w_0$. 
Then $W^g_{2\textrm{-deep}}(y_1^*) \cap W^g_{2\textrm{-deep}}(y_2^*) = w\{(0,0),(0,1),(\eps_1-\eps_2,1),(\eps_2-\eps_1,1)\}$. 
Fact \ref{item:adjacent} precludes $w(0,1)$ from being in $W_{\obv}(x^*)$. 
Moreover, if $w(\eps_1-\eps_2,1)$ and $w(\eps_2-\eps_1,1)$ are in $W_{\obv}(x^*)$, then so is $w(0,0)$.
Changing $w$ if necessary, we assume that $\eps$ is $w(\eps_1-\eps_2,1)$.
Then $W_{\mathrm{obv}}(x^*)$ is one of cases (\ref{case2}), (\ref{case3}), or (\ref{case4}).
\begin{itemize} 
\item In case (\ref{case2}), $S(x^*)$ has six elements by Corollary \ref{cor:surjtheta}, and furthermore $W^g_{2\textrm{-deep}}(x^*) \subset \cap_{y\in S(x^*)} W^g_{2\textrm{-deep}}(y^*) = \{(\eps,1)\}$ by \ref{item:upperbound}---this is \ref{case2'}; see Figure \ref{CaseII-2}.
\begin{figure}[h]
    \centering
     \includegraphics[scale=.25]{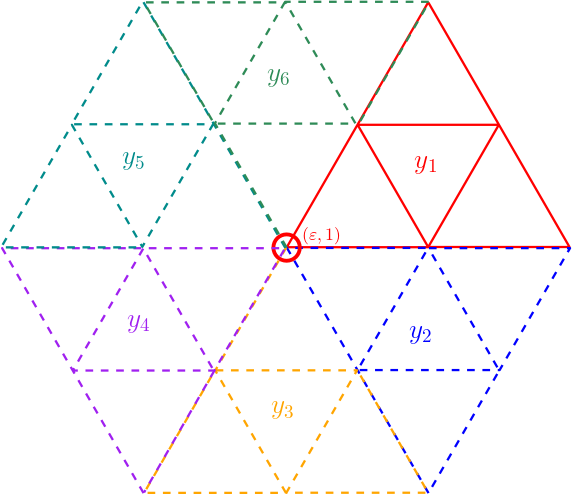} 
        \caption{Case where $W_{\mathrm{obv}}(x^*)=\{(\eps,1)\}$ and $\#S(x^*)=\#W$. The set $W_{\mathrm{obv}}(x^*)$ is pictured by the red circle and the set $S(x^*)$ by the $6$ triangles. Again, the simple walks only produce new specializations.}
        \label{CaseII-2}
\end{figure}
\item In case (\ref{case3}), $S(x^*)$ has four elements, and $\cap_{y\in S(x^*)} W^g_{2\textrm{-deep}}(y^*)$ is given by \ref{case3'}; see Figure \ref{CaseII-3}.
\begin{figure}[h]
    \centering
     \includegraphics[scale=.12]{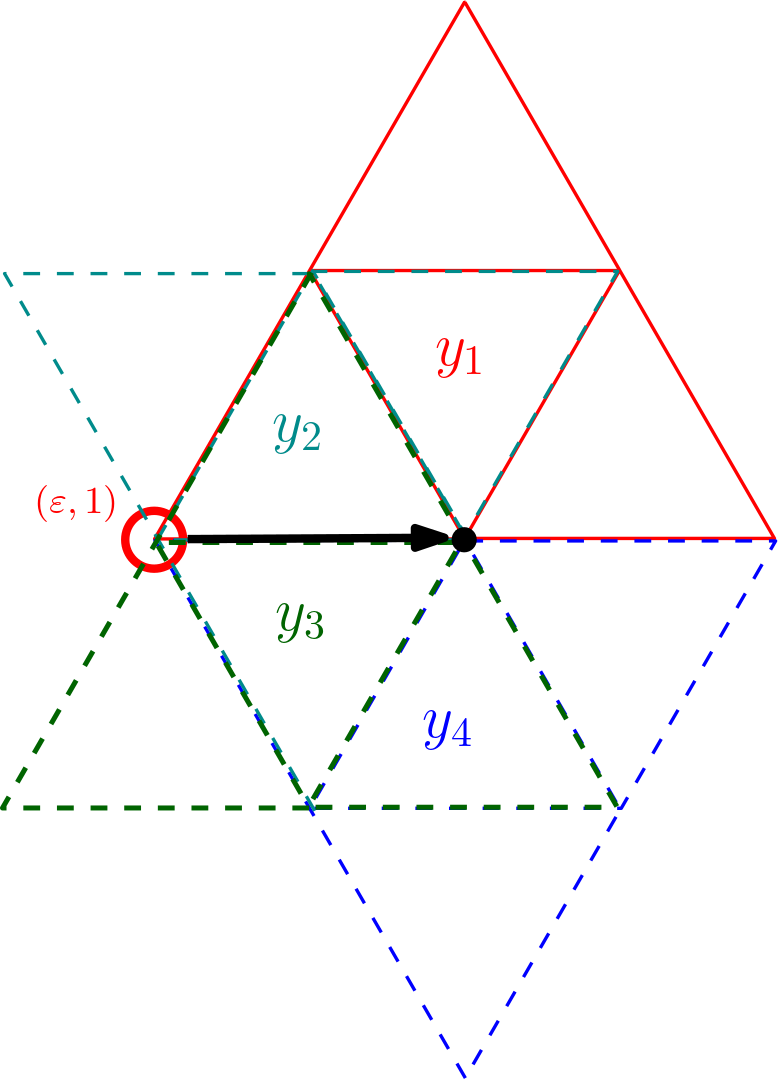} 
        \caption{Case where $\#S(x^*)=4$. We have the four specializations given by the triangles, and the starting obvious weight $(\varepsilon,1)$ by the red circle.
        The simple walk producing the new obvious weight is pictured by the black thickened arrow.}
        \label{CaseII-3}
\end{figure}
\item In case (\ref{case4}), given by Figure \ref{CaseII-4}, $S(x^*)=\{y_1,y_2\}$ and $W^g_{2\textrm{-deep}}(y_1^*) \cap W^g_{2\textrm{-deep}}(y_2^*)$ is given by \ref{case4'}. 
\begin{figure}[h]
    \centering
     \includegraphics[scale=.12]{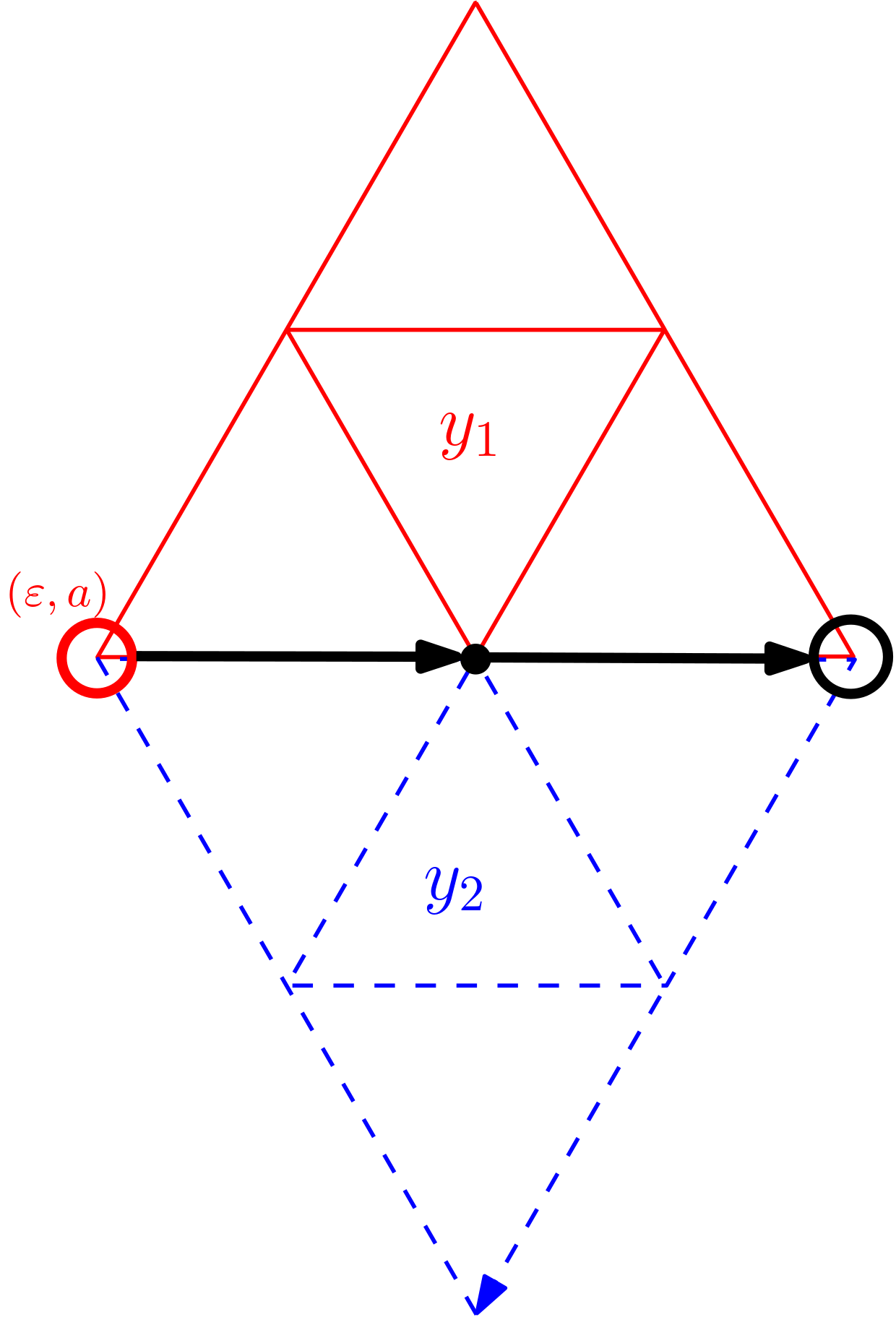} 
        \caption{Case where $\#S(x^*)=2$. We have the two specializations given by the red and blue triangle, and the starting obvious weight $(\varepsilon,1)$ by the red circle.
        The simple walks producing the new obvious weights are pictured by the black thickened arrows.}
        \label{CaseII-4}
\end{figure}
\end{itemize}
\item
Finally, we suppose that $(\eps,1)\in W_{\mathrm{obv}}(x^*)$ for some $\eps \in \Lambda_W$ \emph{and} a simple walk starting with $(\eps',1)$ for \emph{any} $\eps' \in \Lambda^\lambda_W$ such that $(\eps',1) \in W_{\mathrm{obv}}(x^*)$ \emph{always} yields a new element of $W_{\obv}(x^*)$. 
Let $(y_1,(\eps,1)) \in SP(x^*)$. 
After possibly changing $\lambda$, there exists $w\in W$ such that $W^g_{2\textrm{-deep}}(y_1^*)$ is 
\[
w\{(0,0),(0,1),(\eps_1-\eps_2,1), (\eps_2-\eps_1,1), (\eps_1,0),(\eps_1,1),(\eps_2,0),(\eps_2,1),(\eps_1+\eps_2,1)\}
\] 
and $\eps = \eps_1+\eps_2$. 
By assumption, we have that $(y_1,w(\eps_1,0))$ and $(y_1,w(\eps_2,0)) \in SP(x^*)$. 
\begin{itemize}
\item
If a simple walk starting with $(y_1,w(\eps_1,0))$ yields $(y_2,w(\eps_1,0)) \in SP(x^*)$ for some $y_2 \in \tld{W}$, then $W^g_{2\textrm{-deep}}(x^*) \subset W^g_{2\textrm{-deep}}(y_1^*) \cap W^g_{2\textrm{-deep}}(y_2^*)$ which is \ref{case5'} (see Figure \ref{CaseIII-1}). 
\begin{figure}[h]
    \centering
     \includegraphics[scale=.15]{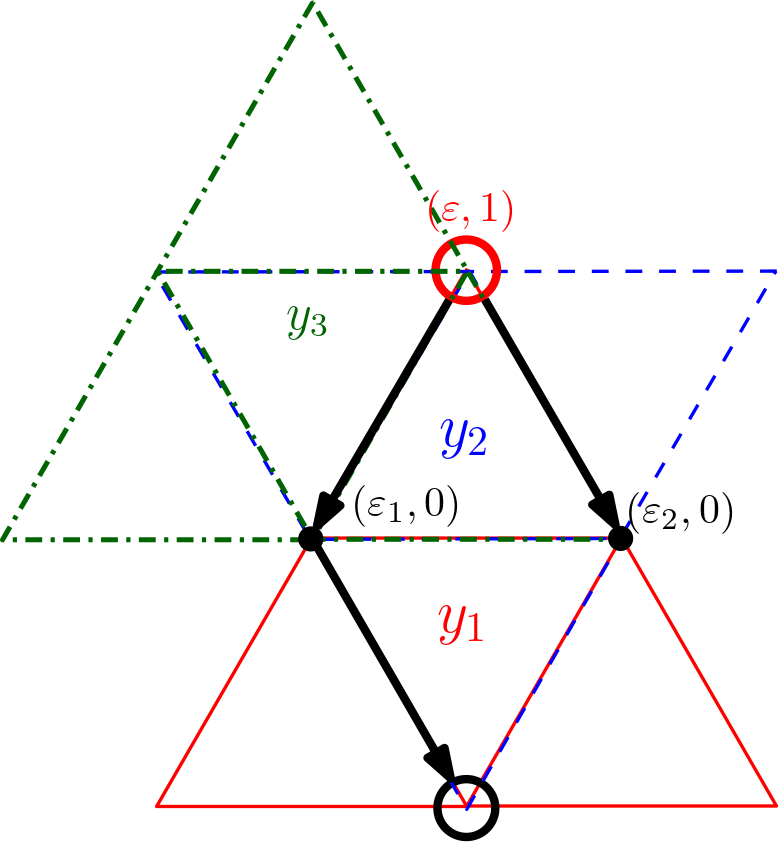} 
        \caption{We picture (for $w=1$) the simple walks pertaining the third bullet point in the analysis of the proof of Lemma \ref{lemma:obvwts}.
        The arrows from $(\eps,1)\in W_{\obv}(x^*)$ to $(\eps_1,0), (\eps_2,0)$ exist by the assumption that any simple walk from $(\eps,1)$ always yields a new element of $W_{\obv}(x^*)$.
        \newline
        \indent
        In the picture above, we consider the case where a simple walk from $(y_1,(\eps_1,0))$ yields $(y_2,(\eps_1,0)) \in SP(x^*)$ with $y_2\neq y_1$ (such $y_{2}$ is uniquely determined.) %
        Since $(y_3,(\eps_2,0))\notin SP(x^*)$, we have the arrow from $(\eps_1,0)$ to $(0,1)$.}
        \label{CaseIII-1}
\end{figure}
We claim that $w(\eps_1,1)\notin W_{\mathrm{obv}}(\rhobar)$. 
If $w(\eps_1,1)\in W_{\mathrm{obv}}(\rhobar)$ then, as argued before with $w(\eps_1+\eps_2,1)$, there would necessarily be two elements in $W_{\obv}(x^*) \subset W^g_{2\textrm{-deep}}(y_1^*) \cap W^g_{2\textrm{-deep}}(y_2^*)$ which correspond to two adjacent vertices in Figure \ref{CaseIII-1}.
However, there are no such elements in \ref{case5'}. 
Similarly, $w(\eps_2,1)\notin W_{\mathrm{obv}}(x^*)$.
A simple walk from $(y_2,w(\eps_1,0)) \in SP(x^*)$ yields two elements in $SP(x^*)$---$(y_1,w(\eps_1,0))$ and $(y_3,w(\eps',a))$. 
If $y_2\neq y_3$, then $w(\eps_2,0) \notin W^g_{2\textrm{-deep}}(y_3^*)$ which contradicts Lemma \ref{lemma:semicont}\eqref{item:SWineq}. 
We conclude that $w(\eps',a) = w(0,1) \in W_{\mathrm{obv}}(\rhobar)$. 
This gives the set in (\ref{case5}).
\item
If the process in Proposition \ref{prop:walk} from $(y_1,(\eps_1,0))\in SP(x^*)$ yields $(y_1,w(\eps_1-\eps_2,1))\in SP(x^*)$ instead of $(y_2,w(\eps_1,0))$, then $S(x^*) = \{y_1\}$ by \ref{item:upperbound} since $y_1$ is the unique $y\in \tld{W}$ such that $W^g_{2\textrm{-deep}}(y^*)$ contains $w(\eps_1-\eps_2,1)$, $w(\eps_1,0)$, $w(\eps_2,0)$, and $w(\eps_1+\eps_2,1)$; see Figure \ref{CaseIII-2}. 
\begin{figure}[h]
    \centering
     \includegraphics[scale=.2]{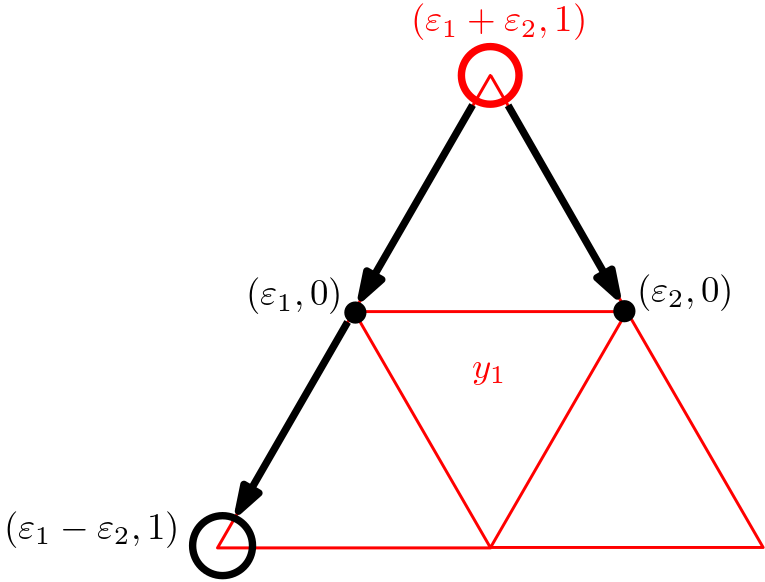} 
        \caption{Again, we consider the third bullet point in the analysis of the proof of Lemma \ref{lemma:obvwts}, but now we consider the case when a simple walk from $(y_1,(\eps_1,0))$ yields $(y_1,(\eps_1-\eps_2,1)) \in SP(x^*)$ (hence, yields the arrow from $(\eps_1,0)$ to $(\eps_1-\eps_2,1)$, cf.~with Figure \ref{CaseIII-1}).}
        \label{CaseIII-2}
\end{figure}
We conclude from Corollary \ref{cor:surjtheta} that $W_{\mathrm{obv}}(x^*) = W_{\obv}(y_1^*)$ which is given by (\ref{case6}). 
Then the upper bounds for $W^g_{2\textrm{-deep}}(x^*)$ in cases \eqref{case5} and \eqref{case6} again follow from \ref{item:upperbound}.
\end{itemize}
\end{itemize}

We now show that every possibility arises. 
\begin{itemize}
\item 
One checks that case \eqref{case6} arises when $x^* = \cI\backslash \cI (t_\lambda w)^*$. 
\item 
If we let $\gamma = t_{(1,0,-1)} w_0$ and $x^* \in \cI\backslash \cI \gamma^* \cI (t_\lambda w)^* (\F)$ but is not equal to $\cI\backslash \cI (t_\lambda w\gamma)^*$, then one can check that $t_\lambda w, \, t_\lambda w \gamma \in S(x^*)$ so that $W_{\mathrm{obv}}(x^*) \subset W^g_{2\textrm{-deep}}(t_\lambda w) \cap W^g_{2\textrm{-deep}}(t_\lambda w \gamma)$ by Lemma \ref{lemma:semicont}. 
This rules out \eqref{case6}. 
One can furthermore check that the image of $W_{\mathrm{obv}}(x^*)$ under \eqref{eq:bij:SW} contains \eqref{case5} so that out of the six possibilities it must equal \eqref{case5}. 
\item 
Next, if $s_1$ and $s_2 \in W$ denote the simple reflections and $x^*$ is generic in $\cI\backslash \cI s_1 s_2 \cI s_2 s_1 (t_\lambda w)^* \cap \cI\backslash \cI s_2 s_1 \cI s_1 s_2 (t_\lambda w)^*$, then $t_\lambda w, \, t_\lambda w w_0 \in S(x^*)$ so that $$W_{\mathrm{obv}}(x^*) \subset W^g_{2\textrm{-deep}}(t_\lambda w) \cap W^g_{2\textrm{-deep}}(t_\lambda w w_0)$$ by Lemma \ref{lemma:semicont}. 
This rules out \eqref{case6}. 
One can furthermore check that the image of $W_{\mathrm{obv}}(x^*)$ under \eqref{eq:bij:SW} contains \eqref{case4} so that out of the six possibilities it must equal \eqref{case4}. 
\item 
If $x^*$ is generic in $\cI\backslash \cI s_2 s_1 \cI s_1 s_2 (t_\lambda w)^*$, then $t_\lambda w, \, t_\lambda w s_2,\, t_\lambda w s_2s_1, \, t_\lambda w w_0 \in S(x^*)$. 
Then by similar arguments as above $W_{\mathrm{obv}}(x^*)$ is contained in \eqref{case3}. 
One can furthermore check that $W_{\mathrm{obv}}(x^*)$ contains \eqref{case3}. 
\item 
If $x^*$ is generic in $\cI\backslash \cI w_0 \cI (t_\lambda w)^*$, then $W^g_{2\textrm{-deep}}(x^*)$ only has one element, namely \eqref{case1}. 
\item If $x^*$ is generic in an upper alcove component, then $W^g_{2\textrm{-deep}}(x^*)$ only has one element corresponding to \eqref{case2}. 
\end{itemize}
\end{proof}

\begin{thm}\label{thm:Wg'}
Suppose that $x^* \in \Fl_{2\textrm{-deep}}^{\nabla_0}$ such that there exists $y_0\in S(x^*)$ with $y_0(0) \in \cC_{6\textrm{-deep}}$.
Then there exist $\lambda-\eta \in C_0$ and $w\in W$ such that under \eqref{eq:bij:SW} the image of $W_{\mathrm{obv}}(x^*)$ is as in Lemma \ref{lemma:obvwts} and the image of $W^g_{2\textrm{-deep}}(x^*)$ is correspondingly one of the following: 
\begin{enumerate}
\item
\label{case1:thm:Wg}
$w\{(0,0)\}$ or $w\{(0,0),(0,1)\}$;
\item 
\label{case2:thm:Wg}
$w\{(\eps_1-\eps_2,1)\}$;
\item 
\label{case3:thm:Wg}
$w\{(0,0),(\eps_1-\eps_2,1)\}$;
\item 
\label{case4:thm:Wg}
$w\{(0,0),(\eps_1-\eps_2,1), (\eps_2-\eps_1,1)\}$;
\item 
\label{case5:thm:Wg}
$w\{(0,1),(\eps_1,0),(\eps_1,1), (\eps_2,0),(\eps_2,1), (\eps_1+\eps_2,1)\}$; {and}
\item 
\label{case6:thm:Wg}
$w\{(0,0),(0,1),(\eps_1-\eps_2,1), (\eps_2-\eps_1,1), (\eps_1,0),(\eps_1,1),(\eps_2,0),(\eps_2,1),(\eps_1+\eps_2,1)\}$.
\end{enumerate}
Moreover, every possibility arises.
\end{thm}
\begin{proof}
We explain how the bounds on $W_{\obv}(x^*)$ and Table \ref{Table:intsct} can be used to determine $W^g_{2\textrm{-deep}}(x^*)$. 
Let $x^* \in \Fl_{2\textrm{-deep}}^{\nabla_0}$ be as in the statement of the theorem, and let $\lambda$ and $w$ be as in Lemma \ref{lemma:obvwts}.
Define $\Sigma^g(x^*)$ to be the image of $W^g_{2\textrm{-deep}}(x^*)$ under \eqref{eq:bij:SW}. 
Table \ref{Table:intsct} (with $s_j$ in the notation of \emph{loc.~cit}.~taken to be $w$ in this proof) implies that the number of irreducible components of the completion of $\tld{U}(\tld{z}, \eta,\nabla_{w^{-1}(\mu+\eta)})_{\F}$ at an $\F$-point is never three (and that each irreducible component is smooth). 
Theorem \ref{thm:local_model_cmpt} then implies that $\# \Sigma^g(x^*) \cap t_\nu s(\Sigma_0) \neq 3$ for all $t_\nu s \in W_a$. 
(The relevant type $\tau$ in Theorem \ref{thm:local_model_cmpt} is $4$-generic since $\tld{w}(y_0^*,\tld{w}(\tau)) \leq \tld{w}(x^*,\tld{w}(\tau)) \in \Adm(\eta)$ by Lemma \ref{lemma:semicont}\eqref{item:semicont} and $y_0(0) \in \cC_{6\textrm{-deep}}$.)
This is the key fact that we will use in our analysis of $\Sigma^g(x^*)$. %

The upper and lower bounds, say $\Sigma^{\mathrm{ub}}(x^*)$ and $\Sigma^{\mathrm{lb}}(x^*)$, respectively, for $\Sigma^g(x^*)$ from Lemma \ref{lemma:obvwts} give upper and lower bounds for $\Sigma^g(x^*) \cap t_\nu s(\Sigma_0)$ for each $t_\nu s \in W_a$. 
For each $(\eps,a) \in \Sigma^{\mathrm{ub}}(x^*)\setminus \Sigma^{\mathrm{lb}}(x^*)$ in cases \eqref{case3:thm:Wg}-\eqref{case6:thm:Wg}, one can choose $t_\nu s \in W_a$ such that
\begin{itemize}
\item $(\Sigma^{\mathrm{ub}}(x^*)\setminus \Sigma^{\mathrm{lb}}(x^*)) \cap t_\nu s(\Sigma_0) = \{(\eps,a)\}$; and
\item $\# \Sigma^{\mathrm{ub}}(x^*) \cap t_\nu s(\Sigma_0) = 3$ or $\# \Sigma^{\mathrm{lb}}(x^*) \cap t_\nu s(\Sigma_0) = 3$.
\end{itemize}
Then the fact above implies that $\# \Sigma^{\mathrm{lb}}(x^*) \cap t_\nu s(\Sigma_0) = 3$ if and only if $(\eps,a) \in \Sigma^g(x^*)$. 
From this, one checks that $\Sigma^g(x^*)$ is as claimed. 
We now illustrate in Figures \ref{fig:7} and \ref{fig:8} the choices of $\tau$ in cases \ref{case3:thm:Wg}-\ref{case6:thm:Wg}. 
\begin{figure}[h]
    \centering
     \includegraphics[scale=.25]{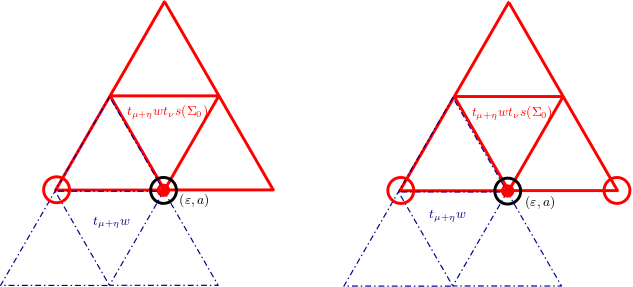} 
        \caption{Choice of $t_\nu s$ in cases \ref{case3:thm:Wg} and \ref{case4:thm:Wg}.
        The red circles and the dot represent $\Sigma^{\mathrm{lb}}(x^*)$, the black circle the element in $\Sigma^{\mathrm{ub}}(x^*)\setminus\Sigma^{\mathrm{lb}}(x^*)$. We have fixed an element $t_{\mu+\eta}w\in \tld{W}$ so that $s^{-1}t_{\nu^*}w^{-1}t_{\mu^*+\eta^*}\in S(x^*)$.} 
        \label{fig:7}
\end{figure}

\begin{figure}[h]
    \centering
     \includegraphics[scale=.15]{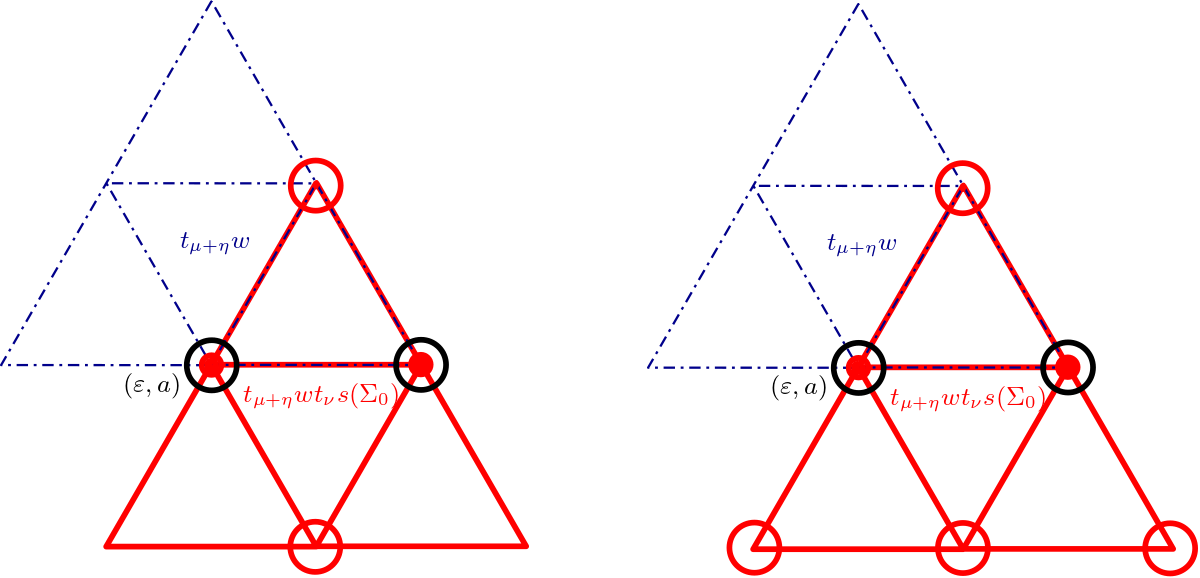} 
        \caption{Choice of $t_\nu s$ in cases \ref{case5:thm:Wg} and \ref{case6:thm:Wg}.The red circles and dots represent $\Sigma^{\mathrm{lb}}(x^*)$, the black circles the element in $\Sigma^{\mathrm{ub}}(x^*)\setminus\Sigma^{\mathrm{lb}}(x^*)$. We have fixed an element $t_{\mu+\eta}w\in \tld{W}$ so that $s^{-1}t_{\nu^*}w^{-1}t_{\mu^*+\eta^*}\in S(x^*)$.} 
        \label{fig:8}
\end{figure} 

Finally, we show that every possibility arises. 
Since Lemma \ref{lemma:obvwts} showed that every one of the six possibilities arises, we only need to show that the two possibilities in case \ref{case1:thm:Wg} arise. 
The case $w\{(0,0)\}$ arises when $x^*$ is generic on a lower alcove component so that $W^g_{2\textrm{-deep}}(x^*)$ only has one element corresponding to this component. 
The case $w\{(0,0),(0,1)\}$ arises when $x^*$ is generic in the intersection of the two components corresponding to $w\{(0,0),(0,1)\}$. 
Indeed, this intersection is two-dimensional so as long as $x^*$ is not in cases \ref{case5:thm:Wg} or \ref{case6:thm:Wg} which are of dimensions one and zero, respectively, the case $w\{(0,0),(0,1)\}$ must apply. 
\end{proof}

\begin{rmk}
\label{rmk:productFL}
The notions in this {section} extend to the case of products: if $x^* = (x^*_i)_{i\in \cJ} \in (\Fl^{\nabla_0}_{2\textrm{-deep}})^\cJ$, we let $W^g_{2\textrm{-deep}}(x^*)$ and $W_{\obv}(x^*)$ be the subsets $\prod_{i\in \cJ} W^g_{2\textrm{-deep}}(x^*_i)$ and $\prod_{i\in \cJ} W_{\obv}(x^*_i)$ of $(\tld{\un{W}}_1 \times (X^*(\un{T}) \cap  C_0 +\eta))/\sim$, respectively. 
The natural analogues of Theorems \ref{thm:Tfixed} and \ref{thm:Wg'} generalize to this setting.
\end{rmk}

\subsection{Classification of geometric weight sets}

Let $\rhobar: G_K \ra \GL_3(\F)$ be a continuous Galois representation.
If $\tau$ is a $4$-generic tame inertial type and $\rhobar$ arises as an $\F$-point of $\cX^{\eta,\tau}_{\F}$ then Theorem \ref{prop:loc:mod:diag:2} attaches a Breuil--Kisin module $\fM\in Y^{[0,2],\tau}(\F)$ to $\rhobar$.
In this scenario, we define the \emph{shape $\tld{w}^*(\rhobar,\tau)\in \tld{W}^{\vee,\cJ}$} of $\rhobar$ with respect to $\tau$ to be the shape of $\fM\in Y^{[0,2],\tau}(\F)$ (see \S \ref{subsub:BKM}).
Let $W^g_{2\textrm{-deep}}(\rhobar)$ be the set of Serre weights $\{\sigma \mid \text{$\sigma$ is $2$-deep and $\rhobar \in \cC_\sigma(\F)$}\}$. 
\begin{defn}
Let $SP(\rhobar)$ be the set of pairs $(\rhobar_1,\sigma)$ with $\rhobar_1$ a tame inertial $\F$-type  and $\sigma \in W^g_{2\textrm{-deep}}(\rhobar)$ such that there exists a $4$-generic tame inertial type $\tau$ with
\begin{itemize}
\item $\sigma \in \JH(\ovl{\sigma}(\tau))$ and 
\item $\tld{w}^*(\rhobar,\tau) = \tld{w}^*(\tld{\rhobar}_1,\tau) \in t_{\un{W}^\vee \eta}$  
\end{itemize}
for any, or equivalently all, extensions $\tld{\rhobar}_1: G_K \ra \GL_3(\F)$ of $\rhobar_1$
(in particular $\sigma\in W_{\mathrm{obv}}(\rhobar_1)$).
Let $W_{\mathrm{obv}}(\rhobar) \subset W^g_{2\textrm{-deep}}(\rhobar)$ be the image of $SP(\rhobar)$ under the projection to $W^g_{2\textrm{-deep}}(\rhobar)$. 
Let $S(\rhobar)$ be the image of $SP(\rhobar)$ under the projection to the set of tame inertial $\F$-types. 
We call an element of $S(\rhobar)$ a \emph{specialization} of $\rhobar$. 
\end{defn}

\begin{defn}\label{defn:generic}
We say that a Galois representation $\rhobar: G_K \ra \GL_3(\F)$ is $m$-generic if the tame inertial $\F$-type $\rhobar^{\semis}|_{I_K}$ is $m$-generic and $\rhobar$ has an $m$-generic specialization. 
\end{defn}

\begin{rmk}
\begin{enumerate}
\item If $m\geq 6$, $\rhobar: G_K \ra \GL_3(\F)$ is semisimple, and
$\rhobar|_{I_K}$ is $m$-generic, then $\rhobar$ is $m$-generic in the sense of Definition \ref{defn:generic} since $\rhobar|_{I_K} \in S(\rhobar)$. 
\item It is shown (in greater generality under a suitable genericity assumption) in \cite{OBW} that $\rhobar^{\semis}|_{I_K} \in S(\rhobar)$ so that the requirement that $\rhobar$ {has} an $m$-generic specialization in Definition \ref{defn:generic} is superfluous when $m$ is sufficiently large. 
\end{enumerate}
\end{rmk}

We now recall the setting of Theorem \ref{thm:local_model_cmpt}.
In particular, we have a pair $(\sigma,\zeta)$ which corresponds to a lowest alcove presentation $(\tld{w},\omega)$ of a Serre weight $\sigma$. Given the auxiliary choice of an appropriate tame inertial type $\tau$ (and letting $(s,\mu)$ be the compatible lowest alcove presentation), we have the diagram:
\begin{equation}
\begin{tikzcd}[column sep=small]
&\tld{C}^{\zeta}_\sigma=\tld{C}_{(\tld{w},\omega)}\ar[hook]{d}\ar{r} & C_{(\tld{w},\omega)}\ar[hook]{d}
\\
& \tld{\Fl}^{[0,2]}_{\cJ}\cdot s^*t_{\mu^*+\eta^*}\ar{d}{T^{\vee,\cJ}_{\F}}\ar{r}& \Fl_{\cJ}
\\
\cC_\sigma\ar[hook]{r}\ar[hook]{d}&\Big[\tld{\Fl}^{[0,2]}_{\cJ}\cdot s^*t_{\mu^*+\eta^*}/T^{\vee, \cJ}_{\F}\text{\textrm{-sh.cnj}}\Big]\ar[hook]{d}{\iota_0}\\
(\cX_{K,3})_{\F}\ar[hook]{r}&\Phi\text{-}\Mod^{\text{\emph{\'et}},n}_{K, \F}
\end{tikzcd}
\end{equation}
In particular the composite of the middle column gives a map $\tld{C}^{\zeta}_{\sigma}\ra\Phi\text{-Mod}_{K,\F}^{\text{\'et},n}$  which does not depend on $\tau$ and which factors through $\cX_{K,3}$. Note also that $\tld{C}^{\zeta}_{\sigma}$ is a subvariety of $(\tld{\Fl}^{\nabla_0})^{\cJ}$, and that the rightmost vertical arrow factors through $(\Fl^{\nabla_0})^{\cJ}$.

The following Proposition relates the Galois theoretic notions in \S \ref{subsub:SpFi} with the geometric notions in \S \ref{subsec:geometricFL} (or rather, its product version as in Remark \ref{rmk:productFL}):
\begin{prop}\label{prop:LMweights}
Let $\rhobar\in \cX_{K,3}(\F)$ be $6$-generic and $(\tld{w},\omega)$ be a lowest alcove presentation of an element of $W^g_{\textrm{gen}}(\rhobar)$ compatible with a $6$-generic lowest alcove presentation of $\rhobar^{\semis}$.
If $\tld{x}^* \in (\tld{\Fl}^{\nabla_0})^{\cJ}(\F)$ has images $x^* \in (\Fl_{2\textrm{-deep}}^{\nabla_0})^{\cJ}(\F)$ and $\rhobar \in \cX_{K,3}(\F)$
then $(\tld{w},\omega)$ is in $W_{\obv}(x^*)$ (resp.~$W^g_{2\textrm{-deep}}(x^*)$) if and only if $F_{(\tld{w},\omega)} \in W_{\obv}(\rhobar)$ (resp.~$W^g_{2\textrm{-deep}}(\rhobar)$). 
\end{prop}
\begin{proof}
This follows from Theorem \ref{thm:local_model_cmpt}, applied to suitably chosen auxiliary $4$-generic types $\tau$ containing $F_{(\tld{w},\omega)}$. 
\end{proof}

\begin{thm}\label{thm:Wg}
If $\rhobar: G_K \ra \GL_3(\F)$ is $6$-generic, then there exist $\lambda \in X^*(\un{T})$ and $w\in \un{W}$ such that $W_{\obv}(\rhobar)$ is $F(\Trns_\lambda(\prod_{j\in \cJ} w_j\Sigma_{\mathrm{obv},j}(\rhobar)))$ where for each $j\in \cJ$, $\Sigma_{\mathrm{obv},j}(\rhobar)$ one of the following sets: 
\begin{enumerate}
\item
\label{case1} 
$\{(0,0)\}$;
\item
\label{case2} 
$\{(\eps_1-\eps_2,1)\}$;
\item
\label{case3} 
$\{(0,0),(\eps_1-\eps_2,1)\}$;
\item 
\label{case4}
$\{(0,0),(\eps_1-\eps_2,1), (\eps_2-\eps_1,1)\}$;
\item 
\label{case5}
$\{(0,1),(\eps_1,0), (\eps_2,0), (\eps_1+\eps_2,1)\}$; and
\item 
\label{case6}
$\{(0,0),(\eps_1-\eps_2,1), (\eps_2-\eps_1,1), (\eps_1,0),(\eps_2,0),(\eps_1+\eps_2,1)\}$.
\end{enumerate}
Furthermore, $W^g_{\mathrm{gen}}(\rhobar) = F(\Trns_\lambda(\prod_{j\in \cJ} w_j\Sigma^{g}_j(\rhobar)))$ where with respect to the six above alternatives for $\Sigma_{\mathrm{obv},j}(\rhobar)$, $\Sigma^{g}_j(\rhobar)$ is 
\begin{enumerate}[label=(\arabic*$'$)]
\item
\label{case1:thm:Wg}
$\{(0,0)\}$ or $\{(0,0),(0,1)\}$;
\item 
\label{case2:thm:Wg}
$\{(\eps_1-\eps_2,1)\}$;
\item 
\label{case3:thm:Wg}
$\{(0,0),(\eps_1-\eps_2,1)\}$;
\item 
\label{case4:thm:Wg}
$\{(0,0),(\eps_1-\eps_2,1), (\eps_2-\eps_1,1)\}$;
\item 
\label{case5:thm:Wg}
$\{(0,1),(\eps_1,0),(\eps_1,1), (\eps_2,0),(\eps_2,1), (\eps_1+\eps_2,1)\}$; or
\item 
\label{case6:thm:Wg}
$\{(0,0),(0,1),(\eps_1-\eps_2,1), (\eps_2-\eps_1,1), (\eps_1,0),(\eps_1,1),(\eps_2,0),(\eps_2,1),(\eps_1+\eps_2,1)\}$. 
\end{enumerate}
Moreover, every possibility arises. 
If $\rhobar$ is furthermore $8$-generic, then $W^g(\rhobar)= F(\Trns_\lambda(\prod_{j\in \cJ} w_j\Sigma^{g}_j(\rhobar)))$ with $\Sigma^{g}_j$ as above.
\end{thm}
\begin{proof}
This follows from Proposition \ref{prop:LMweights}, Theorem \ref{thm:Wg'}, and Proposition \ref{prop:gWE}.
\end{proof}

\begin{rmk}\label{rmk:genbound}
By the proof of Theorem \ref{thm:Wg} (and Lemma \ref{lemma:obvwts}), if $\rhobar$ is $6$-generic and has an $m$-generic specialization, then every specialization is $(m-4)$-generic.  
\end{rmk}

\begin{thm}\label{thm:semicont}
Let $\tau$ and $\tau'$ be $4$-generic tame inertial types and $\rhobar \in \cX^{\eta,\tau}(\F)$ and $\rhobar' \in \cX^{\eta,\tau'}(\F)$. 
Suppose further that $\tld{x}^*,\tld{x}'^* \in (\tld{\Fl}^{\nabla_0})^{\cJ}(\F)$ have images $x^*,{x'}^* \in (\Fl_{2\textrm{-deep}}^{\nabla_0})^{\cJ}(\F)$ and $\rhobar,\rhobar' \in \cX_3(\F)$, respectively such that ${x'}^*$ is in the closure of the $\tld{T}^{\vee}_\F$-orbit of $x^*$. 
Then $\rhobar' \in \cX^{\eta,\tau}(\F)$ and $\tld{w}^*(\rhobar',\tau) \leq \tld{w}^*(\rhobar,\tau)$. 
Moreover, $W^g_{\mathrm{gen}}(\rhobar) \subset W^g_{\mathrm{gen}}(\rhobar')$. 
\end{thm}
\begin{proof}
Since $\rhobar \in \cX^{\eta,\tau}(\F)$, $x^*\in C_\sigma$ for some $\sigma \in \JH(\ovl{\sigma}(\tau))$ by Theorem \ref{thm:local_model_cmpt}. 
Since $C_\sigma$ is closed and $\tld{T}^{\vee}_\F$-stable, ${x'}^* \in C_\sigma$. 
We conclude that $\rhobar' \in \cX^{\eta,\tau}(\F)$ again using Theorem \ref{thm:local_model_cmpt}. 
Similarly, ${x'}^*$ is in the closure of $\cI \backslash \cI \tld{w}^*(\rhobar,\tau) \cI \tld{w}^*(\tau)$ which implies the desired inequality. 

If $\rhobar \in \cC_\sigma$ for a generic weight $\sigma$, then the same argument above shows that $\rhobar' \in \cC_\sigma$. 
\end{proof}

\begin{cor}\label{cor:semicont}
Let $\tau$ and $\tau'$ be $4$-generic tame inertial types and $\rhobar \in \cX^{\eta,\tau}(\F) \cap \cX^{\eta,\tau'}(\F)$. 
Suppose that $\rhobar_1 \in \cX^{\eta,\tau'}(\F)$ such that $\tld{w}^*(\rhobar_1,\tau') = \tld{w}^*(\rhobar,\tau')$. 
Then $\rhobar_1 \in \cX^{\eta,\tau}(\F)$ and $\tld{w}^*(\rhobar_1,\tau) \leq \tld{w}^*(\rhobar,\tau)$. 
\end{cor}
\begin{proof}
This follows from Theorem \ref{thm:semicont} since there is a contracting $\tld{T}_\F^\vee$-cocharacter for each translated Schubert cell (see \cite[Lemma 3.4.7]{MLM}). 
\end{proof}
\section{Results for patching functors}
\label{sec:PF:global_applications}

We start in \S \ref{subsec:WMPF} by recalling the formalism of weak (minimal) patching functors and we prove abstract versions of Serre weight conjectures assuming the modularity of an obvious weight (see Propositions \ref{prop:supportcyclebound}, \ref{prop:minimalcycle} below).
This assumption is removed in Section \ref{subsec:AM} if the weak patching functor comes from an \emph{arithmetic module}. %
In \S \ref{sec:cyc}, we prove results on cyclicity of patching functors arising from arithmetic modules and we finally give global applications of the above results in \ref{subsec:global}.

\subsection{Patching functors and Serre weights}
\label{subsec:WMPF}
We recall the setup and the basic definitions for weak minimal patching functors.
Recall from \S \ref{subsec:notations} that we write the finite \'etale $\Z_p$-algebra $\cO_p$ as the product $\prod\limits_{v\in S_p} \cO_v$, where $S_p$ is a finite set and for each $v\in S_p$, $\cO_v$ is the ring of integers in a finite unramified extension $F^+_v$ of $\Q_p$, and that  $^L \un{G}$ denotes the Langlands dual group of $G_0\defeq \Res_{\cO_p/\Z_p} ({\GL_3}_{/\cO_p})$.
Following \S \ref{subsub:Lp}, an $\F$-valued $L$-homomorphism $\rhobar: G_{\Q_p} \ra{}^L \un{G}(\F)$ (resp.~a tame inertial $L$-parameter $\tau: I_{\Q_p} \ra\un{G}^\vee(E)$) is identified with a collection $(\rhobar_v)_{v\in S_p}$ of continuous homomorphisms $\rhobar_v:G_{F^+_v}\ra\GL_3(\F)$ (resp.~with a collection $(\tau_v)_{v\in S_p}$ of tame inertial types $\tau_v:I_{F^+_v}\ra\GL_3(E)$).

Let $\rhobar$ be an $L$-homomorphism over $\F$ with corresponding collection $(\rhobar_v)_{v \in S_p}$.
We write $R_\infty$ for the $\cO$-algebra $R_{\rhobar} \widehat{\otimes}_{\cO} R^p$ where
\[
R_{\rhobar} \defeq \widehat{\bigotimes}_{v\in S_p,\cO} R_{\rhobar_v}^\square
\]
and $R^p$ is a (nonzero) complete local Noetherian equidimensional flat $\cO$-algebra with residue field $\F$ such that each irreducible component of $\Spec R^p$ and of $\Spec \overline{R}^p$ is geometrically irreducible (we remind the reader that $\ovl{M}$ denotes $M\otimes_\cO\F$ for any $\cO$-module $M$). 
We suppress the dependence on $R^p$ below.
For a Weil--Deligne inertial $L$-parameter $\tau$, let $R_\infty(\tau)$
be $R_\infty \otimes_{R_{\rhobar}} R_{\rhobar}^{\eta,\tau}$ 
where
\[
R_{\rhobar}^{\eta,\tau} \defeq \widehat{\bigotimes}_{v\in S_p,\cO} R_{\rhobar_{v}}^{\eta_{v},\tau_{v}}.%
\]
Let $X_\infty$, $X_\infty(\tau)$, and $\ovl{X}_\infty(\tau)$
be $\Spec R_\infty$, $\Spec R_\infty(\tau)$, and $\Spec \ovl{R}_\infty(\tau)$ 
respectively. 
Let $\Mod(X_\infty)$ be the category of coherent sheaves over $X_\infty$, and let $\Rep_{\cO}(\GL_3(\cO_p))$ denote the category of topological $\cO[\GL_3(\cO_p)]$-modules which are finitely generated over $\cO$. 

Recall from \S \ref{subsub:ILLC} that given a tame inertial $L$-parameter $\tau$ we have an irreducible smooth $E$-representation $\sigma(\tau)$ attached to it. 
If $\sigma^\circ(\tau)\subseteq \sigma(\tau)$ is an $\cO$-lattice, we write $\ovl{\sigma}^\circ(\tau)$ for $\sigma^\circ(\tau)\otimes_\cO\F$ in what follows.

\begin{defn}\label{minimalpatching}
A \emph{weak patching functor} for an $L$-homomorphism $\rhobar: W_{\Q_p} \ra$ $^L \un{G}(\F)$ is a nonzero covariant exact functor $M_\infty:\Rep_{\cO}(\GL_3(\cO_p))\ra \Coh(X_{\infty})$ satisfying the following: if $\tau$ is an inertial $L$-parameter and $\sigma^\circ(\tau)$ is an $\cO$-lattice in $\sigma(\tau)$ then
\begin{enumerate}
\item 
\label{it:minimalpatching:1}
$M_\infty(\sigma^\circ(\tau))$ is {either zero or} a maximal Cohen--Macaulay sheaf on $X_\infty(\tau)$; and
\item 
\label{it:minimalpatching:2}
for all $\sigma \in \JH(\ovl{\sigma}^\circ(\tau))$, $M_\infty(\sigma)$ is a maximal Cohen--Macaulay sheaf on $\ovl{X}_\infty(\tau)$ (or is $0$).
\item
\label{it:minimalpatching:3}
Suppose $\sigma^\circ$ is an $\cO$-lattice in a principal series representation $R_1(\mu)$.
Then $M_\infty(\sigma^\circ)$ is supported on the potentially \emph{semistable} locus of type $(\eta, \tau(1,\mu))$ in $X_\infty$.
\end{enumerate}
We say that a weak patching functor $M_\infty$ is \emph{minimal} if $R^p$ is formally smooth over $\cO$ and whenever $\tau$ is an inertial $L$-parameter, $M_\infty(\sigma^\circ(\tau))[p^{-1}]$, which is locally free over (the regular scheme) $\Spec R_\infty(\tau)[p^{-1}]$, has rank at most one on each connected component.
\end{defn}

\begin{rmk}
The above definition of weak patching functor is slightly weaker than that in \cite[Definition 6.2.1]{MLM} and closer in spirit to that of \cite[Definition 4.2.1]{LLL}: the purpose of the third item is to eliminate non-regular Serre weights.
\end{rmk}

Let $d$ be the (common) dimension of $\ovl{X}_\infty(\tau)$ for any inertial $L$-parameter $\tau$. 
If $M$ is an $\ovl{R}_\infty$-module whose action factors through $\ovl{R}_\infty(\tau)$ for some inertial $L$-parameter $\tau$, let $Z(M)$ be the associated $d$-dimensional cycle. 
Note that $Z(M_\infty(-))$ is additive in exact sequences. 

We now fix an $L$-homomorphism $\rhobar: W_{\Q_p} \ra$ $^L \un{G}(\F)$ and a weak patching functor $M_\infty$. 
Let $W(\rhobar)$ be the set of Serre weights $\sigma$ such that $M_\infty(\sigma) \neq 0$.

\begin{prop}
\label{prop:gen:8}
If $\rhobar: W_{\Q_p} \ra$ $^L \un{G}(\F)$ is an $L$-homomorphism with $6$-generic semisimplification, then $W_{\mathrm{gen}}(\rhobar)=W(\rhobar)$.
\end{prop}
\begin{proof} We adapt the argument in Proposition \ref{prop:gWE}: If $F(\lambda) \in W(\rhobar)$, then $M_\infty(F(\lambda))\neq 0$ so that $M_{\infty}(R_1(\lambda))\neq 0$. 
This implies that $\rhobar$, and hence $\rhobar^{\mathrm{ss}}$, has a potentially semistable lift of type $(\eta,\tau(1,\lambda))$. 
The rest of the argument is the same. 
\end{proof}

\begin{prop}\label{prop:WE}
\begin{enumerate}
\item If $\rhobar^{\semis}|_{I_K}$ is $7$-generic, then $W(\rhobar) \subset W^?(\rhobar^{\semis}|_{I_K})$. 
\item If $\rhobar^{\semis}|_{I_K}$ is $7$-generic, then for any $4$-generic $\rhobar_1\in S(\rhobar)$, $W(\rhobar) \subset W^?(\rhobar_1)$.
(Note that $S(\rhobar)$ consists of tame
inertial $\F$-types so that $W^?(\rhobar_1)$ is defined.)
\end{enumerate}
\end{prop}
\begin{proof}
Let $\sigma$ be in $W(\rhobar)$. 
By the proof of Propositions \ref{prop:gen:8} and \ref{prop:gWE}, $R_1(\lambda)$ is $5$-generic if $\sigma = F(\lambda)$ so that $\sigma$ is $3$-deep. 
Let $\rhobar_1$ be a $4$-generic element of $S(\rhobar)$ or $\rhobar^{\semis}|_{I_K}$. 
Suppose that $\sigma\notin W^?(\rhobar_1)$. 
By (the proof of) \cite[Lemma 2.3.13]{LLLM2} and Proposition \ref{prop:JH}, we can find a $1$-generic type $\tau$ such that $\sigma\in \JH(\sigma(\tau))$ and $\JH(\sigma(\tau))\cap W^?(\rhobar_1)=\emptyset$. 
That $\sigma \in W(\rhobar)$ implies that $\rhobar$, and thus $\rhobar^{\semis}$, has a potentially crystalline lift of type $\tau$ as in the proof of Proposition \ref{prop:gen:8}. 
\cite[Proposition 3.3.2]{LLL} implies that $\tau$ is $4$-generic. 
\cite[Lemma 5]{Enns} or Corollary \ref{cor:semicont} then implies that $\tld{\rhobar}_1 \in \cX^{\eta,\tau}(\F)$ for any extension $\tld{\rhobar}_1$ of $\rhobar_1$ to an $L$-homomorphism. 
Theorem \ref{thm:local_model_cmpt} implies that $\JH(\ovl{\sigma}(\tau)) \cap W^g_{\mathrm{gen}}(\tld{\rhobar}_1) \neq \emptyset$. 
On the other hand, 
\[
\JH(\ovl{\sigma}(\tau)) \cap W^g_{\mathrm{gen}}(\tld{\rhobar}_1) = \JH(\ovl{\sigma}(\tau)) \cap W^?_{\mathrm{gen}}(\rhobar_1) = \JH(\ovl{\sigma}(\tau)) \cap W^?(\rhobar_1) = \emptyset
\]
by Corollary \ref{cor:sspts}. 
This is a contradiction.
\end{proof}

For a Serre weight $\sigma$, let $\mathfrak{p}(\sigma)$ be the prime ideal or unit ideal in $R_{\rhobar}$ corresponding to the pullback of the stack $\cC_\sigma$ to $\Spec R_{\rhobar}$. 
For an inertial $L$-parameter $\tau$, let $I(\tau)$ be the kernel of the surjection $R_{\rhobar} \surj R_{\rhobar}^{\eta,\tau}$. Observe that if $I(\tau)\subset \fp(\sigma)\neq 1$, then $\fp(\sigma)$ induces a minimal prime of $\ovl{R}^{\eta,\tau}_{\rhobar}$, and all minimal primes arise this way.

\begin{lemma}\label{lemma:prodcomp}
Suppose that $\tau$ is an inertial $L$-parameter corresponding to a collection of $4$-generic tame inertial types $(\tau_v)_{v\in S_p}$. 
Then any minimal prime ideal of $R_\infty(\tau)$ is of the form $I(\tau)  R_\infty+ \mathfrak{p}R_\infty$ for some minimal prime ideal $\mathfrak{p} \subset R^p$. 

If $M$ is a nonzero finitely generated maximal Cohen--Macaulay $R_\infty(\tau)$-module, then $\Ann_{R_{\rhobar}} (\ovl{M}) = I(\tau)+(\varpi)$. 
\end{lemma}
\begin{proof}
Since $R_{\rhobar}^{\eta,\tau}$ is geometrically irreducible (its special fiber is reduced after arbitrary finite extension of $\F$ and hence is normal; see the proof of \cite[Lemma 3.5.4]{LLLM2}), the first part follows from \cite[Lemma 3.3(5)]{BLGGHT2}. Similarly, any minimal prime of $\ovl{R}_\infty(\tau)$ is of the form $\fp(\sigma)  \ovl{R}_\infty+ \ovl{\fp}\ovl{R}_\infty$, where $\fp(\sigma)$ {corresponds} to a minimal prime of $\ovl{R}^{\eta,\tau}_{\rhobar}$, and $\ovl{\fp}$ is a minimal prime of $\ovl{R}^p$.

If $M$ is a nonzero finitely generated maximal Cohen--Macaulay $R_\infty(\tau)$-module, then $Z(\ovl{M})$ is at least the reduction of the cycle in $\Spec R_\infty(\tau)[1/p]$ corresponding to a minimal prime of $R_\infty(\tau)$. 
In particular, for any prime $\fp(\sigma)$ of $R_{\rhobar}$ inducing a minimal prime of $\ovl{R}^{\eta,\tau}_{\rhobar}$, $\Ann_{R_\infty(\tau)} (\ovl{M})$ is contained in a prime induced by $\fp(\sigma) \ovl{R}_\infty+ \ovl{\fp} \ovl{R}_\infty$ for some minimal prime $\ovl{\fp}$ of $\ovl{R}^p$. 
Since $R_\infty/(\fp(\sigma) R_\infty+ \ovl{\fp} R_\infty) \cong {R_{\rhobar}/\fp(\sigma) \widehat{\otimes} R^p/\ovl{\fp}}$, $(\fp(\sigma) R_\infty+ \ovl{\fp} R_\infty) \cap R_{\rhobar} = \fp(\sigma)$ by Lemma \ref{lemma:completedtensor}. 
We conclude that $\Ann_{R_{\rhobar}} (\ovl{M}) \subset \fp(\sigma)$ for each minimal prime ideal $\fp(\sigma)\ovl{R}_\infty(\tau)$ of $\ovl{R}_\infty(\tau)$. 
Since $\ovl{R}_\infty(\tau)$ is reduced, $\Ann_{R_{\rhobar}} (\ovl{M}) \subset I(\tau)+(\varpi)$. 
The reverse inclusion is clear.
\end{proof}

\begin{lemma}\label{lemma:completedtensor}
Let $\F$ be a field. 
If $R$ and $S$ are complete Noetherian local $\F$-algebras with residue field $\F$, then the natural map $R \ra R \widehat{\otimes}_{\F} S$, $r\mapsto r \widehat{\otimes} 1$ is an injection.
\end{lemma}
\begin{proof}
Let $\fm_S \subset S$ be the maximal ideal. 
The composition $R \ra R \widehat{\otimes}_{\F} S \ra R \widehat{\otimes}_{\F} (S/\fm_S) \cong R \otimes_{\F} \F$ is the isomorphism given by $r \mapsto r\otimes 1$. 
The result follows. 
\end{proof}

For the rest of the section, we assume that $\rhobar$ is $8$-generic. In particular, every element of $S(\rhobar)$ is $4$-generic by Remark \ref{rmk:genbound}. 

\begin{lemma}\label{lemma:supportcyclebound}
If $\sigma_1$ is a Serre weight and $\Ann_{R_{\rhobar}} M_\infty(\sigma_1) \subset \mathfrak{p}(\sigma_2) \subsetneq R_{\rhobar}$ for a Serre weight $\sigma_2$, then $\sigma_2 \uparrow \sigma_1$. 
\end{lemma}
\begin{proof}
Since $M_\infty(\sigma_1)$ is nonzero by assumption, Proposition \ref{prop:WE} implies that $\sigma_1 \in W^?(\rhobar_1)$ for any specialization $\rhobar_1$ of $\rhobar$. 
Then $\sigma_1$ is $6$-deep because $\rhobar$ is $8$-generic. 
If $\sigma_1 \in \JH(\ovl{\sigma}(\tau))$ for a tame inertial type $\tau$ (necessarily $4$-generic), then 
\[
\cap_{\sigma \in \JH(\ovl{\sigma}(\tau))} \mathfrak{p}(\sigma) = (\varpi) + I(\tau) \subset \Ann_{R_{\rhobar}} M_\infty(\sigma_1) \subset \mathfrak{p}(\sigma_2) \subsetneq R_{\rhobar}
\]
by Theorem \ref{thm:local_model_cmpt}.
This implies that $\sigma_2 \in \JH(\ovl{\sigma}(\tau))$. 
We conclude that $\sigma_1$ covers $\sigma_2$ (Definition \ref{defn:cover}). 
The result now follows from Lemma \ref{lemma:coverup}. 
\end{proof}

\begin{lemma}\label{lemma:obvmodular}
If $W_{\mathrm{obv}}(\rhobar) \cap W(\rhobar)$ is nonempty, then $W_{\mathrm{obv}}(\rhobar) \subset W(\rhobar)$. 
\end{lemma}
\begin{proof}
Let $\sigma_0 \in W_{\mathrm{obv}}(\rhobar)$. 
{We claim that there is an $n\in \bN$ and sequences of tame inertial types $(\tau_i)_{i=1}^n$, specializations $(\rhobar_i)_{i=1}^n$ (elements in $S(\rhobar))$, and (not necessarily distinct) Serre weights $(\sigma_i)_{i=1}^n$ such that}
\begin{itemize}
\item $\{\sigma_i\}_{i=0}^{n} = W_{\mathrm{obv}}(\rhobar)$; 
\item $\rhobar_i \in S(\rhobar)$ for all $i = 1,\ldots, n$; 
\item $W^?(\rhobar_i) \cap \JH(\ovl{\sigma}(\tau_i)) = W_{\mathrm{obv}}(\rhobar_i) \cap \JH(\ovl{\sigma}(\tau_i))  = \{\sigma_{i-1},\sigma_i\}$ for all $i = 1,\ldots, n$. 
\end{itemize}
{Indeed, the proof of Corollary \ref{cor:surjtheta} gives a sequence of elements $(y,(\tld{w},y\tld{w}^{-1}(0)))$ in $SP(x^*)$. We define the sequences by taking the specializations $\rhobar_i$ corresponding to the elements $y$, taking the Serre weights $\sigma_i$ corresponding to the elements $F_{(\tld{w},y\tld{w}^{-1}(0))}$, and taking the tame inertial types $\tau_i$ to be $\tau(u,y(\tld{w}^{-1}\tld{w}_h^{-1}w_0s\tld{w})^{-1}(0))$ where $u$ is the image of $y(\tld{w}^{-1}\tld{w}_h^{-1}w_0s\tld{w})^{-1}$ in $W$ (see also the proof of Proposition \ref{prop:walk}). 
We will use these sequences to prove the result by induction.}

Suppose that $\sigma_{i-1} \in W(\rhobar)$ for some $1 \leq i \leq n$. 
Then $M_\infty(\ovl{\sigma}^\circ(\tau_i))$ is nonzero. 
Since $M_\infty(\ovl{\sigma}^\circ(\tau_i))$ is a nonzero finitely generated maximal Cohen--Macaulay $\ovl{R}_\infty(\tau_i)$-module, Proposition \ref{prop:WE} and Lemma \ref{lemma:prodcomp} give 
\[
\prod_{\sigma' \in \JH(\ovl{\sigma}(\tau_i)) \cap W^?(\rhobar_i)} \Ann_{R_{\rhobar}} M_\infty(\sigma') \subset \Ann_{R_{\rhobar}} M_\infty(\ovl{\sigma}^\circ(\tau_i)) = I(\tau_i) +(\varpi)\subset \fp(\sigma_i).
\]
Then $\Ann_{R_{\rhobar}} M_\infty(\sigma_{i-1}) \subset \fp(\sigma_i)$ or $\Ann_{R_{\rhobar}} M_\infty(\sigma_i) \subset \fp(\sigma_i)$. 
The former contradicts Lemma \ref{lemma:supportcyclebound} and so $\sigma_i \in W(\rhobar)$. 
\end{proof}

\begin{lemma}\label{lemma:mintype}
If $\rhobar$ is semisimple and $8$-generic, and $\sigma \in W^?(\rhobar)$, then there exists a tame inertial $L$-homomorphism $\tau$ such that 
\begin{enumerate}
\item $\sigma \in \JH(\ovl{\sigma}(\tau))$; and
\item $\sigma' \in W^?(\rhobar) \cap \JH(\ovl{\sigma}(\tau))$ implies that $\sigma' \uparrow \sigma$. 
\end{enumerate}
\end{lemma}
\begin{proof}
This follows from \cite[Lemma 3.5.9]{LLLM2}. 
\end{proof}

In fact, $\tau$ is unique. 
We say that $\tau$ is minimal with respect to $\rhobar$ and $\sigma$. 

\begin{prop}\label{prop:supportcyclebound}
Let $\rhobar: W_{\Q_p} \ra \, ^L \un{G}(\F)$ be an $8$-generic $L$-homomorphism and let $M_\infty$ be a weak patching functor. 
If $W_{\mathrm{obv}}(\rhobar) \cap W(\rhobar)$ is nonempty, then $\Ann_{R_{\rhobar}} M_\infty(\sigma) \subset \fp(\sigma)$ for all Serre weights $\sigma$. 
In particular, $W^g(\rhobar) \subset W(\rhobar)$.
\end{prop}
\begin{proof}
The inclusion is trivial if $\sigma \notin W^g(\rhobar)$.
Suppose that $\sigma \in W^g(\rhobar)$.
Choose an $8$-generic $\rhobar'\in S(\rhobar)$ (e.g.~$\rhobar^{\semis}$) and choose the tame inertial type $\tau$ which is minimal with respect to $\rhobar'$ and $\sigma$. 

For any Serre weight $\sigma' \uparrow \sigma$, $\sigma' \in \JH(\ovl{\sigma}(\tau))$. 
Theorem \ref{thm:Wg} implies that $\JH(\ovl{\sigma}(\tau)) \cap W_{\mathrm{obv}}(\rhobar)$ is nonempty.
Lemma \ref{lemma:obvmodular} implies that $M_\infty(\ovl{\sigma}^\circ(\tau))$ is nonzero for any lattice $\sigma^\circ(\tau) \subset \sigma(\tau)$. 
Since $M_\infty(\ovl{\sigma}^\circ(\tau))$ is a nonzero finitely generated maximal Cohen--Macaulay $\ovl{R}_\infty(\tau)$-module, Proposition \ref{prop:WE} and Lemma \ref{lemma:prodcomp} give 
\[
\prod_{\sigma' \in \JH(\ovl{\sigma}(\tau)) \cap W^?(\rhobar')} \Ann_{R_{\rhobar}} M_\infty(\sigma') \subset \Ann_{R_{\rhobar}} M_\infty(\ovl{\sigma}^\circ(\tau)) = I(\tau) +(\varpi)\subset \fp(\sigma).
\]
Then $\Ann_{R_{\rhobar}} M_\infty(\sigma') \subset \fp(\sigma)$ for some $\sigma' \in \JH(\ovl{\sigma}(\tau)) \cap W^?(\rhobar')$. 
Lemma \ref{lemma:supportcyclebound} implies that $\sigma \uparrow \sigma'$. 
That $\tau$ is minimal with respect to $\rhobar'$ and $\sigma$ implies that $\sigma' \uparrow \sigma$, hence $\sigma = \sigma'$.
\end{proof}

\begin{prop}\label{prop:minimalcycle}
Let $\rhobar: W_{\Q_p} \ra \, ^L \un{G}(\F)$ be an $8$-generic $L$-homomorphism and let $M_\infty$ be a weak minimal patching functor. 
Assume that $W_{\obv}(\rhobar)\cap W(\rhobar)$ is nonempty.
Then $Z(M_\infty(\sigma))$ is the irreducible or zero cycle corresponding to the prime or unit ideal $\fp(\sigma)R_\infty$. In particular, $W(\rhobar) = W^g(\rhobar)$.
\end{prop}
\begin{proof} 
Let $\tau$ be a $4$-generic tame inertial type. 
Let $\cC_\sigma(\rhobar)$ be the irreducible or zero cycle corresponding to the ideal $\fp(\sigma) R_\infty$. 
Then 
\[
Z(\ovl{R}_\infty(\tau)) \geq Z(M_\infty(\ovl{\sigma^\circ(\tau)})) = \sum_{\sigma \in \JH(\ovl{\sigma}(\tau))} Z(M_\infty(\sigma)) \geq \sum_{\sigma \in \JH(\ovl{\sigma}(\tau))} \cC_\sigma(\rhobar),
\]
where the first inequality follows from the fact that $M_\infty$ is minimal (see \cite[{Proposition 7.14}]{LLLM}) and the second inequality follows from Proposition \ref{prop:supportcyclebound}. 
However the first and last expression are equal by Theorem \ref{thm:local_model_cmpt}, which forces the inequalities to be equalities. 
We conclude that the result holds for all $\sigma \in \JH(\ovl{\sigma}(\tau))$ for a $4$-generic tame inertial type $\tau$. 
In particular, the result holds for all generic $\sigma$. 

Finally, suppose $\sigma$ is non-generic. 
Then Proposition \ref{prop:gen:8} shows that $Z(M_\infty(\sigma))=0$ and Proposition \ref{prop:gWE} shows that $\sigma \notin W^g(\rhobar)$, so that $\fp(\sigma)R_\infty = R_\infty$.
\end{proof}

\subsection{Arithmetic patched modules}
\label{subsec:AM}

Let $R_\infty$ be as in \S \ref{subsec:WMPF} and set $F_p\defeq \cO_p\otimes_{\Zp}\Qp$.

\begin{defn}
\label{defn:patch}
An \emph{arithmetic $R_\infty[\GL_3(F_p)]$-module} for an $L$-homomorphism $\rhobar: W_{\Qp}\ra{}^L\un{G}(\F)$ is a non-zero $\cO$-module $M_\infty$ with commuting actions of $R_\infty$ and $\GL_3(F_p)$ satisfying the following axioms:
\begin{enumerate}
\item
\label{it:patch:1}
the $R_\infty[\GL_3(\cO_p)]$-action on $M_\infty$ extends to $R_\infty[\![\GL_3(\cO_p)]\!]$ making $M_\infty$ a finitely generated $R_\infty[\![\GL_3(\cO_p)]\!]$-module;
\item
\label{it:patch:2}
The functor $\Hom_{\cO[\![\GL_3(\cO_p)]\!]}(-,M_\infty^\vee)^\vee: \Rep_{\cO}(\GL_3(\cO_p))\ra \Coh(X_{\infty})$, denoted $M_\infty(-)$, is a weak patching functor;
\item
\label{it:patch:3}
the action of $\cH_{\GL_3(\cO_p)}^{\GL_3(F_p)}(\sigma(\tau)) \cong \cH_{\GL_3(\cO_p)}^{\GL_3(F_p)}(\sigma(\tau)^\circ)[1/p]$ on $M_\infty(\sigma(\tau)^\circ)[1/p]$ factors through the composite
\[
\cH_{\GL_3(\cO_p)}^{\GL_3(F_p)}(\sigma(\tau)) \stackrel{\eta_\infty}{\longrightarrow} R_{\rhobar}^{\eta,\tau}[1/p]\longrightarrow R_{\infty}(\tau)[1/p]
\]
where the map $\eta_\infty$ is the map denoted by $\eta$ in \cite[Theorem 4.1]{CEGGPS} except with $r_p$ normalized so that $r_p(\pi) = \mathrm{rec}_p(\pi\otimes |\det|^{(n-1)/2})$.
\end{enumerate}
We say that an arithmetic $R_\infty[\GL_3(F_p)]$-module $M_\infty$ is \emph{minimal} if $M_\infty(-)$ is. 
\end{defn}

Let $I$ be the preimage of $B_0(\F_p)$ under the reduction map $G_0(\Z_p) \ra G_0(\F_p)$.
Let $I_1$ be the (unique) pro-$p$ Sylow subgroup of $I$. 
Let $\chi: I/I_1 \ra \cO^\times$ be a character. 
Let $\theta(\chi)$ be $\ind_I^{G(\cO_p)} \chi$.
If $\chi$ is regular i.e.~$\chi = \chi^s$ implies $s = 1$ for $s \in W(\GL_3^{S_p})$, then $\theta(\chi)[1/p]$ is absolutely irreducible.

\begin{lemma}\label{lemma:cosoc}
If $\chi: I/I_1 \ra \cO^\times$ is a regular character, then the ${G_0(\Z_p)}$-cosocle of $\theta(\chi)$ is isomorphic to the unique Serre weight $\sigma(\chi)$ with $\sigma(\chi)^{I_1} \isom \ovl{\chi}$.
\end{lemma}
\begin{proof}
By Frobenius reciprocity, $\Hom_{{G_0(\Z_p)}}(\theta(\chi),\sigma) \cong \Hom_I(\chi,\sigma) \cong \Hom_I(\ovl{\chi},\sigma^{I_1})$. 
Then \[\Hom_{G(\Z_p)}(\theta(\chi),\sigma) \neq 0\] if and only if $\sigma = \sigma(\chi)$ in which case it is one-dimensional.
\end{proof}

Let $s \in W(\GL_3^{S_p})$ and $\chi^s$ be the character such that $\chi^s(t) = \chi({s}^{-1}t{s})$ for $t \in T_0(\Z_p) \cong T(\cO_p) \cong \prod_{v\in S_p} T(\cO_v)$. 
The representations $\theta(\chi)$ and $\theta(\chi^s)$ are isomorphic. 

\begin{lemma}\label{lemma:image}
Let $\chi: I/I_1 \ra \cO^\times$ be a regular character. 
Fix $s \in W(\GL_3^{S_p})$, and let $\theta(\chi) \ra \ovl{\theta(\chi^s)}$ be a nonzero map which is unique up to scalar. 
Let $I(\chi,s)$ be the image of this map. 
Then $\sigma(\chi) \in \JH(I(\chi,s))$.
\end{lemma}
\begin{proof}
The natural surjection $\theta(\chi) \onto I(\chi,s)$ induces a surjection on ${G_0(\Z_p)}$-cosocles by Lemma \ref{lemma:cosoc}. 
Thus the cosocle of $I(\chi,s)$ is isomorphic to $\sigma(\chi)$. 
\end{proof}

Let $\chi = \otimes_{v\in S_p} \chi_v$ be as above and decompose each $\chi_v$ as $\chi_{v,1} \otimes \chi_{v,2} \otimes \chi_{v,3}$ in the usual way. 
For each $v\in S_p$, let $\tau_v$ (resp.~$\tau_{v,1}$) be the tame inertial type $(\chi_{v,1} \oplus \chi_{v,2} \oplus \chi_{v,3}) \circ \mathrm{Art}_{F_v}^{-1}$ (resp.~$\chi_{v,1} \circ \mathrm{Art}_{F_v}^{-1}$). 
Then letting $\tau \defeq (\tau_v)_{v\in S_p}$ (resp.~$\tau_1 \defeq (\tau_{v,1})_{v\in S_p}$) we have that $\sigma(\tau) = \theta(\chi)[1/p]$ (resp.~$\sigma(\tau_1)$ is the inflation of $\chi_1$ to $\cO_p^\times$).
For each $v\in S_p$, let $U^{\tau_{1,v}}_{\tau_v}$ be the endomorphism defined in \cite[\S 10.1.2]{LGC} so that $U^{\tau_1}_\tau\defeq \prod_{v\in S_p} U^{\tau_{1,v}}_{\tau_v}$ is an endomorphism of $\ind_{\GL_3(\cO_p)}^{\GL_3(F_p)} \theta(\chi)$. 

\begin{lemma}\label{lemma:outer}
With $\chi$ and $\tau$ as above, suppose that $\tau$ is $4$-generic. 
If $\sigma\in \JH(\ovl{\sigma}(\tau))$ is not an outer weight, then $\eta_\infty(U_{\tau_1,\tau})$ vanishes on $\cC_{\sigma}(\rhobar)$.
%
\end{lemma}
\begin{proof}
Up to a unit, for each $v\in S_p$, the image of $\eta_\infty(U_{\tau_{v,1},\tau_v}) \pmod{\varpi} \in \F$ in (the completion of) the second column of Table \ref{table:coord1} is a nonempty product of diagonal elements modulo $v$ by \cite[Corollary 3.7]{DL}. 
One can check that each of these diagonal elements modulo $v$ is contained in each of the ideals in the final column corresponding to $(0,0)$, $(\eps_1,0)$, or $(\eps_2,0)$. 
\end{proof}

\begin{lemma}\label{lemma:supportupperbound}
Let $\rhobar$ be $8$-generic, and let $M_\infty$ be an arithmetic $R_\infty[\GL_3(F_p)]$-module. 
Then $\supp_{R_{\rhobar}} \, M_\infty(\sigma) \subset \cC_\sigma(\rhobar)$. 
In particular, $W(\rhobar) \subset W^g(\rhobar)$. 
\end{lemma}
\begin{proof}
Let $\sigma'$ be a Serre weight such that $\cC_{\sigma'}(\rhobar) \subset \supp_{R_{\rhobar}}M_\infty(\sigma)$. \
We will show that $\sigma' = \sigma$. 
Set $\chi$ to be the Teichm\"uller lift of $\sigma^{I_1}$ i.e.~$\sigma =\sigma(\chi)$. 
Let $\tau$ and $\tau_1$ be defined in terms of $\chi$ as before e.g.~$\sigma(\tau) \cong \theta(\chi)[1/p]$. 
Since $\sigma$ covers $\sigma'$ by Lemma \ref{lemma:supportcyclebound}, $\sigma' \in \JH(\ovl{\sigma}(\tau))$. 
Since the only weight in $\JH_{\mathrm{out}}(\ovl{\sigma}(\tau))$ that $\sigma$ covers is $\sigma$ itself, we conclude that $\sigma' = \sigma$ or $\sigma'$ is not in $\JH_{\mathrm{out}}(\ovl{\sigma}(\tau))$. 

We have
\begin{align*}
\cC_{\sigma'}(\rhobar) &\subset \supp_{R_{\rhobar}}\, M_\infty(\sigma) \\& \subset \supp_{R_{\rhobar}}\, M_\infty(I(\chi,s)) \\&= \supp_{R_{\rhobar}}\,\eta_\infty(U_{\tau_1,\tau}) M_\infty(\theta(\chi)/(\varpi)) \\& \subset \supp_{R_{\rhobar}}\, \eta_\infty(U_{\tau_1,\tau}) \ovl{R}_{\rhobar}^\tau,
\end{align*}
where the second inclusion follows from Lemma \ref{lemma:image} and the equality follows from the fact that $M_\infty(\theta(\chi))/(\eta_\infty(U_{\tau_1,\tau}),\varpi) \cong M_\infty((\theta(\chi^s)\otimes_{\cO} \F)/I(\chi,s))$ where $s_v = (132)$ for all $v\in S_p$ by \cite[(10.1.9)]{LGC} (and using the exactness of $M_\infty(-)$). 
Then $(\eta_\infty(U_{\tau_1,\tau}),\varpi)_{\fp(\sigma')}/(\varpi)_{\fp(\sigma')} \neq 0$. 
Since $\ovl{R}_{\rhobar}^\tau$ is reduced, Lemma \ref{lemma:outer} implies that $\sigma'\in \JH_{\mathrm{out}}(\ovl{\sigma}(\tau))$. 
\end{proof}

\begin{thm}
\label{thm:axiomaticSWC}
Let $\rhobar$ be $8$-generic and $M_\infty$ be an arithmetic $R_\infty[\GL_3(F_p)]$-module. 
For a Serre weight $\sigma$, $\supp_{R_{\rhobar}}\, M_\infty(\sigma) = \cC_\sigma(\rhobar)$. In particular, $W(\rhobar) = W^g(\rhobar)$.
If $M_\infty$ is furthermore minimal, then $Z(M_\infty(\sigma))$ is the irreducible or zero cycle corresponding to the prime or unit ideal $\fp(\sigma)R_\infty$. 
\end{thm}
\begin{proof}
If $\sigma$ is nongeneric, then $\sigma \notin W(\rhobar)$ and $\sigma \notin W^g(\rhobar)$ as in the proof of Proposition \ref{prop:minimalcycle} and the desired equality holds. 
Since $M_\infty$ is nonzero, there is a generic $\sigma \in W(\rhobar)$. 
Choose a $4$-generic tame inertial type $\tau$ such that $\sigma \in \JH(\ovl{\sigma}(\tau))$. 
Then $M_\infty(\sigma(\tau))$ is nonzero and in fact 
\[
\supp_{R_{\rhobar}}\, M_\infty(\ovl{\sigma}(\tau)) = \cup_{\sigma'\in \JH(\ovl{\sigma}(\tau))} \cC_{\sigma'}(\rhobar)
\]
by Theorem \ref{thm:local_model_cmpt}. 
On the other hand, we have
\[
\supp_{R_{\rhobar}}\, M_\infty(\ovl{\sigma}(\tau)) = \cup_{\sigma'\in \JH(\ovl{\sigma}(\tau))} \supp_{R_{\rhobar}}\, M_\infty(\sigma').
\]
Then Lemma \ref{lemma:supportupperbound} implies that $\supp_{R_{\rhobar}}\, M_\infty(\sigma') = \cC_{\sigma'}(\rhobar)$ for all $\sigma'\in \JH(\ovl{\sigma}(\tau))$ (in particular, $\supp_{R_{\rhobar}}\, M_\infty(\sigma) = \cC_{\sigma}(\rhobar)$). 
It is easy to see from \S \ref{sec:reduction} and Theorem \ref{thm:Wg} that $W_{\mathrm{obv}}(\rhobar) \cap \JH(\ovl{\sigma}(\tau))$ is nonempty. 
Combined with the above, $W_{\mathrm{obv}}(\rhobar) \cap W(\rhobar)$ is nonempty. 
By Proposition \ref{prop:supportcyclebound}, $W^g(\rhobar) \subset W(\rhobar)$. 
With Lemma \ref{lemma:supportupperbound}, we have $W^g(\rhobar) = W(\rhobar)$. 
By the above parenthetical, $\supp_{R_{\rhobar}}\, M_\infty(\sigma) = \cC_{\sigma}(\rhobar)$ if $\sigma \in W(\rhobar)$ while it holds trivially otherwise. 

If $M_\infty$ is minimal, then the last part now follows from Proposition \ref{prop:minimalcycle}. 
\end{proof}

\subsection{Cyclicity for patching functors}
\label{sec:cyc}
In this section, we show that certain patched modules for tame types are locally free of rank one over the corresponding local deformation space.
The argument follows closely that of \cite[\S 5.2]{LLLM2}.

Recall from \S \ref{subsub:ILLC} the irreducible smooth $E$-representation $\sigma(\tau)$ attached to a tame inertial $L$-parameter $\tau$.
Given $\sigma\in \JH(\ovl{\sigma}(\tau))$ we write $\sigma(\tau)^\sigma$ for an $\cO$-lattice, unique up to homothety, in $\sigma(\tau)$ with cosocle $\sigma$. 
For an $L$-parameter $\rhobar:G_{\Qp}\ra{}^L\un{G}(\F)$, we write $W^g(\rhobar,\tau)$ for the intersection $W^g(\rhobar)\cap \JH(\ovl{\sigma}(\tau))$.
Throughout this section, we fix an $L$-parameter $\rhobar$ and a weak minimal patching functor $M_\infty$ for $\rhobar$ which comes from an arithmetic $R_\infty[\GL_3(F_p)]$-module.
The main result of this section is the following:

\begin{thm}\label{thm:cyclic}
Suppose that $\rhobar:G_{\Qp}\ra{}^L\un{G}(\F)$ is a $11$-generic $L$-parameter arising from an $\F$-point of $\cX^{\eta,\tau}$ for tame inertial $L$-parameter $\tau$ $($in particular, $\tau$ is $9$-generic$)$ and let $\tld{z}\defeq \tld{w}^*(\rhobar,\tau)$. 
Let $F(\lambda)\in W^g(\rhobar,\tau)$ be a Serre weight such that for all $j\in \cJ$ 
\begin{equation}
\label{up:each:emb}
\lambda_{\pi^{-1}(j)}\in X_1(T) \text{ is in alcove $\tld{w}_h\cdot C_0$ if } \ell(\tld{z}_j)\leq 1.
\end{equation} 
Then $M_\infty(\sigma(\tau)^{F(\lambda)})$ is a free $R_\infty(\tau)$-module of rank $1$. 
\end{thm}

\noindent The proof is similar to the case when $\rhobar$ is semisimple (\cite[Theorem 5.1.1]{LLLM2} with slightly weaker genericity assumptions), and we will indicate the necessary modifications.
First, \cite[Theorem 5.1.1]{LLLM2} relies on a structure theorem for lattices in generic Deligne--Lusztig representations of $\un{G}_0(\Fp)$ (Theorem 4.1.9 in \emph{loc.~cit}.).
The following proposition improves the genericity hypothesis of that result. 
We refer the reader to \emph{loc.~cit}.~for unexplained notation or terminology.
\begin{prop}
\label{prop:lattice:10-gen}
Let $R$ be $R_s(\mu)$ where $\mu-\eta \in \un{C}_0$ is $9$-deep. %
Then the radical filtration of $\ovl{R}^\sigma$ is predicted by the extension graph with respect to $\sigma$, and the graph distance, the radical distance and the saturation distance from $\sigma$ all coincide on $\Gamma(\ovl{R}^\sigma)$.
\end{prop}
\begin{proof}

As we now explain, the proof of \cite[Proposition 4.3.7]{LLLM2} works for $9$-generic $R$ using some minor improvements to genericity hypotheses. 
Replace $R^{\mathrm{expl},\nabla}_{\ovl{\fM},\tld{w}}$ with a suitable completion of $\Big(\cO\big(\tld{U}(\tld{w},\eta,\nabla_{\tau,\infty})\big)\Big)$ and the primes $\fp^{\mathrm{expl}}(\sigma)$ with suitable completions of the primes corresponding to $\tld{\fP}_{\sigma,\tld{w}}$ in Theorem \ref{thm:local_model_cmpt}. 
The results of \cite[\S 3.5, 3.6]{LLLM2} appearing in the proof of \cite[Proposition 4.3.7]{LLLM2} hold for $7$-generic $\rhobar_{\cS}$. 
Indeed, \cite[Theorem 3.5.2]{LLLM2} holds by the same argument using Proposition \ref{prop:WE} in place of \cite[Proposition 3.5.6]{LLLM2}. 
The rest of the results follow from Theorem \ref{thm:local_model_cmpt}. 
In particular, it holds for the $\rhobar_{\cS}$ chosen in the proof of \cite[Proposition 4.3.7]{LLLM2} since $R$ is $9$-generic.  (\cite[Proposition 3.4.5]{LLLM2} holds with $(n-3)$ replaced by $(n-2)$. Indeed \cite[Proposition 3.3.2]{LLL} holds with $m-n$ replaced by $m-n+1$. The proof shows this stronger result \emph{and} that all lowest alcove presentations of $\tau$ are $(m-n)$-generic.). 
All the subsequent statements appearing in \cite[\S 4.3]{LLLM2} then hold for $9$-generic $R$ (note that \cite[Theorem 4.2.16]{LLLM2} holds for $8$-generic $R$). 
\end{proof}

We prove Theorem \ref{thm:cyclic} through a series of lemmas. 
Until the end of the proof of Theorem \ref{thm:cyclic}, fix $\rhobar$ $8$-generic, $\tau$ a tame inertial $L$-parameter such that $\rhobar\in \cX^{\eta,\tau}(\F)$ (in particular $6$-generic), $\tld{z} = \tld{w}^*(\rhobar,\tau)$, $\tld{w} \defeq \tld{z}^*$, and $\lambda$ satisfying \eqref{up:each:emb}. 
We write $\ovl{\sigma}(\tau)^\sigma$ for $\sigma(\tau)^\sigma\otimes_{\cO}\F$ in what follows.

Below, we modify the proofs of \cite[\S 5.1]{LLLM2}. 
We will refer to the following as the usual modifications: we replace $\tau_{\cS}$, $\rhobar_{\cS}$, $W^?(\rhobar_\cS)$, $W^{?}(\rhobar_{\cS},\tau_{\cS})$, $\tld{w}_i$, and $\tld{w}_i^*$ in \emph{loc.~cit}.~by $\tau$, $\rhobar$, $W^g(\rhobar)$, $W^{g}(\rhobar,\tau)$, $\tld{z}_i$, and $\tld{w}_i$, respectively.
(In \cite[\S 5.1]{LLLM2}, the set $S_p$ is denoted $\cS$.)

\begin{lemma} 
\label{lem:4:wgt}
Assume that $\tau$ is $9$-generic $($for instance, if $\rhobar$ is $11$-generic$)$ and $\ell(\widetilde{z}_j) > 1$ for all $j\in\cJ$. 
Let $V$ be a quotient of $\ovl{\sigma}(\tau)^\sigma$.
Then $M_\infty(V)$ is a cyclic $R_\infty(\tau)$-module.
\end{lemma}
\begin{proof}
First, the scheme-theoretic support of $M_\infty(\sigma)$ is (nonempty and) generically reduced by Theorem \ref{thm:axiomaticSWC} and hence reduced e.g.~by the proof of \cite[Lemma 3.6.2]{LLLM2}. 
It is then formally smooth by Table \ref{Table:intsct}, and so $M_\infty(\sigma)$ is free over its scheme-theoretic support by the Auslander--Buchsbaum--Serre theorem and the Auslander--Buchsbaum formula. 

Now the proof of \cite[Lemma 5.1.3]{LLLM2} applies after the usual modifications. 
Moreover, in the setup for \cite[Proposition 4.3.7, Lemma 3.6.10]{LLLM2}, $R^{\mathrm{expl},\nabla}_{\ovl{\fM},\tld{w}}$ should be replaced by $\Big(\cO\big(\tld{U}(\tld{z},\eta,\nabla_{\tau,\infty})\big)\Big)^\wedge_{x}$ (for a suitable $x\in \tld{U}(\tld{z},\eta,\nabla_{\tau,\infty})(\F)$) and $N$ by $\log_2 \# W^g(\rhobar,\tau)$. 
(\cite[Lemma 3.6.10]{LLLM2} holds for $\rhobar$ and $\tau$ with $W^?(\rhobar,\tau)$ replaced by $W^g(\rhobar,\tau)$ by Theorem \ref{thm:local_model_cmpt}.) 
\end{proof}

We now assume that $\rhobar$ is $11$-generic (in particular, $\tau$ is $9$-generic). 
We fix a semisimple $\rhobar^{\speci}: G_{\Qp}\ra{}^{L}\un{G}(\F)$ such that $\tld{w}^*(\rhobar^\speci,\tau) = \tld{w}^*(\rhobar,\tau) = \tld{z}$. 
By Corollary \ref{cor:semicont}, if $\rhobar \in \cX^{\eta,\tau'}(\F)$ for a $4$-generic $\tau'$, then $\rhobar^{\speci} \in \cX^{\eta,\tau'}(\F)$ and
\begin{equation}
\label{eq:sc}
\tld{w}^{*}(\rhobar^{\speci},\tau') \leq \tld{w}^{*}(\rhobar,\tau').
\end{equation}
Let $(s,\mu-\eta)$ be the $7$-generic lowest alcove presentation for $\rhobar^\speci$ compatible with the implicit $9$-generic lowest alcove presentation of $\tau$ so that $\rhobar^\speci|_{I_{\Q_p}} \cong \taubar(s,\mu)$. 

Now let $V$ be a quotient of $\ovl{\sigma}(\tau)^\sigma$ such that there exist subsets $\Sigma_{V,j}\subseteq \tld{w}_j^{-1}(\Sigma_0)$ such that
\begin{align*}
\label{it:bij:cyclic}
\prod_{i\in\cJ} \Sigma_{V,i}&\stackrel{\sim}{\longrightarrow} \JH(V)\\
(\omega,a)&\mapsto \sigma_{(\omega,a)}\defeq F(\Trns_{\mu}(s\omega,a))
\end{align*}
is a bijection. 
We will show the cyclicity of $M_\infty(V)$ by inducting on the complexity of the set $W^g(\rhobar,\tau) \cap \JH(V)$. 

Let $\Sigma^g_j \subset r(\Sigma_0)$ such that $(\omega,a) \mapsto F(\Trns_{\mu}(s\omega,a))$ defines a bijection from $\Sigma^g_j \ra W^g(\rhobar)$. (The sets $\Sigma^g_j$ exist by \S \ref{sec:reduction}, Corollary \ref{cor:g=?}, Theorem \ref{thm:Wg}, and Corollary \ref{cor:semicont}.)

\begin{lemma}\label{lemma:broom}
Suppose that for all $j\in\cJ$ either $\ell(\tld{z}_j) > 1$ or $\Sigma_{V,j} \subset \{(\eps,1),\,(0,0),\,(\eps_1,0),\,(\eps_2,0)\}$ for some $\eps\in\{0,\, \eps_1,\,\eps_2\}$.
Then $M_\infty(V)$ is a cyclic $R_\infty(\tau)$-module.

\end{lemma}
\begin{proof}
We induct on 
\[
n\defeq \#\big\{i\in\cJ\ :\  \ell(\tld{z}_{i})\leq 1\text{ and }\#\ovl{\Sigma}_{V,i}=3 \big\}
\]
If $n=0$ we let $\tau'$ be the tame inertial $L$-parameter corresponding to $\tau'_{\cS}$ constructed in the first paragraph of \cite[Lemma 5.1.4]{LLLM2} with respect to $\rhobar^{\speci}$.
Since $\sigma \in \JH(\ovl{\sigma}(\tau))$ by construction, $\cX^{\eta,\tau'}(\F)$ contains the $11$-generic $\rhobar$ so that $\tau'$ is $9$-generic. 
Then $\rhobar$ arises from a point in $\tld{U}(\tld{z}',\eta,\nabla_{\tau',\infty})(\F)$ by Theorem \ref{thm:local_model_cmpt} since $\sigma \in \JH(\ovl{\sigma}(\tau'))$ and $\ell(\tld{z}'_j) > 1$ by \eqref{eq:sc}. 
The result follows now from Lemma \ref{lem:4:wgt}.

Observe that if $\ell(\tld{z}_j) \leq 1$, then $\lambda_{\pi^{-1}(j)} \in \tld{w}_h \cdot C_0$ so that $(\eps,1)$ is necessarily in $\Sigma^g_j$ since $F(\lambda) \in W^g(\rhobar) \cap \JH(V)$. 
The general case then follows verbatim as in the proof of the general case in \emph{loc.~cit}.~with the usual modifications. 
Moreover, references to \cite[Theorem 3.6.4, Table 3, Theorem 4.1.9]{LLLM2} are replaced by references to Theorem \ref{thm:local_model_cmpt}, Table \ref{Table:intsct}, and Proposition \ref{prop:lattice:10-gen} respectively; and references to \cite[Lemma 3.6.12, Lemma 3.6.16(3.19)]{LLLM2} are replaced by references to Lemma \ref{lem:ideal:relation:1} and \ref{lem:ideal:relation:2} after localization at $x$ (in fact \cite[Lemma 3.6.12]{LLLM2} is sufficient for the case of $\tld{z}_j = t_{\un{1}}$). 
\end{proof}

\begin{lemma} \label{lemma:6weight}
Suppose that $\ell(\widetilde{z}_j) > 1$ or 
\[\Sigma_{V,j} \subset \tld{w}_j^{-1} (\Sigma_0 \setminus \{(\nu_1,0),(\nu_2,1),(\nu_3,0)\})\]
where $(\nu_1,\nu_2,\nu_3)$ is $(\eps_1-\eps_2,\eps_1, \eps_1+\eps_2)$, $(\eps_2-\eps_1, \eps_2, \eps_1+\eps_2)$, or $(\eps_1-\eps_2,0,\eps_2-\eps_1)$.
Then $M_\infty(V)$ is a cyclic $R_\infty(\tau)$-module.
\end{lemma}
\begin{proof}
This follows from the proof of \cite[Lemma 5.1.5]{LLLM2} with the usual modifications. 
(In the reduction step in the first paragraph of the proof, one possibly changes $\tau$ and so possibly changes $\rhobar^{\speci}$. This only affects this proof.) 
References to \cite[Theorem 4.1.9]{LLLM2} are replaced by Proposition \ref{prop:lattice:10-gen} above, and references to \cite[Lemma 5.1.4]{LLLM2} are replaced by references to Lemma \ref{lemma:broom}. 

Finally, we can and do choose $V^2$ in the final paragraph of the proof of \cite[Lemma 5.1.5]{LLLM2} so that if $(\eps',0) \in V^2_{i'}$ in the notation of \emph{loc.~cit}., then $(\eps',0) \in \Sigma^g_{i'}$. 
Indeed, $\ell(\tld{z}_{i'}) \leq 1$ and \cite[\S 8]{LLLM} ensure that $\tld{U}(\tld{z}_{i'}, \eta_{i'},\nabla_{s_{i'}^{-1}(\mu_{i'}+\eta_{i'})})_{\F}$ has at least $5$ components, where $\tau = \tau(s,\mu+\eta)$, so that $\# \Sigma^g_{i'} \geq 6$ and contains two of $\tld{w}_{i'}^{-1}((0,0),(\eps_1,0),(\eps_2,0))$ by Theorem \ref{thm:Wg}. 
Then by Theorem \ref{thm:axiomaticSWC}, we can apply \cite[Lemma 10.1.13]{EGS} as described in \emph{loc.~cit}. 
\end{proof}

\begin{rmk} There was a gap in the proof of \cite[Lemma 5.1.5]{LLLM2}: in the proof of \emph{loc.~cit}.~one needs to possibly change the type $\tau$ to an auxiliary type, which may cause a loss of $2$ in the genericity. Since we need to apply \cite[Theorem 4.1.9]{LLLM2} to this auxiliary type, one needs to increase the genericity assumption by $2$ in \cite[Theorem 5.1.1]{LLLM2}.
\end{rmk}
\begin{lemma}\label{lemma:8weight}
Suppose that $\ell(\widetilde{z}_j) > 1$ or $\Sigma_{V,j} \subset \tld{w}_j^{-1} (\Sigma_0 \setminus \{(\nu,0)\})$ where $\nu$ is $\eps_1-\eps_2,\eps_2-\eps_1,$ or $\eps_1+\eps_2$.
Then $M_\infty(V)$ is a cyclic $R_\infty(\tau)$-module.
\end{lemma}
\begin{proof}
This follows from the proof of \cite[Lemma 5.1.6]{LLLM2} with the usual modifications and using Lemma \ref{lemma:6weight} and Lemmas \ref{lem:ideal:relation:3} and \ref{lem:ideal:relation:4} below (completed at $x$). 
\end{proof}

\begin{lemma}\label{lemma:9weight}
With $V$ as described before Lemma \ref{lemma:broom}, $M_\infty(V)$ is a cyclic $R_\infty(\tau)$-module.
\end{lemma}
\begin{proof}
The argument in the proof of \cite[Lemma 5.1.7]{LLLM2} holds verbatim with the usual modifications and the reference to \cite[Lemma 5.1.6]{LLLM2} replaced by a reference to Lemma \ref{lemma:8weight}. 
\end{proof}

\begin{proof}[Proof of Theorem \ref{thm:cyclic}]
Theorem \ref{thm:cyclic} follows from the proof of \cite[Theorem 5.1.1]{LLLM2} using Lemma \ref{lemma:9weight} in place of \cite[Lemma 5.1.7]{LLLM2}.
\end{proof}

In the lemmas below we refer the reader to \S \ref{sec:Appendix} for unexplained notation.
These lemmas are algebraizations of \cite[Lemmas 3.6.12,~3.6.14,~3.6.16(3.19),~3.6.16~(3.17) and (3.18)]{LLLM2}. 
Their proof follows verbatim in our setting by replacing $\ovl{R}^{\mathrm{expl},\nabla}_{\ovl{\fM},\tld{z}}$ and the ideals $\mathfrak c_{(\omega,a)}$ of \emph{loc.~cit}.~with $\tld{U}(\tld{z}, \eta_j,\nabla_{s_j^{-1}(\mu_j+\eta_j)})_{\F}$ (with $\mu_j\in X^*(T)$ $4$-deep) and the ideals $\tld{\fP}_{(\omega,a),\alpha t_{\un{1}}}$ respectively.
(The second displayed equation in the statement of the Lemma \ref{lem:ideal:relation:4} is not covered by \cite[Lemma 3.6.16(3.17)]{LLLM2}, but the proof is analogous.)
Alternatively, one observes that all the ideal equalities we need to verify can be checked after projecting $\tld{U}(\tld{z}_j, \eta_j,\nabla_{s_j^{-1}(\mu_j+\eta_j)})_{\F}$ to $\Fl$, where there is a contracting $T^\vee_{\F}$-action with unique fixed point $\tld{z}_j$.
Since all the ideals involved are $T^\vee_{\F}$-equivariant one only needs to check the equalities after completion at $\tld{z}_j$, which are exactly the results of \cite[\S 3.6.3]{LLLM2}.

\begin{lemma}[Lemma 3.6.12\cite{LLLM2}]
\label{lem:ideal:relation:1}
In $\tld{U}(t_{\un{1}}, \eta_j,\nabla_{s_j^{-1}(\mu_j+\eta_j)})_{\F}$, we have the following ideal relation:
\begin{equation*}
\Big(\tld{\fP}_{(0,0),t_{\un{1}}}\cap \tld{\fP}_{(0,1),t_{\un{1}}}\cap \tld{\fP}_{(\eps_1,0),t_{\un{1}}}\Big)+\Big(\tld{\fP}_{(0,0),t_{\un{1}}}\cap \tld{\fP}_{(0,1),t_{\un{1}}}\cap \tld{\fP}_{(\eps_2,0),t_{\un{1}}}\Big)=\Big(\tld{\fP}_{(0,0),t_{\un{1}}}\cap \tld{\fP}_{(0,1),t_{\un{1}}}\Big)
\end{equation*}
\end{lemma}
\begin{lemma}[Lemma 3.6.16(3.19)\cite{LLLM2}]
\label{lem:ideal:relation:2}
In $\tld{U}(\alpha t_{\un{1}}, \eta_j,\nabla_{s_j^{-1}(\mu_j+\eta_j)})_{\F}$, we have the following ideal relation:
\begin{equation*}
\Big(\tld{\fP}_{(0,1),\alpha t_{\un{1}}}\cap \tld{\fP}_{(0,0),\alpha t_{\un{1}}}\cap \tld{\fP}_{(\eps_2,0),\alpha t_{\un{1}}}\Big)+\Big(\tld{\fP}_{(0,1),\alpha t_{\un{1}}}\cap \tld{\fP}_{(0,0),\alpha t_{\un{1}}}\cap \tld{\fP}_{(\eps_2-\eps_1,0),\alpha t_{\un{1}}}\Big)=\Big(\tld{\fP}_{(0,1),\alpha t_{\un{1}}}\cap \tld{\fP}_{(0,0),\alpha t_{\un{1}}}\Big)
\end{equation*}
\end{lemma}
\begin{lemma}[Lemma 3.6.14 \cite{LLLM2}]
\label{lem:ideal:relation:3}
 In $\tld{U}(t_{\un{1}}, \eta_j,\nabla_{s_j^{-1}(\mu_j+\eta_j)})_{\F}$, we have the following ideal relation:
\begin{align*}
&\Big(\tld{\fP}_{(0,0),t_{\un{1}}}\cap \tld{\fP}_{(0,1),t_{\un{1}}}\cap \tld{\fP}_{(\eps_1,0),t_{\un{1}}}\cap\tld{\fP}_{(\eps_1,1),t_{\un{1}}}\cap \tld{\fP}_{(\eps_2,0),t_{\un{1}}}\Big)+\\
&\qquad+\Big(\tld{\fP}_{(0,0),t_{\un{1}}}\cap \tld{\fP}_{(0,1),t_{\un{1}}}\cap \tld{\fP}_{(\eps_2,0),t_{\un{1}}}\cap\tld{\fP}_{(\eps_2,1),t_{\un{1}}}\cap \tld{\fP}_{(\eps_1,0),t_{\un{1}}}\Big)=\Big(\tld{\fP}_{(0,0),t_{\un{1}}}\cap \tld{\fP}_{(0,1),t_{\un{1}}}\cap\tld{\fP}_{(\eps_1,0),t_{\un{1}}}\cap \tld{\fP}_{(\eps_2,0),t_{\un{1}}}\Big)
\end{align*}
\end{lemma}
\begin{lemma}[Lemma 3.6.16(3.17),(3.18)\cite{LLLM2}]
\label{lem:ideal:relation:4}
 In $\tld{U}(\alpha t_{\un{1}}, \eta_j,\nabla_{s_j^{-1}(\mu_j+\eta_j)})_{\F}$, we have the following ideal relations:
\begin{align*}
&\Big(\tld{\fP}_{(0,1),\alpha t_{\un{1}}}\cap \tld{\fP}_{(0,0),\alpha t_{\un{1}}}\cap \tld{\fP}_{(\eps_2,0),\alpha t_{\un{1}}}\cap \tld{\fP}_{(\eps_2-\eps_1,0),\alpha t_{\un{1}}}\cap \tld{\fP}_{(\eps_2,1),\alpha t_{\un{1}}}\Big)+\\&\qquad+\Big(\tld{\fP}_{(0,1),\alpha t_{\un{1}}}\cap \tld{\fP}_{(0,0),\alpha t_{\un{1}}}\cap \tld{\fP}_{(\eps_2,0),\alpha t_{\un{1}}}\cap \tld{\fP}_{(\eps_1,1),\alpha t_{\un{1}}}\Big)=\Big(\tld{\fP}_{(0,1),\alpha t_{\un{1}}}\cap \tld{\fP}_{(0,0),\alpha t_{\un{1}}}\cap \tld{\fP}_{(\eps_2,0),\alpha t_{\un{1}}}
\Big)\\
\\
&\Big(\tld{\fP}_{(\eps_2,1),\alpha t_{\un{1}}}\cap \tld{\fP}_{(\eps_2,0),\alpha t_{\un{1}}}\cap \tld{\fP}_{(0,0),\alpha t_{\un{1}}}\cap \tld{\fP}_{(\eps_2-\eps_1,0),\alpha t_{\un{1}}}\cap \tld{\fP}_{(0,1),\alpha t_{\un{1}}}\Big)+\\&\qquad\Big(\tld{\fP}_{(\eps_2,1),\alpha t_{\un{1}}}\cap \tld{\fP}_{(\eps_2,0),\alpha t_{\un{1}}}\cap \tld{\fP}_{(0,0),\alpha t_{\un{1}}}\cap \tld{\fP}_{(\eps_1,1),\alpha t_{\un{1}}}\Big)=\Big(\tld{\fP}_{(\eps_2,1),\alpha t_{\un{1}}}\cap \tld{\fP}_{(\eps_2,0),\alpha t_{\un{1}}}\cap \tld{\fP}_{(0,0),\alpha t_{\un{1}}}
\Big)
\end{align*}
and
\begin{align*}
&\Big(\tld{\fP}_{(\eps_1,1),\alpha t_{\un{1}}}\cap \tld{\fP}_{(0,1),\alpha t_{\un{1}}}\cap \tld{\fP}_{(\eps_2,0),\alpha t_{\un{1}}}\cap \tld{\fP}_{(0,1),\alpha t_{\un{1}}}\Big)+\\&\qquad+\Big(\tld{\fP}_{(\eps_1,1),\alpha t_{\un{1}}}\cap \tld{\fP}_{(0,1),\alpha t_{\un{1}}}\cap \tld{\fP}_{(\eps_2,0),\alpha t_{\un{1}}}\cap \tld{\fP}_{(\eps_2,1),\alpha t_{\un{1}}}\Big)=\Big(\tld{\fP}_{(\eps_1,1),\alpha t_{\un{1}}}\cap \tld{\fP}_{(0,1),\alpha t_{\un{1}}}\cap \tld{\fP}_{(\eps_2,0),\alpha t_{\un{1}}}\Big)
\end{align*}
\end{lemma}

\subsubsection{Gauges for patching functors}
\label{sec:gauge}

Let $\tau$ be a tame inertial $L$-parameter $\tau$ and $\rhobar:G_{\Qp}\ra{}^L\un{G}(\F)$ be an $11$-generic $L$-homomorphism arising from an $\F$-point $\cX^{\eta,\tau}(\F)$. 
Let $M_\infty$ be a weak minimal patching functor coming from an arithmetic $R_\infty[\GL_3(F_p)]$-module. 

The scheme $X_\infty(\tau)$ is normal and Cohen--Macaulay by Theorem \ref{thm:local_model_main} and Remark \ref{rmk:CM}.
We let $Z \subset X_{\infty}(\tau)$ be the locus of points lying on two irreducible components of the special fiber of $X_{\infty}(\tau)$ (in particular $Z\subset X_{\infty}(\tau)$ has codimension at least two) and
\begin{equation*}
j: U \defeq X_{\infty}(\tau) \setminus Z \into X_{\infty}(\tau)
\end{equation*}
be the natural open immersion.

\begin{thm}\label{thm:gauge}
Let $\sigma,\,\kappa\in \JH(\ovl{\sigma}(\tau))$ and let $\iota: \sigma(\tau)^\kappa \into \sigma(\tau)^\sigma$ be a saturated injection.
For any $\theta \in W^g(\rhobar)$ let $m(\theta)$ be the multiplicity with which $\theta$ appears in the cokernel of $\iota$.
Suppose further that $\sigma$ is as in Theorem \ref{thm:cyclic}.
Then the induced injection $M_\infty(\iota):  M_\infty(\sigma(\tau)^\kappa) \into M_\infty(\sigma(\tau)^\sigma)$ has image 
\[j_*j^* \Big(\prod_{\theta\in W^g(\rhobar)} \mathfrak{p}(\theta)^{m(\theta)}R_\infty(\tau)\Big) M_\infty(\sigma(\tau)^\sigma).\]
where $\mathfrak{p}(\theta)$ is the minimal prime ideal of $(R_{\rhobar}^{\tau})_{\F}$ corresponding to $\theta$ via Theorem \ref{thm:local_model_cmpt} and localization at $\rhobar\in \cX_{{\F}}^{\eta,\tau}$.
\end{thm}
\begin{proof}
The proof follows verbatim the argument of \cite[Lemmas 5.2.1 and 5.2.2, Theorem 5.2.3]{LLLM2} after replacing occurrences of $W^{?}(\rhobar_{\cS})$ in \emph{loc.~cit.}~with $W^g(\rhobar)$.
\end{proof}

\subsection{Global applications}
\label{subsec:global}

In this section, we deduce global applications of Theorems \ref{thm:axiomaticSWC}, \ref{thm:cyclic}, and \ref{thm:gauge} generalizing results of \cite[\S 5.3]{LLLM2} in the tamely ramified setting. 
Let $F^+$ be a totally real field, $S_p$ the set of places of $F^+$ dividing $p$, and $F/F^+$ a CM extension. 
We assume that all places of $S_p$ are unramified over $\Q$ and split in $F$.
We start with the following modularity lifting result. 

\begin{thm} \label{thm:modularity_lifting}
Let $F/F^+$ be a CM extension, and let $r: G_F \ra \GL_3(E)$ be a continuous representation such that 
\begin{itemize}
\item $r$ is unramified at all but finitely many places;
\item $r$ is potentially crystalline at places dividing $p$ of type $(\eta,\tau)$ where $\tau$ is a tame inertial type that admits a lowest alcove presentation $(s,\mu)$ with $\mu$ $4$-deep in alcove $\un{C}_0$;
\item $r^c \cong r^\vee \varepsilon^{-2}$;
\item $\zeta_p\notin \ovl{F}^{\ker\ad\rbar}$ and $\rbar(G_{F(\zeta_p)}) \subset \GL_3(\F)$ is an adequate subgroup; and
\item $\rbar \cong \rbar_\iota(\pi)$ for some $\pi$ a regular algebraic conjugate self-dual cuspidal (RACSDC) automorphic representation of $\GL_3(\A_F)$ of weight $0$ so that $\sigma(\tau)$ is a $K$-type for $\pi$ at places dividing $p$.
\end{itemize}
Then $r$ is automorphic i.e.~$r \cong r_\iota(\pi')$ for some $\pi'$ a RACSDC automorphic representation of $\GL_3(\A_F)$.%
\end{thm}
\begin{proof}
This follows from standard base change and Taylor--Wiles patching arguments using Corollary \ref{thm:local_model_main}, cf.~\cite[Theorem 9.2.1]{MLM}, \cite[Theorem 7.4]{LLLM}, and its addendum \cite[\S 6]{LLLM2}.
\end{proof}

We now suppose that $F^+\neq \Q$. 
Let $\cO_{F^+,p}\defeq \cO_{F^{+}}\otimes_{\Z}\Zp$ be the finite \'etale $\Zp$-algebra denoted $\cO_p$ in \S \ref{subsec:WMPF}, \ref{subsec:AM}.
We fix an outer form $G_{/F^+}$ of $\GL_n$ which splits over $F$ and is definite at all archimedean places of $F^+$.
There exists $N\in \N$, a reductive model $\cG$ of $G$ defined over $\cO_{F^+}[1/N]$, and an isomorphism
\begin{equation}
\label{iso integral}
\iota:\,\cG_{/\cO_{F}[1/N]} \stackrel{\iota}{\rightarrow}{\GL_3}_{/\cO_{F}[1/N]}
\end{equation}
(cf.~\cite[\S 7.1]{EGH}). 
Given a subset $\cP$ of finite complement in the set of finite places of $F^+$ which split in $F$ and are coprime to $pN$, let $\mathbb{T}_{\cP}$ be the universal Hecke algebra $\mathbb{T}_{\cP}$ for places in $\cP$ (see \cite[\S 9.1]{MLM}). 
Given a compact open $U=U^pU_p\leq G(\mathbb{A}_{F^+}^{\infty,p})\times \cG(\cO_{F^+,p})$ and a finite $\cO$-module $W$ with a continuous $U$-action, let $S(U,W)$ be the space of algebraic automorphic forms of level $U$ and coefficients $W$, as in \cite[(9.2)]{MLM}. 
If $U$ is unramified at places in $\cP$, then $S(U,W)$ has a natural $\mathbb{T}_{\cP}$-action. 
Let $\mathbb{T}_{\cP}(U,W)$ be the quotient of $\mathbb{T}_{\cP}$ acting faithfully on $S(U,W)$.

Let $\cG_3$ be the group scheme over $\mathbb Z$ defined in \cite[\S 2.1]{CHT}.
We consider a continuous Galois representation $\rbar: G_{F^+} \ra \cG_3(\F)$ which is \emph{automorphic} in the sense of \cite{MLM}, i.e.~for which there exists a maximal ideal $\fm\subseteq \bT_{\cP}(U,W)$, for some level $U$ and coefficients $W$ satisfying
\[
\det\left(1-\rbar(\mathrm{Frob}_w)X\right)=\sum_{j=0}^2 (-1)^j(\mathbf{N}_{F/\Q}(w))^{\binom{j}{2}}T_w^{(j)}X^j\mod{\fm}
\]
for all $w\in \cP$.
Note that the collection $(\rbar|_{G_{F^+_v}})_{v\in S_p}$ defines an $\F$-valued $L$-parameter, which will be denoted as $\rbar_p$ in what follows.
For such $\rbar$, we define as in \cite[Definition 9.1.1]{MLM} the set $W(\rbar)$ of \emph{modular Serre weights for $\rbar$}.

\begin{thm}\label{thm:SWC}
Let $\rbar: G_{F^+} \ra \cG_3(\F)$ be an automorphic Galois representation.
Assume further that 
\begin{itemize}
\item $\rbar|_{G_F}(G_{F(\zeta_p)})$ is adequate; and
\item $\rbar_p$ is $8$-generic.
\end{itemize}
Then 
\[
W(\rbar) = W^{g}(\rbar_p).
\]
\end{thm}
\begin{proof}
The proof of \cite[Theorem 5.3.3]{LLLM2} applies verbatim after replacing Theorem 3.5.2 of \emph{loc.~cit}.~with Theorem \ref{thm:axiomaticSWC} above.
\end{proof}

\subsubsection{Mod $p$ multiplicity one}\label{subsec:multone}

We continue using the setup from \S \ref{subsec:global}.
We assume further that $F/F^+$ is unramified at all finite places.
We now let $S_0$ denote the set of finite places of $F^+$ away from $p$ where $\rbar$ ramifies and assume that every place of $S_0$ splits in $F$.
For each $v\in S_0$, with fixed lift $\tld{v}$ in $F$, we let $\tau_{\tld{v}}$ be the minimally ramified type in the sense of \cite[Definition 2.4.14]{CHT} corresponding to $\rbar|_{G_{F_{\tld{v}}}}:G_{F_{\tld{v}}}\ra\GL_3(\F)$ and $\sigma(\tau_v)\defeq \sigma(\tau_{\tld{v}})\circ\iota_{\tld{v}}$ be the $\cG(\cO_{F^+_{v}})$-representation attached to it (where $\iota_{\tld{v}}$ is the localization at ${\tld{v}}$ of the isomorphism (\ref{iso integral}); $\sigma(\tau_v)$ is independent of the choice of $\tld{v}|v$, cf.~\cite[\S 5.3]{LLLM2}).
Fix an $\cO$-lattice $W_{S_0}$ in ${\otimes}_{v\in S_0}\sigma(\tau_v)$.
We have the following mod $p$ multiplicity one result.
\begin{thm}\label{thm:modpmultone}
Let $\rbar: G_{F^{+}}\ra\cG_3(\F)$ be a continuous Galois representation such that $\rbar_p$ is $11$-generic.  Let $\tau$ and $F(\lambda)\in W^g(\rbar_p,\tau)$ be as in the statement of Theorem \ref{thm:cyclic}.
Assume moreover that:
\begin{itemize}
\item $\rbar: G_{F^{+}}\ra\cG_3(\F)$ is automorphic;
\item $\rbar|_{G_F}(G_{F(\zeta_p)})$ is adequate; and
\item the places at which $\rbar$ ramifies split in $F$.
\end{itemize}
Then 
\[
S\Big(U^p\cG(\cO_{{F^+,p}}),\big(\ovl{\sigma}(\tau)^{F(\lambda)} \circ \prod_{v\in {S_p}} \iota_{\tld{v}}\big)^\vee \otimes_{\cO}W_{S_0}\Big)[\mathfrak{m}]
\]
is one-dimensional over $\F$, where $\mathfrak{m}$ is the maximal ideal in the Hecke algebra $\mathbb{T}_{\cP}$ corresponding to $\rbar$.
\end{thm}
\begin{proof}
The proof of \cite[Theorem 5.3.4]{LLLM2} applies verbatim after replacing the reference to \cite[Theorem 5.2.1]{LLLM2} by a reference to Theorem \ref{thm:cyclic} above.
\end{proof}

\subsubsection{Breuil's lattice conjecture} 
We now consider an automorphic Galois representation $r:G_F \ra \GL_3(E)$ as in Theorem \ref{thm:modularity_lifting} which is \emph{minimally ramified}, i.e.~for any place $\tld{v}$ of $F$ lying above some $v \in S_0$, the Galois representation $r|_{G_{F_{\tld{v}}}}$ is minimally ramified in the sense of \cite[Definition 2.4.14]{CHT}. 
Let $\lambda$ be the kernel of the system of Hecke eigenvalues $\alpha:\bT_{\cP}\rightarrow \cO$ associated to $r$, i.e.~ $\alpha$ satisfies 
\[
\det\left(1-r^{\vee}(\mathrm{Frob}_w)X\right)=\sum_{j=0}^3 (-1)^j(\mathbf{N}_{F/\Q}(w))^{\binom{j}{2}}\alpha(T_w^{(j)})X^j
\]
for all $w\in\cP$. 
For $U^p\leq G(\mathbb{A}_{F^+}^{\infty,p})$ and a finite $\cO$-module $W$ with a continuous $U^p$-action, let 
\[S(U^p, W) \defeq \varinjlim_{U_{p}} S(U^pU_{p}, W) \textrm{ and } \tld{S}(U^p, W) \defeq \varprojlim_{s\in\N} S(U^p,W/\varpi^s)\]
where $U_{p}$ runs over the compact open subgroups of $\cG(\cO_{F^+,p})$. 
By Theorem \ref{thm:modularity_lifting}, $\tld{S}(U^p, W_{S_0})[\lambda]$ is nonzero. 

\begin{thm}\label{thm:lattice}
Let $r:G_F \ra \GL_3(E)$ and $\tau$ be as in Theorem \ref{thm:modularity_lifting}. 
Assume furthermore that $r$ is minimally ramified and that the places at which $\rbar$ ramifies are split in $F$. 
Finally, assume that $\rbar_p$ is $11$-generic.
Then the lattice $\sigma(\tau) \cap \tld{S}(U^p, W_{S_0})[\lambda] \subset \sigma(\tau) \cap \tld{S}(U^p,W_{S_0})[\lambda]\otimes_{\cO} E $ depends only on $r_p$.
\end{thm}
\begin{proof}
The proof of \cite[Theorem 5.3.5]{LLLM2} applies verbatim after replacing occurrences of $W^?(\rbar_{\cS})$ in \emph{loc.~cit.}~with $W^g(\rbar_p)$ and the reference to \cite[Theorem 5.2.3]{LLLM2} with a reference to Theorem \ref{thm:gauge}.
\end{proof}

\section{Appendix: tables}
\label{sec:Appendix}

In the following tables we write $\alpha$, $\beta$ and $\gamma$ for the elements of $\tld{W}$ corresponding to $(12)$, $(23)$ and $w_0t_{(1,0,-1)}$ respectively.
Moreover, the image of $\un{1}\defeq (1,1,1)\in X^*(T)$ in $\tld{W}$ is denoted as $t_{\un{1}}$.
We identify the elements above with matrices in $\GL_3(\Z(\!(v)\!))$ via the embedding $\tld{W}\into \GL_3(\Z(\!(v)\!))$ defined by 
$\alpha\mapsto \text{\tiny{$\begin{pmatrix}0&1&0\\1&0&0\\0&0&1\end{pmatrix}$}}$, $\beta\mapsto \text{\tiny{$\begin{pmatrix}1&0&0\\0&0&1\\0&1&0\end{pmatrix}$}}$, $\gamma\mapsto \text{\tiny{$\begin{pmatrix}0&0&v^{-1}\\0&1&0\\v&0&0\end{pmatrix}$}}$ and $t_{\un{1}}\mapsto \text{\tiny{$\begin{pmatrix}v&0&0\\0&v&0\\0&0&v\end{pmatrix}$}}$.
\begin{table}[H]
\caption{\textbf{The relevant $\cO$-algebras for Proposition \ref{prop:loc:mod:diag:1}}}
\label{table:coord1}
\centering
\adjustbox{max width=\textwidth}{
\begin{tabular}{| c | c | c |}
\hline
& &\\
$\widetilde{z}_jt_{-\un{1}}$ & $A^{(j)}$ &
$\cO(\tld{U}(\tld{z}_j, \leq\!\eta_j))$
\\ 
& &\\
\hline
& &\\
$\alpha \beta \alpha \gamma$ & $\begin{pmatrix} (v + p)^2 c_{11}^* & 0 & 0 \\ v(v+p) c_{21} & (v + p) c_{22}^* & 0 \\ v (c_{31}+ (v+p) c'_{31}) & v c_{32} & c_{33}^*\end{pmatrix}$ &$
\cO[c_{11}^*,c_{21},c_{31},c'_{31},c_{22}^*,c_{32},c_{33}^*]$\\
& &\\
\hline
& &\\
$\beta \gamma \alpha \gamma$ &  $\begin{pmatrix} (v + p) c_{11}^* & (v + p)c_{12} & 0 \\ 0 & (v+p)^2 c_{22}^* & 0 \\ v c_{31} & v (c_{32}+(v+p) c'_{32})& c_{33}^*\end{pmatrix}$&$
\cO[c_{11}^*,c_{12},c_{22}^*, c_{31}, c_{32},c'_{32}, c_{33}^*]$
\\& &\\
\hline\hline
& &\\
$\beta \alpha \gamma$ & $\begin{pmatrix} (v + p) c_{11} & (v + p) c_{12}^* & 0\\ v(v + p) c_{21}^*& (v + p) c_{22} &0\\v (c_{31}+(v + p)  c_{31}')&v c_{32}& c_{33}^* \end{pmatrix} $&  $\begin{matrix}
\cO[c_{11},c_{12}^*,c_{21}^*,c_{22},c_{31},c_{31}',c_{32},c_{33}^* ]/I_{\tld{z}_j}
\\
\\
I_{\tld{z}_j}=(c_{11} c_{22} +p c_{12}^* c_{21}^*)
 \end{matrix}$
\\& &\\
\hline& &\\
$\alpha \beta \gamma$ 
&  $\begin{pmatrix} (v + p)^2 c_{11}^* &0&0\\v(c_{21} + (v+p) c_{21}') & c_{22}& c_{23}^*\\v(c_{21} c_{33} (c_{23}^*)^{-1}+ (v+p)c_{31}')&vc_{32}^*&c_{33} \end{pmatrix}$&
$\begin{matrix} \cO[c_{11}^*,c_{21} ,c'_{21}, c_{22}, c_{23}^*,c_{31}',c_{32}^*,c_{33}]/I_{\tld{z}_j} 
 \\ \\I_{\tld{z}_j}=(c_{22} c_{33} + p c_{32}^* c_{23}^*) \end{matrix} $
\\& &\\
\hline
\hline
& &\\
$\alpha \beta \alpha$ & $\begin{pmatrix} c_{11} & c_{11} c_{32} (c_{31}^*)^{-1}& c_{13} + (v + p) c_{13}^* \\ 0  & (v + p)c_{22}^*& (v + p)c_{23}\\v c_{31}^* & vc_{32}&c_{33} + (v+p)c_{33}' \end{pmatrix}$& $\begin{matrix} \cO[ c_{11},c_{12},c_{13},c_{13}^*,c_{22}^*,c_{23}, c_{31}^*,c_{32},c_{33},c_{33}']/I_{\tld{z}_j} \\ \\I_{\tld{z}_j}=
\left(\begin{matrix}
c_{11}c_{32}-c_{12}c_{31}^*,\,
c_{11} c_{33}+ p c_{13} c_{31}^*\
\\
c_{11} c_{33}' - c_{13} c_{31}^* + p c_{13}^*c_{31}^*
\end{matrix}\right)
\end{matrix} $
\\
& &\\
\hline
\hline
& &\\
$\alpha \beta$ &  $\begin{pmatrix}
c_{31}c_{12}(c_{32}^*)^{-1}&c_{12}&c_{13}+(v+p) c_{13}^*\\
vc_{21}^*&c_{22}&c_{23}+(v+p) c_{23}'\\
vc_{31}&vc_{32}^*&\left(c_{31}c_{23}(c_{21}^*)^{-1}+(v+p)c_{33}'\right)
\end{pmatrix}$&
$\begin{matrix}
\cO[
c_{12},c_{13},c_{13}^*,c_{21}^*,c_{22},c_{23},c_{23}',c_{31},c_{32}^*,c_{33}']/I_{\tld{z}_j}
\\ \\
I_{\tld{z}_j}=\left(\begin{matrix}c_{22}c_{31}+pc_{21}^*c_{32}^*,\ 
c_{12}c_{23}-c_{22} c_{13}\\
c_{21}^*c_{32}^*c_{13}-pc_{21}^*c_{32}^*c_{13}^*-c'_{33}c_{21}^*c_{12}=0\end{matrix}\right)
\end{matrix}$
\\
& &\\
\hline
& &\\
$\beta \alpha$ &  $\begin{pmatrix}
c_{11}&\left((c_{31}^*)^{-1}c_{11}c_{32}+ (v+p)c_{12}^*\right)&c_{13}\\
0& (v+p) c_{22}'&  (v+p)c_{23}^*\\
c_{31}^*v&c_{32}v&c_{33}+ (v+p)c_{33}'
\end{pmatrix}$&
$\begin{matrix}
\cO[
c_{11},c_{12}^*,c_{13},c_{22}',c_{23}^*,
c_{31}^*,c_{32},c_{33},c_{33}']/I_{\tld{z}_j}
\\ \\
I_{\tld{z}_j}=\left(\begin{matrix}c_{11}c_{33}+pc_{31}^*c_{13}\\
c_{22}'(c_{11} c_{33}' - c_{13} c^*_{31})-pc_{23}^*c_{12}^*c_{31}^*\end{matrix}\right)
\end{matrix}$
\\
& &\\
\hline
\hline
& &\\
$\alpha$ &  $\begin{pmatrix}
c_{11}&c_{12}+ (v+p)c_{12}^*&c_{13}\\
c_{21}^*v&c_{22}+ (v+p)c_{22}' &c_{23}\\
c_{31}v&c_{32}v&\left(c_{33}+(v+p) c_{33}^*\right)
\end{pmatrix}$&$\begin{matrix}
\cO[
c_{11},c_{12},c_{12}^*,c_{13},
c_{21}^*,c_{22},c_{22}',c_{23},
c_{31},c_{32},c_{33},c_{33}^*]/I_{\tld{z}_j}
\\ \\
I_{\tld{z}_j}=\left(\begin{matrix}
c_{11}c_{22}+pc_{12}c^*_{21},\  
c_{11}c_{23}+pc_{13}c^*_{21},\ 
c_{12}c_{23}-c_{13}c_{22},\ 
c_{11}c_{32}-c_{31}c_{12},\ 
c_{11}c_{33}+pc_{31}c_{13}, 
\\
c_{12}c_{33}+pc_{32}c_{13},\ 
pc_{21}^*c_{32}+c_{22}c_{31},\ 
c^*_{21}c_{33}-c_{23}c_{31},\ 
c_{22}c_{33}+pc_{32}c_{23},
\\
c_{11}c_{22}'c_{33}^*+c_{13}c_{21}^*c_{32}-c_{13}c_{22}'c_{31}-c_{12}c_{21}^*c_{33}^*+pc_{21}^*c_{12}^*c_{33}^*\end{matrix}\right)
\end{matrix}$
\\& &\\
\hline
\hline
& &\\
id &  $\begin{pmatrix}
c_{11}+c_{11}^*(v+p)&c_{12}&c_{13}\\
vc_{21}&c_{22}+ c_{22}^*(v+p) &c_{23}\\
vc_{31}&vc_{32}&c_{33}+c_{33}^*(v+p)
\end{pmatrix}$&
$\begin{matrix}
\cO[c_{11},c_{11}^*,c_{12},c_{13},c_{21},c_{22}, c_{22}^*,c_{23},c_{31},c_{32},c_{33},c_{33}^*]/I_{\tld{z}_j}
\\ \\
I_{\tld{z}_j}=\left(\begin{matrix}%
c_{11}c_{22}+pc_{12}c_{21},\  
c_{11}c_{23}+pc_{13}c_{21},\ 
c_{12}c_{23}-c_{13}c_{22},\ 
c_{11}c_{32}-c_{31}c_{12},\ 
c_{11}c_{33}+pc_{31}c_{13}, 
\\
c_{12}c_{33}+pc_{32}c_{13},\ 
pc_{21}c_{32}+c_{22}c_{31},\ 
c_{21}c_{33}-c_{23}c_{31},\ 
c_{22}c_{33}+pc_{32}c_{23},
\\
c_{11}c^*_{22}c^*_{33}+c_{22}c^*_{33}c^*_{11}+c_{33}c^*_{11}c^*_{22}-c_{11}^*c_{23}c_{32}-c_{22}^*c_{13}c_{31}-c_{33}^*c_{12}c_{21}+c_{21}c_{13}c_{32}\end{matrix}\right)
\end{matrix}$
\\
& &\\
\hline
\hline 
\end{tabular}}
\caption*{
We list the $\cO$-algebras $\cO(\tld{U}(\tld{z}_j, \leq\!\eta_j))$ appearing in Proposition \ref{prop:loc:mod:diag:1}. 
Note that $\cO(\tld{U}(\tld{z}_j, \leq\!\eta_j))\cong \cO(\tld{U}(\delta\tld{z}_j\delta^{-1}, \leq\!\eta_j))$ using the following change of coordinates in terms of universal matrices $A^{(j)}$: for $?\in\{\emptyset,\ast,\prime\}$ we have $c^{?}_{ik}\mapsto c^{?}_{(i+1)(k+1)}$, where, for $1\leq i,k\leq k$, the integers $(i+1), (k+1)\in\{1,2,3\}$ are taken modulo $3$.
}
\end{table}

\begin{table}[H]
\caption{\textbf{The relevant $\cO$-algebras for Proposition \ref{prop:loc:mod:diag:2}}}
\label{Table:integraleqs}
\centering
\adjustbox{max width=\textwidth}{
\begin{tabular}{| c | c |}
\hline
& \\
$\widetilde{z}_jt_{-\un{1}}$ &  $\cO(\tld{U}(\tld{z}_j,\eta,\nabla_{(s,\mu)}))$ \\ 
&  \\
\hline
& \\
$\alpha \beta \alpha \gamma$ &    $\cO[c_{11}^*,c_{21},c'_{31},c_{22}^*,c_{32},c_{33}^*]$ \\
&\\
\hline
&  \\
$\beta \gamma \alpha \gamma$ &  $\cO[c_{11}^*,c_{12},c_{22}^*, c_{31}, c'_{32}, c_{33}^*]$\\
&  \\
\hline
\hline
&  \\
$\beta \alpha \gamma$ &  

$\begin{matrix}
\cO[c_{11},c_{12}^*,c_{21}^*,c_{22},c_{31}',c_{32},c_{33}^* ]/I_{\tld{z}_j,\nabla_{(s,\mu)}}
\\
\\
I_{\tld{z}_j,\nabla_{(s,\mu)}}=(c_{11} c_{22} +p c_{12}^* c_{21}^*)
 \end{matrix}$
 \\
&  \\
\hline
&  \\
$\alpha \beta \gamma$ &   $\begin{matrix} \cO[c_{11}^*,c'_{21}, c_{22}, c_{23}^*,c_{31}',c_{32}^*,c_{33}]/I_{\tld{z}_j,\nabla_{s_j(\mu_j+\eta_j)}} 
 \\ \\I_{\tld{z}_j,\nabla_{(s,\mu)}}=(c_{22} c_{33} + p c_{32}^* c_{23}^*) \end{matrix} $
 \\
&  \\
\hline \hline
&  \\
$\alpha \beta \alpha$ &$\begin{matrix}\cO[c_{32}, c_{23}, c_{33}', c_{31}^*, c_{22}^*, c_{13}^*] /I_{\tld{z}_j,\nabla_{(s,\mu)}} 
 \\ \\I_{\tld{z}_j,\nabla_{s_j(\mu_j+\eta_j)}}=\Big(c_{11}\Big((\mathbf{a}-\mathbf{b})c_{23}c_{32}-(\mathbf{a}-\mathbf{c})c_{22}^*c_{33}'\Big)+p(e'-\mathbf{a}+\mathbf{c})c_{31}^*c_{22}^*c_{13}^*\Big)
\end{matrix} 
$
\\
& \\
\hline \hline
& \\
$\alpha \beta$&$\begin{matrix}\cO[\![c_{31}, c_{22}, c_{12}, c_{23}', c_{33}', c_{21}^*, c_{13}^*, c_{32}^*]\!] /I_{\tld{z}_j,\nabla_{s_j(\mu_j+\eta_j)}} 
 \\ \\I_{\tld{z}_j,\nabla_{(s,\mu)}}=\left(\begin{array}{l}c_{12}\Big((\mathbf{a}-\mathbf{b})c_{31}c'_{23}+(\mathbf{b}-\mathbf{c})c^*_{21}c'_{33}\Big)-p(e'-\mathbf{a}+\mathbf{c})c^*_{21}c^*_{32}c^*_{13}\\
 c_{22}c_{31}+pc_{21}^*c_{32}^*
 \end{array}\right)
\end{matrix} $\\
&\\
\hline
\hline
& \\

$\beta \alpha$ &$\begin{matrix}\cO[\![c_{11},c_{22}', c_{13}, c_{32},  c_{33}', c_{31}^*, c_{12}^*, c_{23}^*]\!]/I_{\tld{z}_j,\nabla_{s_j(\mu_j+\eta_j)}} 
 \\ \\I_{\tld{z}_j,\nabla_{(s,\mu)}}=\left(\begin{array}{l}c_{11}\Big((\mathbf{a}-\mathbf{b})c_{32}c^*_{23}-(\mathbf{a}-\mathbf{c})c'_{22}c'_{33}\Big)-p(e'-\mathbf{a}+\mathbf{c})c^*_{12}c^*_{23}c^*_{31}
 \\c_{22}'\Big(c_{11} c_{33}' - c_{13} c^*_{31}\Big) - pc_{23}^*c_{12}^*c_{31}^*
 \end{array}\right)
\end{matrix} $\\
&\\
\hline
\end{tabular}}
\caption*{
We list the $\cO$-algebras $\cO(\tld{U}(\tld{z}_j,\eta_j,\nabla_{(s,\mu)}))$ appearing in Proposition \ref{prop:loc:mod:diag:2}. 
The triple $(\mathbf{a},\mathbf{b},\mathbf{c})$ is $(s'_{\mathrm{or}, j'})^{-1}(\mathbf{a}^{\prime \, (j')})$.
Note that $\cO(\tld{U}(\tld{z}_j,\eta,\nabla_{(s,\mu)}))\cong \cO(\tld{U}(\delta\tld{z}_j\delta^{-1},\eta_j,\nabla_{(s,\mu)}))$ using the following change of coordinates: for $?\in\{\emptyset,\ast,\prime\}$ we have $c^{?}_{ik}\mapsto c^{?}_{(i+1)(k+1)}$ (where, for $1\leq i,k\leq k$, the integers $(i+1), (k+1)\in\{1,2,3\}$ are taken modulo $3$) and moreover $\mathbf{a}\mapsto\mathbf{b}$, $\mathbf{b}\mapsto\mathbf{c}$, 
$\mathbf{c}\mapsto \mathbf{a}-e'$.
}
\end{table}

\begin{table}[H]
\caption{}
\label{Table:intsct}
\centering
\adjustbox{max width=\textwidth}{
\begin{tabular}{| c | c | c | c |}
\hline
&&&\\
$\tld{z}_{j}t_{-\un{1}}$&$\tld{U}(\tld{z}_j, \eta_j,\nabla_{s_j^{-1}(\mu_j+\eta_j)})_{\F}$&$(\eps_j,a_j)\in \Sigma_0$&$\widetilde{\fP}_{(\eps_j,a_j),\tld{z}_j}$\\
&&&\\
\hline
&&&\\
$\alpha \beta \alpha \gamma$ &   $\begin{pmatrix} v^2 c_{11}^{*} & 0 & 0 \\ v^2 c_{21} & v c_{22}^{*} & 0 \\ v^2 d_{31} & v c_{32} & c_{33}^{*}\end{pmatrix}$&$(\eps_1+\eps_2,0)$&$(0)$ 
\\
&&&\\
\hline
\hline
&&&\\
$\beta \gamma \alpha \gamma$ & $\begin{matrix}  \begin{pmatrix} v  c_{11}^{*} & v c_{12} & 0 \\ 0 & v^2 c_{22}^{*} & 0 \\ v c_{31} & v c_{32}+v^2 d_{32}& c_{33}^{*}\end{pmatrix}\\
\\
(-1-b+c)c_{32} c_{11}^{*}-(-1-a+c)c_{12}c_{31}=0
\end{matrix}$&$(\eps_2,1)$&$(0)$ 
\\
&&&\\
\hline
\hline
&&&\\
\multirow{2}{*}{$\beta\alpha\gamma$}&\multirow{2}{*}{$\begin{matrix}\begin{pmatrix}vc_{11}&v c_{12}^*&0\\v^2c_{21}^*&v c_{22}&0\\v(c_{31}+v d_{31})&vc_{32}&c_{33}^*\end{pmatrix}\\
\\
I_{\tld{z}_j,\F}=0;\\
\\
(-1-a+c)c_{12}^*c_{31}-(-1-b+c)c_{32}c_{11}=0
\end{matrix}$}&$(\eps_1+\eps_2,0)$&$(c_{11})$\\
&&&\\
&&&\\
\cline{3-4}&&&\\
&&$(\eps_2,1)$&$(c_{22})$ 
\\
&&&\\
&&&\\
&&&\\
\hline
\hline
&&&\\
\multirow{2}{*}{$\alpha\beta\gamma$}&\multirow{2}{*}{$\begin{matrix}\begin{pmatrix}v^2c_{11}&0&0\\v(c_{21}+vd_{21})&c_{22}&c_{23}^*\\v(c_{21}c_{33}(c_{23}^*)^{-1}+v d_{31})&vc^*_{32}&c_{33}\end{pmatrix}\\
\\
I_{\tld{z}_j,\F}=0;\\
\\
(-1-a+c)c_{21}c^*_{32}+(b-c)d_{31}c_{22}=0
\end{matrix}$}&$(\eps_1+\eps_2,0)$&$(c_{22})$ 
\\
&&&\\
&&&\\
\cline{3-4}&&&\\
&&
$(\eps_1,1)$&$(c_{33})$ \\
&&&\\
&&&\\
&&&\\
\hline
\hline
&&&\\
\multirow{2}{*}{$\alpha\beta\alpha$}&\multirow{2}{*}{$
\begin{matrix}
\begin{pmatrix}
c_{11}&c_{11}c_{32}(c^{*}_{31})^{-1}&d_{33}c_{11}(c^{*}_{31})^{-1}+vc_{13}^{*}\\
0&v c^{*}_{22}& vc_{23} \\
v c^{*}_{31}&v c_{32}&v d_{33}
\end{pmatrix}\\
\\
I_{\tld{z}_j,\F}=0;
\\
\\
c_{11}((a-b)c_{23}c_{32}-(a-c)c^{*}_{33}d_{33})=0
\end{matrix}
$}&$(0,0)$&$(c_{11})$
\\
&&&\\
&&&\\
\cline{3-4}&&&\\
&&
$(0,1)$&$((a-b)c_{23}c_{32}-(a-c)c^{*}_{33}d_{33})$ \\
&&&\\
&&&\\
&&&\\
\hline
\end{tabular}}
\caption*{The table records data relevant to Theorem \ref{thm:local_model_cmpt}. %
The first column records the components of $\tld{z}$. 
The second column records the coordinates of $(\tld{U}(\tld{z}, \eta,\nabla_{\tau,\infty})_{\F})$ in terms of the universal matrix $A^{(j)}$ and the relations between its coefficients.
Recall that in the statement of Theorem \ref{thm:local_model_cmpt} the Serre weight $\sigma$ is parametrized by $(\mu+\eta-\lambda+s(\eps),\un{a})\in \un{\La}_R^{\lambda}\times\cA$.
The ideal corresponding to the closed immersion  $\tld{\fP}_{\sigma,\tld{z}}\into \tld{U}(\tld{z},\eta,\nabla_{\tau,\infty})_{\F}$ is of the form $\sum_{j=0}^{f-1}\tld{\fP}_{(\eps_j,a_j),\tld{z}_j}$, where each $\tld{\fP}_{(\eps_j,a_j),\tld{z}_j}$ is a minimal prime ideal of $\cO(\tld{U}(\tld{z}_j, \eta_j,\nabla_{s_j^{-1}(\mu_j+\eta_j)})_{\F})$.
The elements $(\eps_j,a_j)\in \Sigma_0$ are specified in the third column and the ideal $\tld{\fP}_{(\eps_j,a_j),\tld{z}_j}$ specified in the fourth column records.
The structure constants that feature in the presentation are given by $(a,b,c)\in\Fp^3$ with $(a,b,c)\equiv s_j^{-1}(\mu_j+\eta_j)\mod p$.
}
\end{table}

\begin{table}[H]
\label{Table:intsct:1}
\centering
\adjustbox{max width=\textwidth}{
\begin{tabular}{| c | c | c | c |}
\hline
&&&\\
$\tld{z}_{j}t_{-\un{1}}$&$\tld{U}(\tld{z}_j, \eta_j,\nabla_{s_j^{-1}(\mu_j+\eta_j)})_{\F}$&$(\eps_j,a_j)\in \Sigma_0$&$\widetilde{\fP}_{(\eps_j, a_i),\tld{z}_j}$\\
&&&\\
\hline
&&&\\
\multirow{4}{*}{$\alpha\beta$}&\multirow{4}{*}{$\begin{matrix}\begin{pmatrix}c_{13}c_{12}(c_{32}^*)^{-1}&c_{12}&c_{13}+v c_{13}^*\\v c_{21}^*&c_{22}&c_{23}+vd_{23}\\v c_{31}&vc^*_{32}&c_{31}c_{23}(c_{21}^*)^{-1}+v d_{33}\end{pmatrix}\\
\\
I_{\tld{z}_j,\F}=0;\\
\\
c_{12}((a-b)c_{31}d_{23}-(b-c)d_{33}c_{21}^*)=0;\\
(-1-a+c)c_{23}c_{32}^*=(-1-a+b)c_{22}d_{33}=0
\end{matrix}$}&$(\eps_1,1)$&$(c_{12},c_{31})$\\&&&\\
\cline{3-4}&&&
\\
&&$(\eps_1-\eps_2,0)$&$(c_{31},d_{33})$\\
&&&\\
\cline{3-4}&&&
\\
&&$(0,0)$&$(c_{12},c_{22})$\\&&&\\
\cline{3-4}&&&
\\
&&$(0,1)$&$(c_{22},(b-c)d_{33}c_{21}^*+(a-b)c_{31}d_{23})$\\
&&&\\
\hline
\hline
&&&\\
\multirow{4}{*}{$\beta\alpha$}&\multirow{4}{*}{$\begin{matrix}\begin{pmatrix}c_{11}&(c_{31}^*)^{-1}c_{11}c_{32}+v c_{12}^*&c_{13}\\0&vd_{22}&vc_{23}^*\\v c^*_{31}&vc_{32}&c_{33}+v d_{33}\end{pmatrix}\\
\\
I_{\tld{z}_j,\F}=0;\\
\\
c_{11}((a-b)c_{32}c^*_{23}-(a-c)d_{22}d_{33})=0;\\
(1+a-c)c_{33}c_{23}^*c_{12}^*=c_{13}((a-b)c_{32}c^*_{23}-(a-c)d_{22}d_{33})
\end{matrix}$}&$(\eps_2,1)$&$(d_{22},c_{11})$\\
&&&\\
\cline{3-4}&&&
\\
&&$(\eps_2-\eps_1,0)$&$(d_{22},c_{32})$\\
&&&\\
\cline{3-4}&&&
\\
&&$(0,0)$&$(c_{11},c_{13})$\\
&&&\\
\cline{3-4}&&&
\\
&&$(0,1)$&$((a-b)c_{32}c^*_{23}-(a-c)d_{22}d_{33},c_{13}c_{31}^*-c_{11}d_{33})=0$\\
&&&\\
\hline
\hline
&&&\\
\multirow{6}{*}{$\alpha$}&\multirow{6}{*}{$\begin{matrix}\begin{pmatrix}
c_{11}&c_{12}+vc_{12}^*& c_{13}\\
vc_{21}^*&c_{22}+vd_{22}&c_{23}\\
vc_{31}&vc_{32}&(c_{21}^*)^{-1}c_{31}c_{23}+vc_{33}^*=0
\end{pmatrix}\\
\\
I_{\tld{z}_j,\F}=0;\\
\\
(a-b) c_{12} c^*_{33}  - (a-c)c_{13}\widetilde{c}_{32}=0;
\\
(e-a + c)  c_{23}\widetilde{c}_{32}  - (e - a +b) c_{22} c^*_{33}  =0;
\\
\\
\hspace{-1cm}(c-1-a)c_{31}c_{23}c^*_{12}-(c-1-a)c_{31}c_{13}d_{22}+\\
\hspace{1cm}+(c-1-b)c_{32}c_{13}c^*_{21}+c_{12}c^*_{33}c^*_{21}-c_{11}d_{22}c^*_{33}=0
\end{matrix}$}&$(\eps_1,1)$&$(c_{11},c_{13},c_{31})$\\
&&&\\
\cline{3-4}&&&\\
&&$(\eps_2,0)$&$(c_{11},c_{31},c_{32}c_{21}^*-d_{22}c_{31})$\\
&&&\\
\cline{3-4}&&&\\&&$(\eps_2,1)$&$ (c_{11},c_{32}c_{21}^*-d_{22}c_{31}, (a-b)c_{13}d_{22}+(-1-a+c)c_{23}c^*_{12})$\\&&&
\\
\cline{3-4}&&&\\&&$(\eps_2-\eps_1,0)$&$(c_{23},d_{22},c_{32}c_{21}^*-d_{22}c_{31})$
\\
&&&\\
\cline{3-4}&&&\\&&$(0,0)$&$(c_{11},c_{13},c_{23})$\\&&&
\\
\cline{3-4}&&&\\&&$(0,1)$&$ (c_{11}c_{33}^*-c_{13}c_{31},c_{23}, (a-b)c_{31}d_{22}+(c-b)(c_{32}c_{21}^*-d_{22}c_{31}))$\\&&&
\\
\hline
\hline
&&&\\
\multirow{6}{*}{$\id$}&\multirow{6}{*}{$
\begin{matrix}
\begin{pmatrix}
c_{11}+vc_{11}^*&c_{12}& c_{13}\\
vc_{21}&c_{22}+vc_{22}^*&c_{23}\\
vc_{31}&vc_{32}&c_{33}+vc_{33}^*
\end{pmatrix}
\\
\\
I_{\tld{z}_j,\F}=0;\\
\\
(c-1-a)c^*_{22}c_{33}+(b-1-a)c_{22}c^*_{33}-(c-1-a)c_{23}c_{32}=0
\\
(a-b)c^*_{33}c_{11}+(c-1-b)c_{33}c^*_{11}-(a-b)c_{13}c_{31}=0
\\
(b-c)c^*_{11}c_{22}+(a-c)c_{11}c^*_{22}-(b-c)c_{12}c_{21}=0
\end{matrix}
$}&$(\eps_1,0)$&$(c_{ii}, \text{\tiny{$i=1,2,3$}}, \ c_{21},c_{31},c_{23})$\\&&&\\
\cline{3-4}&&&\\&&$(\eps_1,1)$&$
(c_{31}, c_{33}, c_{11}, (-1-a+c)c_{32}c_{13} - (-1-a+b)c_{12}c^*_{33}, c_{21}c_{13} - c_{23}c^*_{11})
$\\&&&\\
\cline{3-4}&&&\\&&$(\eps_2,0)$&$(c_{ii}, \text{\tiny{$i=1,2,3$}},\ c_{12},c_{31},c_{32})$\\&&&\\
\cline{3-4}&&&\\&&$(\eps_2,1)$&$(c_{12}, c_{22}, c_{11}, (a-b)c_{21}c_{13} -(-1-b+c)c_{23}c^*_{11}, c_{21}c_{32} - c_{31}c^*_{22})
$\\&&&\\
\cline{3-4}&&&\\&&$(0,0)$&$(c_{ii}, \text{\tiny{$i=1,2,3$}},\ c_{13},c_{23},c_{12})$\\&&&\\
\cline{3-4}&&&\\&&$(0,1)$&$
(c_{23}, c_{33}, c_{22}, (a-b)c_{21}c_{32} -(a-c)c_{31}c^*_{22}, c_{32}c_{13} - c_{12}c^*_{33})$\\&&&\\
\hline
\end{tabular}}
\caption*{The table records further data relevant to Theorem \ref{thm:local_model_cmpt}.
The meaning of the columns is the same as in Table \ref{Table:intsct}.}
\end{table}

\newpage
\bibliography{Biblio}

\newcommand{\etalchar}[1]{$^{#1}$}
\providecommand{\bysame}{\leavevmode\hbox to3em{\hrulefill}\thinspace}
\providecommand{\MR}{\relax\ifhmode\unskip\space\fi MR }
\providecommand{\MRhref}[2]{%
  \href{http://www.ams.org/mathscinet-getitem?mr=#1}{#2}
}
\providecommand{\href}[2]{#2}
\begin{thebibliography}{LLHLM20}

\bibitem[BGHT11]{BLGGHT2}
Tom {Barnet-Lamb}, David {Geraghty}, Michael {Harris}, and Richard {Taylor},
  \emph{{A family of Calabi-Yau varieties and potential automorphy. II.}},
  {Publ. Res. Inst. Math. Sci.} \textbf{47} (2011), no.~1, 29--98 (English).

\bibitem[CEG{\etalchar{+}}16]{CEGGPS}
Ana Caraiani, Matthew Emerton, Toby Gee, David Geraghty, Vytautas
  Pa{\v{s}}k{\=u}nas, and Sug~Woo Shin, \emph{Patching and the {$p$}-adic local
  {L}anglands correspondence}, Camb. J. Math. \textbf{4} (2016), no.~2,
  197--287. \MR{3529394}

\bibitem[CHT08]{CHT}
Laurent Clozel, Michael Harris, and Richard Taylor, \emph{Automorphy for some
  {$l$}-adic lifts of automorphic mod {$l$} {G}alois representations}, Publ.
  Math. Inst. Hautes \'Etudes Sci. (2008), no.~108, 1--181, With Appendix A,
  summarizing unpublished work of Russ Mann, and Appendix B by Marie-France
  Vign{\'e}ras. \MR{2470687 (2010j:11082)}

\bibitem[CL18]{CL}
Ana Caraiani and Brandon Levin, \emph{Kisin modules with descent data and
  parahoric local models}, Ann. Sci. \'{E}c. Norm. Sup\'{e}r. (4) \textbf{51}
  (2018), no.~1, 181--213. \MR{3764041}

\bibitem[DL21]{DL}
Andrea {Dotto} and Daniel {Le}, \emph{{Diagrams in the mod \(p\) cohomology of
  Shimura curves}}, {Compos. Math.} \textbf{157} (2021), no.~8, 1653--1723
  (English).

\bibitem[EG]{EGstack}
Matthew Emerton and Toby Gee, \emph{Moduli stacks of \'etale
  $(\varphi,{\Gamma})$-modules and the existence of crystalline lifts}, Annals
  of Mathematical Studies, to appear (2021).

\bibitem[EG14]{EG}
\bysame, \emph{A geometric perspective on the {B}reuil-{M}\'ezard conjecture},
  J. Inst. Math. Jussieu \textbf{13} (2014), no.~1, 183--223. \MR{3134019}

\bibitem[EGH13]{EGH}
Matthew Emerton, Toby Gee, and Florian Herzig, \emph{Weight cycling and
  {S}erre-type conjectures for unitary groups}, Duke Math. J. \textbf{162}
  (2013), no.~9, 1649--1722. \MR{3079258}

\bibitem[EGS15]{EGS}
Matthew Emerton, Toby Gee, and David Savitt, \emph{Lattices in the cohomology
  of {S}himura curves}, Invent. Math. \textbf{200} (2015), no.~1, 1--96.
  \MR{3323575}

\bibitem[Enn19]{Enns}
John Enns, \emph{On weight elimination for {${\rm GL}_n(\Bbb Q_{p^f})$}}, Math.
  Res. Lett. \textbf{26} (2019), no.~1, 53--66. \MR{3963975}

\bibitem[GG12]{gee-geraghty}
Toby Gee and David Geraghty, \emph{Companion forms for unitary and symplectic
  groups}, Duke Math. J. \textbf{161} (2012), no.~2, 247--303. \MR{2876931}

\bibitem[GHLS17]{GHLS}
Toby {Gee}, Florian {Herzig}, Tong {Liu}, and David {Savitt},
  \emph{{Potentially crystalline lifts of certain prescribed types}}, {Doc.
  Math.} \textbf{22} (2017), 397--422 (English).

\bibitem[GHS18]{GHS}
Toby Gee, Florian Herzig, and David Savitt, \emph{General {S}erre weight
  conjectures}, J. Eur. Math. Soc. (JEMS) \textbf{20} (2018), no.~12,
  2859--2949. \MR{3871496}

\bibitem[Her]{florian-thesis}
Florian Herzig, \emph{The weight in a serre type conjecture for tame
  $n$-dimensional galois representations}, PhD thesis.

\bibitem[Her09]{herzig-duke}
\bysame, \emph{The weight in a {S}erre-type conjecture for tame
  {$n$}-dimensional {G}alois representations}, Duke Math. J. \textbf{149}
  (2009), no.~1, 37--116. \MR{2541127 (2010f:11083)}

\bibitem[Her11]{florian-inv}
\bysame, \emph{The classification of irreducible admissible mod {$p$}
  representations of a {$p$}-adic {${\rm GL}\sb n$}}, Invent. Math.
  \textbf{186} (2011), no.~2, 373--434. \MR{2845621}

\bibitem[HLM17]{HLM}
Florian Herzig, Daniel Le, and Stefano Morra, \emph{On mod $p$ local-global
  compatibility for $\mathrm{GL}_3$ in the ordinary case}, Compositio Math.
  \textbf{153} (2017), no.~11, 2215--2286.

\bibitem[LLHL19]{LLL}
Daniel Le, Bao~Viet Le~Hung, and Brandon Levin, \emph{Weight elimination in
  {S}erre-type conjectures}, Duke Math. J. \textbf{168} (2019), no.~13,
  2433--2506. \MR{4007598}

\bibitem[LLHLMa]{MLM}
Daniel Le, Bao~Viet Le~Hung, Brandon Levin, and Stefano Morra, \emph{Local
  models for galois deformation rings and applications},
  \url{https://arxiv.org/pdf/2007.05398.pdf}, preprint (2020).

\bibitem[LLHLMb]{OBW}
\bysame, \emph{Obvious {S}erre weights and a tameness criterion for mod $p$
  {G}alois representations}, in preparation (2021).

\bibitem[LLHLM18]{LLLM}
\bysame, \emph{Potentially crystalline deformation rings and {S}erre weight
  conjectures: shapes and shadows}, Invent. Math. \textbf{212} (2018), no.~1,
  1--107. \MR{3773788}

\bibitem[LLHLM20]{LLLM2}
\bysame, \emph{Serre weights and {B}reuil's lattice conjectures in dimension
  three}, Forum Math. Pi \textbf{8} (2020), e5, 135. \MR{4079756}

\bibitem[LLHM{\etalchar{+}}]{LGC}
Daniel Le, Bao~Viet Le~Hung, Stefano Morra, Chol Park, and Zicheng Qian,
  \emph{Moduli of fontaine--laffaille representations and a mod-$p$
  compatibility result}, \url{https://arxiv.org/pdf/2109.02720.pdf}, preprint
  (2021).

\bibitem[LMP18]{LMP}
Daniel {Le}, Stefano {Morra}, and Chol {Park}, \emph{{On \(\mod p\)
  local-global compatibility for \(\mathrm{GL}_3(\mathbb{Q}_p)\) in the
  non-ordinary case}}, {Proc. Lond. Math. Soc. (3)} \textbf{117} (2018), no.~4,
  790--848 (English).

\bibitem[MP17]{MP}
Stefano {Morra} and Chol {Park}, \emph{{Serre weights for three-dimensional
  ordinary Galois representations}}, {J. Lond. Math. Soc., II. Ser.}
  \textbf{96} (2017), no.~2, 394--424 (English).

\end{thebibliography}
\bibliographystyle{amsalpha}

\end{document}